%% file: draft_final.tex
\documentclass[12pt]{article} %

\input{preamble}\usepackage[title]{appendix}%
\usepackage{url}
\usepackage{xr}%
\title{Optimal inference in a class of regression models\thanks{We thank Don
    Andrews, Isaiah Andrews, Matias Cattaneo, Gary Chamberlain, Denis
    Chetverikov, Yuichi Kitamura, Soonwoo Kwon, Ulrich Müller and Azeem Shaikh
    for useful discussions. We thank the editor, three anonymous referees, and
    numerous seminar and conference participants for helpful comments and
    suggestions. All remaining errors are our own. The research of the first
    author was supported by National Science Foundation Grant SES-1628939. The
    research of the second author was supported by National Science Foundation
    Grant SES-1628878.}}
\author{Timothy B. Armstrong\thanks{email: \texttt{timothy.armstrong@yale.edu}}\\
  Yale University \and
  Michal Kolesár\thanks{email: \texttt{mkolesar@princeton.edu}}\\
  Princeton University} \date{\today}

\begin{document}
\maketitle

\begin{abstract}
  We consider the problem of constructing confidence intervals (CIs) for a
  linear functional of a regression function, such as its value at a point, the
  regression discontinuity parameter, or a regression coefficient in a linear or
  partly linear regression. Our main assumption is that the regression function
  is known to lie in a convex function class, which covers most smoothness
  and/or shape assumptions used in econometrics. We derive finite-sample optimal
  CIs and sharp efficiency bounds under normal errors with known variance. We
  show that these results translate to uniform (over the function class)
  asymptotic results when the error distribution is not known. When the function
  class is centrosymmetric, these efficiency bounds imply that minimax CIs are
  close to efficient at smooth regression functions. This implies, in
  particular, that it is impossible to form CIs that are substantively tighter using
  data-dependent tuning parameters, and maintain coverage over the whole
  function class. We specialize our results to inference on the regression
  discontinuity parameter, and illustrate them in simulations and an empirical
  application.
\end{abstract}

\clearpage

\section{Introduction}

In this paper, we study the problem of constructing confidence intervals (CIs)
for a linear functional $Lf$ of a regression function $f$ in a broad class of
regression models with fixed regressors, in which $f$ is known to belong to some
convex function class $\mathcal{F}$. The linear functional may correspond to the
regression discontinuity parameter, an average treatment effect under
unconfoundedness, or a regression coefficient in a linear or partly linear
regression. The class $\mathcal{F}$ may contain smoothness restrictions
(e.g.~bounds on derivatives, or assuming $f$ is linear as in a linear
regression), and/or shape restrictions (e.g.~monotonicity, or sign restrictions
on regression coefficients in a linear regression). Often in applications, the
function class will be indexed by a smoothness parameter $C$, such as when
$\mathcal{F}=\mathcal{F}_{\text{Lip}}(C)$, the class of Lipschitz continuous
functions with Lipschitz constant $C$.

Our main contribution is to derive finite-sample optimal CIs and sharp
efficiency bounds that have implications for data-driven model and bandwidth
selection in both parametric and nonparametric settings. To derive these
results, we assume that the regression errors are normal, with known variance.
When the error distribution is unknown, we obtain analogous uniform asymptotic
results under high-level regularity conditions. We derive sufficient low-level
conditions in an application to regression discontinuity.

First, we characterize one-sided CIs that minimize the maximum $\beta$ quantile
of excess length over a convex class $\mathcal{G}$ for a given quantile $\beta$.
The lower limit $\hat{c}$ of the optimal CI $\hor{\hat{c}, \infty}$ has a simple
form: take an estimator $\hat{L}$ that trades off bias and variance in a certain
optimal sense and is linear in the outcome vector, and subtract (1) the standard
deviation of $\hat{L}$ times the usual critical value based on a normal
distribution and (2) a bias correction to ensure coverage. This bias correction,
in contrast to bias corrections often used in practice, is based on the maximum
bias of $\hat{L}$ over $\mathcal{F}$, and is therefore non-random.

When $\mathcal{G}=\mathcal{F}$, this procedure yields minimax one-sided CIs.
Setting $\mathcal{G}\subset\mathcal{F}$ to a class of smoother functions is
equivalent to ``directing power'' at these smoother functions while maintaining
coverage over $\mathcal{F}$, and gives a sharp bound on the scope for adaptation
for one-sided CIs. We show that when $\mathcal{F}$ is centrosymmetric
(i.e.~$f\in\mathcal{F}$ implies $-f\in\mathcal{F}$), the scope for adaptation is
severely limited: when $\mathcal{G}$ is a class of functions that are, in a
certain formal sense, ``sufficiently smooth'' relative to $\mathcal{F}$, CIs
that are minimax for $\beta$ quantile of excess length also optimize excess
length over $\mathcal{G}$, but at a different quantile. Furthermore, they are
also highly efficient at such smooth functions for the same quantile. For
instance, a CI for the conditional mean at a point that is minimax over the
Lipschitz class $\mathcal{F}_{\text{Lip}}(C)$ is asymptotically 95.2\% efficient
at constant functions relative to a CI that directs all power at constant
functions. For function classes that bound a derivative of higher order, the
efficiency is even higher.

Second, we derive a confidence set that minimizes its expected length at a
single function $g$. We compare its performance to the optimal fixed-length CI
of \citet{donoho94} (i.e.\ CI of the form $\hat{L}\pm \chi$, where $\hat{L}$ is
an affine estimator, and $\chi$, which doesn't depend on the outcome vector and
is therefore non-random, is chosen to ensure coverage). Similarly, to the
one-sided case, we find that, when $\mathcal{F}$ is centrosymmetric, the optimal
fixed-length CIs are highly efficient at functions that are smooth relative to
$\mathcal{F}$. For instance, the optimal fixed-length CI for a conditional mean
at a point when $f\in\mathcal{F}_{\text{Lip}}(C)$ is asymptotically 95.6\%
efficient at any constant function $g$ relative to a confidence set that
optimizes its expected length at $g$.

An important practical implication of these results is that explicit a priori
specification of the smoothness constant $C$ cannot be avoided: procedures that
try to determine the smoothness of $f$ from the data (and thus implicitly
estimate $C$ from the data), including data-driven bandwidth or variable
selectors, must either fail to substantively improve upon the minimax CIs or
fixed-length CIs (that effectively assume the worst case smoothness), or else
fail to maintain coverage over the whole parameter space. We illustrate this
point through a Monte Carlo study in a regression discontinuity (RD) setting, in
which we show that popular data-driven bandwidth selectors lead to substantial
undercoverage, even when combined with bias correction or undersmoothing (see
\Cref{sec:monte-carlo-evidence}). To avoid having to specify $C$, one
has to strengthen the assumptions on $f$. For instance, one can impose shape
restrictions that break the centrosymmetry, as in \citet{cai_adaptive_2013} or
\citet{armstrong15as}, or self-similarity assumptions that break the convexity,
as in \citet{GiNi10} or \citet{cck14as}. Alternatively, one can weaken the
coverage requirement in the definition of a CI, by, say, only requiring average
coverage as in \citet{CaLoMa14} or \citet{hall_simple_2013}.

We apply these results to the problem of inference in RD\@. We show, in the
context of an empirical application from \citet{lee08}, that the fixed-length
and minimax CIs are informative and simple to construct, and we give a detailed
guide to implementing them in practice. We also consider CIs based on local
linear estimators, which have been popular in RD due to their high minimax
asymptotic MSE efficiency, shown in \citet{cfm97}. Using the same function
classes as in \citet{cfm97}, we show that in the \citeauthor{lee08} application,
when a triangular kernel is used, such CIs are highly efficient relative to the
optimal CIs discussed above.

Our finite-sample approach allows us to use the same framework and methods to
cover problems that are often seen as outside of the scope of nonparametric
methods. For instance, the same CIs can be used in RD whether the running
variable is discrete or continuous; one does not need a different modeling
approach, such as that of \citet{lee_regression_2008}. Similarly, we do not need
to distinguish between ``parametric'' or ``nonparametric'' constraints on $f$;
our results apply to inference in a linear regression model that efficiently use
a priori bounds and sign restrictions on the regression coefficients. Here our
efficiency bounds imply that the scope for efficiency improvements from CIs
formed after model selection
\citep{andrews_hybrid_2009,mccloskey_bonferroni-based_2012} is severely limited
unless asymmetric or non-convex restrictions are imposed, and they also limit
the scope for improvement under certain non-convex restrictions such as the
sparsity assumptions used in \citet{belloni_inference_2014}. We discuss these
issues in an earlier version of this paper \citep{ArKo15optimal}.

Our results and setup build on a large statistics literature on optimal
estimation and inference in the nonparametric regression model. This literature
has mostly been concerned with estimation (e.g., \citet{stone80},
\citet{IbKh85}, \citet{fan93}, \citet{donoho94}, \citet{cfm97}); the literature
on inference has mostly been focused on bounding rates of convergence. The
results most closely related to ours are those in \citet{low97}, \citet{CaLo04}
and \citet{cai_adaptive_2013}, who derive lower bounds on the expected length of
a two-sided CI over a convex class $\mathcal{G}$ subject to coverage over a
convex class $\mathcal{F}$. These results imply that, when $\mathcal{F}$ is
constrained only by bounds on a derivative, one cannot improve the rate at which
a two-sided CI shrinks by ``directing power'' at smooth functions. We contribute
to this literature by (1) deriving a sharp lower bound for one-sided CIs, and
for two-sided CIs when $\mathcal{G}$ is a singleton, (2) showing that the
negative results for ``directing power'' at smooth functions generalize to the
case when $\mathcal{F}$ is centrosymmetric, and deriving the sharp bound on the
scope for improvement, (3) deriving feasible CIs under unknown error
distribution and showing their asymptotic validity and efficiency, including in
non-regular settings; and (4) computing the bounds and CIs in an application to
RD\@.

The remainder of this paper is organized as follows. \Cref{sec:simple-example}
illustrates our results in an application to RD, and gives a detailed guide to
implementing our CIs. \Cref{sec:general_results} derives the main results under
a general setup. \Cref{sec:appl-regr-disc} considers an empirical application.
Proofs, long derivations, and additional results are collected in appendices.
\Cref{sec:ap:main-proofs} contains proofs for the main results in
\Cref{sec:general_results}. \Cref{covariate_section_append} discusses extensions
to incorporate covariates in the RD application. \Cref{sec:comp-with-other}
compares our CIs to other approaches, and includes a Monte Carlo study.
Additional details for constructing CIs studied in \Cref{sec:general_results}
are in \Cref{sec:ap:details}. \Cref{sec:addit-deta-rd} contains additional
details for the RD application. Asymptotic results are collected in Supplemental
\Cref{sec:sap:unknown_error,sec:sap:asym_efficiency_bounds,,sec:sap:rd_asym}.

\section{Application to regression discontinuity}\label{sec:simple-example}

In this section, we explain our results in the context of an application to
sharp regression discontinuity (RD). \Cref{optimal_cis_subsec} illustrates the
theoretical results, while \Cref{sec:pract-impl} gives step-by-step instructions
for implementing our confidence intervals (CIs) in practice.

We observe $\{y_{i}, x_{i}\}_{i=1}^{n}$, where the running variable $x_{i}$ is
deterministic, and
\begin{equation}\label{eq:fixed_design_eq}
  y_i=f(x_i)+u_{i}, \quad u_{i}\sim
  \mathcal{N}(0,\sigma^2(x_i))\;\text{independent across
    $i$,}
\end{equation}
with $\sigma^2(x)$ known.\footnote{This assumption is made to deliver
  finite-sample results---when the distribution of $u_{i}$ is unknown, with
  unknown conditional variance, we show in \Cref{sec:addit-deta-rd} that
  these results lead to analogous uniform-in-$f$ asymptotic results.} The
running variable determines participation in a binary treatment: units above a
given cutoff, which we normalize to $0$, are treated; units with $x_{i}<0$ are
controls. Let $f_{+}(x)=f(x)\1{x\geq 0}$ and $f_{-}(x)=f(x)\1{x< 0}$ denote the
part of the regression function $f$ above and below the cutoff, so that
$f=f_{+}+f_{-}$. The parameter of interest is the jump of the regression
function at zero, and we denote it by $Lf=f_{+}(0)-f_{-}(0)$, where
$f_-(0)=\lim_{x\uparrow 0} f_-(x)$. If the regression functions of potential
outcomes are continuous at zero, then $Lf$ measures the average treatment effect
for units with $x_{i}=0$.

We assume that $f$ lies in the class of functions $\mathcal{F}_{RDT, p}(C)$,
\begin{equation*}
  \mathcal{F}_{RDT, p}(C)= \big \{f_{+}+f_{-}\colon
    f_+\in\mathcal{F}_{T, p}(C;\mathbb{R}_+),\;
    f_{-}\in\mathcal{F}_{T,p}(C;\mathbb{R}_{-}) \big\},
\end{equation*}
where $\mathcal{F}_{T,p}(C;\mathcal{X})$ consists of functions $f$ such that the
approximation error from a $(p-1)$th-order Taylor expansion of $f(x)$ about $0$
is bounded by $C|x|^{p}$, uniformly over $\mathcal{X}$,
\begin{equation*}
  \mathcal{F}_{T,p}(C;\mathcal{X}) =\big\{f\colon \big|
      f(x)-\textstyle\sum_{j=0}^{p-1}
      f^{(j)}(0) x^{j}/j!\big|\le C|x|^p\text{ all
    }x\in\mathcal{X}\big\}.
\end{equation*}
This formalizes the notion that locally to $0$, $f$ is $p$-times differentiable
with the $p$th derivative at zero bounded by $p!C$. \citet{SaYl78} and
\citet{cfm97} considered minimax MSE estimation of $f(0)$ in this class when $0$
is a boundary point. Their results formally justify using local polynomial
regression to estimate the RD parameter. This class does not impose any
smoothness of $f$ away from cutoff, which may be too conservative in
applications. We consider inference under global smoothness in \citet{ArKo15},
where we show that for the $p=2$ case, the resulting CIs are about 10\% tighter
in large samples (see also \Cref{sec:monte-carlo-evidence} for a Monte
Carlo study under global smoothness).

\subsection{Optimal CIs}\label{optimal_cis_subsec}

For ease of exposition, we focus in this subsection on the case $p=1$, so that
the parameter space is given by $\mathcal{F}=\mathcal{F}_{RDT,1}(C)$, and assume
that the errors are homoskedastic, $\sigma^{2}(x_{i})=\sigma^{2}$. In
\Cref{sec:pract-impl}, we discuss implementation of the CIs in the general case
where $p\ge 1$.

Consider first the problem of constructing one-sided CIs for $Lf$. In
particular, consider the problem of constructing CIs $\hor{\hat{c},\infty}$ that
minimize the maximum $\beta$th quantile of excess length,
$\sup_{f\in\mathcal{F}}q_{f,\beta}(Lf-\hat{c})$, where $q_{f,\beta}$ denotes the
$\beta$th quantile of the excess length $Lf-\hat{c}$. We show in
\Cref{sec:onesided-cis} that such CIs can be obtained by inverting tests of the
null hypothesis $H_{0}\colon f_{+}(0)-f_{-}(0)\leq L_{0}$ that maximize their
minimum power under the alternative
$H_{1}\colon f_{+}(0)-f_{-}(0)\geq L_{0}+2b$, where the half-distance $b$ to the
alternative is calibrated so that the minimum power of these tests equals
$\beta$.

To construct such a test, note that if we set $\mu=(f(x_{1}),\dotsc,f(x_{n}))'$,
and $Y=(y_{1},\dotsc,y_{n})'$, we can view the testing problem as an $n$-variate
normal mean problem $Y\sim\mathcal{N}(\mu,\sigma^{2}I_{n})$, in which the vector
of means $\mu$ is constrained to take values in the convex sets
$M_{0}=\{(f(x_{1}),\dotsc,\allowbreak f(x_{n}))'\colon f\in\mathcal{F},
f_{+}(0)-f_{-}(0)\leq L_{0}\}$ under the null, and
$M_{1}=\{(g(x_{1}),\dotsc,g(x_{n}))'\colon \allowbreak g\in\mathcal{F},
g_{+}(0)-g_{-}(0)\geq L_{0}+2b\}$ under the alternative. The convexity of the
null and alternative sets implies that this testing problem has a simple
solution: by \Cref{convex_testing_lemma}, the minimax test is given by the
uniformly most powerful test of the simple null $\mu=\mu_{0}^{*}$ against the
simple alternative $\mu=\mu_{1}^{*}$, where $\mu^{*}_{0}$ and $\mu^{*}_{1}$
minimize the Euclidean distance between the null and alternative sets $M_{0}$
and $M_{1}$, and thus represent points in $M_{0}$ and $M_{1}$ that are hardest
to distinguish. By the Neyman-Pearson lemma, such test rejects for large values
of $(\mu_{1}^{*}-\mu_{0}^{*})'Y$. Because by \Cref{convex_testing_lemma}, this
test controls size over all of $M_{0}$, the points $\mu_{1}^{*}$ and
$\mu_{0}^{*}$ are called ``least favorable'' (see Theorem 8.1.1 in
\citealp{LeRo05}).

To compute $\mu^{*}_{0}=(f^{*}(x_{1}),\dotsc,f^{*}(x_{n}))'$ and
$\mu^{*}_{1}=(g^{*}(x_{1}),\dotsc,g^{*}(x_{n}))'$, we thus need to find
functions $f^{*}$ and $g^{*}$ that solve
\begin{equation}\label{eq:simple-example:minimization-problem}
  (f^{*},g^{*})=\argmin_{f,g\in \mathcal{F}}
  \sum_{i=1}^{n}(f(x_{i})-g(x_{i}))^{2}
  \qquad \text{subject to
    $Lf\leq L_{0} $, $Lg\geq L_{0}+2b$}.
\end{equation}
A simple calculation shows that the least favorable functions solving this
minimization problem are given by
\begin{equation}\label{eq:fg-star}
  \begin{aligned}
    g^{*}(x)&=\1{x\geq 0}(L_{0}+b)+C h_{+}\cdot k_{+}(x/h_{+})-Ch_{-} \cdot k_{-}(x/h_{-}),\\
    f^{*}(x)&=2\cdot \1{x\geq 0}(L_{0}+b)-g^{*}(x),
  \end{aligned}
\end{equation}
where $k(u)=\max\{0,1-\abs{u}\}$ is the triangular kernel,
$k_{+}(u)=k(u)\1{u\geq 0}$ and $k_{-}(u)=k(u)\1{u<0}$, and the ``bandwidths''
$h_{+},h_{-}$ are determined by a condition ensuring that $Lg^{*}\geq L_{0}+2b$,
\begin{equation}\label{eq:h-b}
  h_{+}+h_{-}
  =b/C,
\end{equation}
and a condition ensuring that positive and negative observations are equally
weighted,
\begin{equation}\label{eq:h-plus-h-minus}
  h_{+}
  \sum_{i=1}^{n}k_{+}(x_{i}/h_{+})=h_{-}\sum_{i=1}^{n}k_{-}(x_{i}/h_{-}).
\end{equation}
Intuitively, to make the null and alternative hardest to distinguish, the least
favorable functions $f^{*}$ and $g^{*}$ converge to each other ``as quickly as
possible'', subject to the constraints $Lf^{*}\leq L_{0}$ and
$Lg^{*}\geq b+L_{0}$, and the Lipschitz constraint---see
Figure~\ref{fig:lipschitz-lf}.

By working out the appropriate critical value and rearranging, we obtain that
the minimax test rejects whenever
\begin{equation}\label{eq:simple-example:minimax-test}
  \hat{L}_{h_{+},h_{-}} -L_{0}-
  \bias_{f^{*}}( \hat{L}_{h_{+},h_{-}}) \geq \sd(\hat{L}_{h_{+},h_{-}})z_{1-\alpha}.
\end{equation}
Here $\hat{L}_{h_{+},h_{-}}$ is a kernel estimator based on a triangular kernel
and bandwidths $h_{+}$ to the left and $h_{-}$ to the right of the cutoff
\begin{equation}\label{eq:triangular-kernel-estimator}
  \hat{L}_{h_{+},h_{-}}=\frac{\sum_{i=1}^{n}(g^{*}(x_{i})-f^{*}(x_{i}))y_{i}}{
    \sum_{i=1}^{n}(g_{+}^{*}(x_{i})-f_{+}^{*}(x_{i}))}=
  \frac{\sum_{i=1}^{n}k_{+}(x_{i}/h_{+})y_{i}}{
    \sum_{i=1}^{n}k_{+}(x_{i}/h_{+})}-
  \frac{\sum_{i=1}^{n}k_{-}(x_{i}/h_{-})y_{i}}{
    \sum_{i=1}^{n}k_{-}(x_{i}/h_{-})},
\end{equation}
$\sd(\hat{L}_{h_{+},h_{-}})= \big(\frac{\sum_{i} k_{+}(x_{i}/h_{+})^{2}}{
  (\sum_{i}k_{+}(x_{i}/h_{+}))^{2}}+ \frac{\sum_{i}k_{-}(x_{i}/h_{-})^{2}}{
  (\sum_{i}k_{-}(x_{i}/h_{-}))^{2}}\big)^{1/2}\cdot\sigma$ is its standard
deviation, $z_{1-\alpha}$ is the $1-\alpha$ quantile of a standard normal
distribution, and
$\bias_{f^{*}}(\hat{L}_{h_{+},h_{-}})= C\sum_{i}\abs{x_{i}}\cdot
\big(\frac{k_{+}(x_{i}/h_{+})}{\sum_{j}k_{+}(x_{j}/h_{+})}+\frac{k_{-}(x_{i}/h_{-})}{\sum_{j}k_{-}(x_{j}/h_{-})}\big)$
is the estimator's bias under $f^{*}$. The estimator $\hat{L}_{h_{+},h_{-}}$ is
normally distributed with variance that does not depend on the true function
$f$. Its bias, however, does depend on $f$. To control size under $H_{0}$ in
finite samples, it is necessary to subtract the largest possible bias of
$\hat{L}_{h}$ under the null, which obtains at $f^{*}$. Since the rejection
probability of the test is decreasing in the bias, its minimum power occurs when
the bias is minimal under $H_{1}$, which occurs at $g^{*}$, and is given by
\begin{equation}\label{eq:simple-example:minimax-power}
  \beta=
  \Phi\Big(2C\textstyle
  \sqrt{h_{+}^{2}\sum_{i}k_{+}(x_{i}/h_{+})^{2}
    +h_{-}^{2}\sum_{i}k_{-}(x_{i}/h_{-})^{2}}/\sigma-z_{1-\alpha}\Big).
\end{equation}
Since the estimator, its variance, and the non-random bias correction are all
independent of the particular null $L_{0}$, the CI based on inverting these
tests as $H_{0}$ varies over $\mathbb{R}$ is given by
\begin{equation}\label{eq:simple-example:oci}
  \hor{\hat{c}_{\alpha,h_{+},h_{-}},\infty},\quad\text{where}\quad
  \hat{c}_{\alpha,h_{+},h_{-}}=
  \hat{L}_{h_{+},h_{-}}-\bias_{f^{*}}(\hat{L}_{h_{+},h_{-}})
  -
  \sd(\hat{L}_{h_{+},h_{-}})z_{1-\alpha}.
\end{equation}
This CI minimizes the $\beta$th quantile maximum excess length with $\beta$
given by the minimax power of the tests~\eqref{eq:simple-example:minimax-power}.
Equivalently, given a quantile $\beta$ that we wish to optimize, let
$h_{+}(\beta)$ and $h_{-}(\beta)$ solve~\eqref{eq:h-plus-h-minus}
and~\eqref{eq:simple-example:minimax-power}. The optimal CI is then given by
$\hor{\hat{c}_{\alpha,h_{+}(\beta),h_{-}(\beta)},\infty}$, and the half-distance
$b$ to the alternative of the underlying tests is determined by~\eqref{eq:h-b}.
The important feature of this CI is that the bias correction is non-random: it
depends on the worst-case bias of $\hat{L}_{h_{+}(\beta),h_{-}(\beta)}$, rather
than an estimate of the bias. Furthermore, it doesn't disappear asymptotically.
One can show that the squared worst-case bias of
$\hat{L}_{h_{+}(\beta),h_{-}(\beta)}$ and its variance are both of the order
$n^{-2/3}$. Consequently, no CI that ``undersmooths'' in the sense that it is
based on an estimator whose bias is of lower order than its variance can be
minimax optimal asymptotically or in finite samples.

An apparent disadvantage of this CI is that it requires the researcher to choose
the smoothness parameter $C$. Addressing this issue leads to ``adaptive'' CIs.
Adaptive CIs achieve good excess length properties for a range of parameter
spaces $\mathcal{F}_{RDT,1}(C_{j})$, $C_{1}<\dotsb <C_{J}$, while maintaining
coverage over their union, which is given by $\mathcal{F}_{RDT,1}(C_{J})$, where
$C_{J}$ is some conservative upper bound on the possible smoothness of $f$. In
contrast, a minimax CI only considers worst-case excess length over
$\mathcal{F}_{RDT,1}(C_{J})$. To derive an upper bound on the scope for
adaptivity, consider the problem of finding a CI that optimizes excess length
over $\mathcal{F}_{RDT,1}(0)$ (the space of functions that are constant on
either side of the cutoff), while maintaining coverage over
$\mathcal{F}_{{RDT},1}(C)$ for some $C>0$.

To derive the form of such CI, consider the one-sided testing problem
$H_{0}\colon Lf\leq L_{0}$ and $f\in\mathcal{F}_{{RDT},1}(C)$ against the
one-sided alternative $H_{1}\colon f(0)\geq L_{0}+b$ and
$f\in\mathcal{F}_{{RDT},1}(0)$ (so that now the half-distance to the alternative
is given by $b/2$ rather than $b$). This is equivalent to a multivariate normal
mean problem $Y\sim\mathcal{N}(\mu,\sigma^{2}I_{n})$, with $\mu\in M_{0}$ under
the null as before, and
$\mu\in\tilde{M}_{1}=\left\{ (f(x_{1}),\dotsc,f(x_{n}))'\colon
  f\in\mathcal{F}_{RDT,1}(0),Lf\geq L_{0}+b\right\}$. Since the null and
alternative are convex, by the same arguments as before, the least favorable
functions minimize the distance between the two sets. The minimizing functions
are given by $\tilde{g}^{*}(x)=\1{x\geq 0}(L_{0}+b)$, and $\tilde{f}^{*}=f^{*}$
(same function as before). Since $\tilde{g}^{*}-\tilde{f}^{*}=(g^{*}-f^{*})/2$,
this leads to the same test and the same CI as before---the only difference is
that we moved the half-distance to the alternative from $b$ to $b/2$. Hence, the
minimax CI that optimizes a given quantile of excess length over
$\mathcal{F}_{{RDT},1}(C)$ also optimizes its excess length over the space of
constant functions, but at a different quantile. Furthermore, in
\Cref{sec:onesided-cis}, we show that the minimax CI remains highly efficient if
one compares excess length at the same quantile: in large samples, the
efficiency at constant functions is $95.2\%$. Therefore, it is not possible to
``adapt'' to cases in which the regression function is smoother than the least
favorable function. Consequently, it is not possible to tighten the minimax CI
by, say, using the data to ``estimate'' the smoothness parameter $C$.

A two-sided CI can be formed as
$\hat{L}_{h_{+},h_{-}}\pm(\bias_{f^{*}}(\hat{L}_{h_{+},h_{-}})
+\sd(\hat{L}_{h_{+},h_{-}})z_{1-\alpha/2})$, thereby accounting for possible
bias of $\hat{L}_{h_{+},h_{-}}$. However, this is conservative, since the bias
cannot be in both directions at once. Since the $t$-statistic
$(\hat{L}_{h_{+},h_{-}}-Lf)/\sd(\hat{L}_{h_{+},h_{-}})$ is normally distributed
with variance one and mean at most
$\bias_{f^{*}}(\hat{L}_{h_{+},h_{-}})/\sd(\hat{L}_{h_{+},h_{-}})$ and least
$-\bias_{f^{*}}(\hat{L}_{h_{+},h_{-}})/\sd(\hat{L}_{h_{+},h_{-}})$, a
nonconservative CI takes the form
\begin{equation*}
  \hat{L}_{h_{+},h_{-}}\pm
  \sd(\hat{L}_{h_{+},h_{-}})
  \cv_{\alpha}(\bias_{f^{*}}(\hat{L}_{h_{+},h_{-}})/\sd(\hat{L}_{h_{+},h_{-}})),
\end{equation*}
where $\cv_{\alpha}(t)$ is the $1-\alpha$ quantile of the absolute value of a
$\mathcal{N}(t,1)$ distribution, which we tabulate in Table~\ref{tab:cv-b}. The
optimal bandwidths $h_{+}$ and $h_{-}$ simply minimize the CI's length,
$2\sd(\hat{L}_{h_{+},h_{-}})\cdot
\cv_{\alpha}(\bias_{f^{*}}(\hat{L}_{h_{+},h_{-}})/\sd(\hat{L}_{h_{+},h_{-}}))$.
It can be shown that the solution satisfies~\eqref{eq:h-plus-h-minus}, so
choosing optimal bandwidths is a one-dimensional optimization problem. Since the
length doesn't depend on the data $Y$, minimizing it does not impact the
coverage properties of the CI\@. This CI corresponds to the optimal affine
fixed-length CI, as defined in \citet{donoho94}. Since the length of the CI
doesn't depend on the data $Y$, it cannot be adaptive. In
\Cref{sec:two-sided-cis_main} we derive a sharp efficiency bound that shows
that, similar to the one-sided case, these CIs are nonetheless highly efficient
relative to variable-length CIs that optimize their length at smooth functions.

The key to these non-adaptivity results is that the class $\mathcal{F}$ is
centrosymmetric (i.e. $f\in\mathcal{F}$ implies $-f\in\mathcal{F}$) and convex.
For adaptivity to be possible, it is necessary (but perhaps not sufficient) to
impose shape restrictions like monotonicity, or non-convexity of $\mathcal{F}$.

\subsection{Practical implementation}\label{sec:pract-impl}

We now discuss some practical issues that arise when implementing optimal
CIs.\footnote{An R package implementing these CIs is available at
  \url{https://github.com/kolesarm/RDHonest}.} To describe the form of the
optimal CIs for general $p\geq 1$, consider first the problem of constructing
CIs based on a linear estimator of the form
\begin{equation}\label{eq:rd-linear-estimator}
  \hat{L}_{h_{+},h_{-}}=\sum_{i=1}^{n}w_+(x_i,h_{+})y_i
  -\sum_{i=1}^{n}w_-(x_i,h_{-})y_i,
\end{equation}
where $h_{+},h_{-}$ are smoothing parameters, and the weights satisfy
$w_{+}(-x,h_{+})=w_{-}(x,h_{-})=0$ for $x\geq 0$. The estimator
$\hat{L}_{h_{+},h_{-}}$ is normally distributed with variance
$\sd(\hat{L}_{h_{+},h_{-}})^{2}
=\sum_{i=1}^{n}(w_+(x_i,h_{+})+w_{-}(x_{i},h_{-}))^{2}\sigma^{2}(x_{i})$, which
does not depend on $f$. A simple argument (see \Cref{sec:addit-deta-rd})
shows that largest possible bias of $\hat{L}_{h_{+},h_{-}}$ over the parameter
space $\mathcal{F}_{RDT,p}(C)$ is given by
\begin{equation}\label{eq:rd-worst-case-bias}
  \maxbias_{\mathcal{F}_{RDT,p}(C)}(\hat{L}_{h_{+},h_{-}})
  =C\sum_{i=1}^n\abs{w_+(x_i,h_{+})+w_-(x_i,h_{-})}\cdot\abs{x_{i}}^{p},
\end{equation}
provided that the weights are such that $\hat{L}_{h_{+},h_{-}}$ is unbiased for
$f$ that takes the form of a $(p-1)$th order polynomial on either side of cutoff
(otherwise the worst-case bias will be infinite). By arguments as in
\Cref{optimal_cis_subsec}, one can construct one- and two-sided CIs based on
$\hat{L}_{h_{+},h_{-}}$ as
\begin{equation}\label{eq:rd-oci}
  \hor{c(\hat{L}_{h_{+},h_{-}}),\infty}\qquad c(\hat{L}_{h_{+},h_{-}})=\hat{L}_{h_{+},h_{-}}-
  \maxbias_{\mathcal{F}_{RDT,p}(C)}(\hat{L}_{h_{+},h_{-}})
  -\sd(\hat{L}_{h_{+},h_{-}})z_{1-\alpha},
\end{equation}
and
\begin{equation}\label{eq:rd-flci}
  \hat{L}_{h_{+},h_{-}}\pm \cv_{\alpha}
  (\maxbias_{\mathcal{F}_{RDT,p}(C)}(\hat{L}_{h_{+},h_{-}})/\sd(\hat{L}_{h_{+},h_{-}}))
  \cdot \sd(\hat{L}_{h_{+},h_{-}}).
\end{equation}
The problem of constructing optimal two- and one- sided CIs can be cast as a
problem of finding weights $w_{+},w_{-}$ and smoothing parameters $h_{+}$ and
$h_{-}$ that lead to CIs with the shortest length, and smallest worst-case
$\beta$ quantile of excess length, respectively. The solution to this problem
follows from a generalization of results in \citet{SaYl78}. The optimal weights
$w_{+}$ and $w_{-}$ are given by a solution to a system of $2(p-1)$ equations,
described in \Cref{sec:addit-deta-rd}. When $p=1$, they reduce to the weights
$w_{+}(x_{i},h_{+})=k_{+}(x_{i}/h_{+})/\sum_{i}k_{+}(x_{i}/h_{+})$ and
$w_{-}(x_{i},h_{-})=k_{-}(x_{i}/h_{+})/\sum_{i}k_{-}(x_{i}/h_{+})$, where
$k_{+}(x_{i})=k(x_{i})\1{x_{i}\geq 0}$ and $k_{-}(x_{i})=k(x_{i})\1{x_{i}< 0}$,
and $k(u)=\max\{0,1-\abs{u}\}$ is a triangular kernel. This leads to the
triangular kernel estimator~\eqref{eq:triangular-kernel-estimator}. For $p>1$,
the optimal weights depend on the empirical distribution of the running variable
$x_{i}$.

An alternative to using the optimal weights is to use a local polynomial
estimator of order $p-1$, with kernel $k$ and bandwidths $h_{-}$ and $h_{+}$ to
the left and to the right of the cutoff. This leads to weights of the form
\begin{equation}\label{eq:lp-weights}
  w_{+}(x_{i},h_{+})= e_{1}'\left(\sum_{i}k_{+}(x_{i}/h_{+})r_{i}r_{i}'\right)^{-1}
  \sum_{i}k_{+}(x_{i}/h_{+})r_{i},
\end{equation}
and similarly for $w_{-}(x_{i},h_{-})$, where
$r_{i}=(1,x_{i},\dotsc,x_{i}^{p-1})$ and $e_{1}$ is the first unit vector. Using
the efficiency bounds we develop in \Cref{sec:general_results}, it can be shown
that, provided that the bandwidths $h_{+}$ and $h_{-}$ to the right and to the
left of the cutoff are appropriately chosen, in many cases the resulting CIs are
highly efficient. In particular, for $p=2$, using the local linear estimator
with the triangular kernel turns out to lead to near-optimal CIs (see
\Cref{sec:appl-regr-disc}).

Thus, given smoothness constants $C$ and $p$, one can construct optimal or
near-optimal CIs as follows:
\begin{enumerate}
\item\label{item:step1} Form a preliminary estimator of the conditional variance
  $\hat{\sigma}(x_{i})$. We recommend using the estimator
  $\hat\sigma^2(x_i)=\hat\sigma^2_+(0)\1{x\geq 0}+\hat\sigma^2_-(0)\1{x<0}$
  where $\hat\sigma^2_+(0)$ and $\hat\sigma^2_-(0)$ are estimates of
  $\lim_{x\downarrow 0}\sigma^2(x)$ and $\lim_{x\uparrow 0}\sigma^2(x)$
  respectively.\footnote{In the empirical application in
    \Cref{sec:appl-regr-disc}, we use estimates based on local linear regression
    residuals.}
\item Given smoothing parameters $h_{+}$ and $h_{-}$, compute the weights
  $w_{+}$ and $w_{-}$ using either~\eqref{eq:lp-weights} (for local polynomial
  estimator), or by solving the system of equations given in
  \Cref{sec:addit-deta-rd} (for the optimal estimator). Compute the worst case
  bias~\eqref{eq:rd-worst-case-bias}, and estimate the variance as
  $\widehat{\sd}(\hat{L}_{h_{+},h_{-}})^{2}
  =\sum_{i}(w_+(x_i,h_{+})+w_{-}(x_{i},h_{-}))^{2}\hat{\sigma}^{2}(x_{i})$.
\item\label{item:step3} Find the smoothing parameters $h^{*}_{+}$ and
  $h^{*}_{-}$ that minimize the $\beta$-quantile of excess length
  \begin{equation}\label{worst-case_beta_quantile_eq}
    2\maxbias_{\mathcal{F}_{RDT,p}(c)}(
    \hat{L}_{h_{+},h_{-}})+
    \sd(\hat{L}_{h_{+},h_{-}})(z_{1-\alpha}+z_{\beta}).
  \end{equation}
  for a given $\beta$. The choice $\beta=0.8$, corresponds to a benchmark used
  in statistical power analysis \citep[see][]{cohen88}. For two-sided CIs,
  minimize the length
  \begin{equation}\label{flci_length_eq}
    2 \widehat\sd(\hat{L}_{h_{+},h_{-}})
    \cv_{\alpha}\left(\maxbias_{\mathcal{F}_{RDT,p}(C)}(
      \hat{L}_{h_{+},h_{-}})/
      \widehat\sd(\hat{L}_{h_{+},h_{-}})\right).
  \end{equation}

\item\label{item:step4} Construct the CI using~\eqref{eq:rd-oci} (for one-sided
  CIs), or~\eqref{eq:rd-flci} (for two-sided CIs), based on
  $\hat{L}_{h^{*}_{+},h^{*}_{-}}$, with
  $\widehat{\sd}(\hat{L}_{h^{*}_{+},h^{*}_{-}})$ in place of the (infeasible)
  true standard deviation.
\end{enumerate}
\begin{remark}\label{remark:nn-variance}
  The variance estimator in step~\ref{item:step1} leads to asymptotically valid
  and optimal inference even when $\sigma^2(x)$ is non-constant, so long as it
  is smooth on either side of the cutoff. However, finite-sample properties of
  the resulting CI may not be good if heteroskedasticity is important for the
  sample size at hand. We therefore recommend using the variance estimator
  \begin{equation}\label{eq:robust-variance}
    \widehat{\sd}_{\text{robust}}(\hat{L}_{h^{*}_{+},h^{*}_{-}})^{2}
    =\sum_{i=1}^{n}(w_+(x_i,h_{+})+w_{-}(x_{i},h_{-}))^{2}\hat{u}_{i}^2
  \end{equation}
  instead of $\widehat{\sd}(\hat{L}_{h^{*}_{+},h^{*}_{-}})$ in
  step~\ref{item:step4}, where $\hat{u}_{i}^{2}$ is an estimate of
  $\sigma^{2}(x_{i})$. When using local polynomial regression, one can set
  $\hat{u}_{i}$ to the $i$th regression residual, in which
  case~\eqref{eq:robust-variance} reduces to the usual Eicker-Huber-White
  estimator. Alternatively, one can use the nearest-neighbor estimator
  \citep{AbIm06match}
  $\hat{u}_{i}^{2}=\frac{J}{J+1}(Y_{i}-J^{-1}\sum_{\ell=1}^{J}Y_{j_{\ell}(i)})^{2}$,
  where $j_{\ell}(i)$ is the $\ell$th closest unit to $i$ among observations on
  the same side of the cutoff, and $J\geq 1$ (we use $J=3$ in the application in
  \Cref{sec:appl-regr-disc}, following \citealp{cct14}). This mirrors the common
  practice of assuming homoskedasticity to compute the optimal weights, but
  allowing for heteroskedasticity when performing inference, such as using OLS
  in the linear regression model (which is efficient under homoskedasticity)
  along with heteroskedasticity-robust standard errors.
\end{remark}

\begin{remark}\label{MSE_bandwidth_remark}
  If one is interested in estimation, rather than inference, one can choose
  $h_{+}$ and $h_{-}$ that minimize the worst-case mean-squared error (MSE)
  $\maxbias_{\mathcal{F}_{RDT,p}(C)}(\hat{L}_{h_{+},h_{-}})^2+\sd(\hat{L}_{h_{+},h_{-}})^2$
  instead of the CI criteria in step~\ref{item:step3}. One can form a CI around
  this estimator by simply following step~\ref{item:step4} with this choice of
  $h_{+}$ and $h_{-}$. In the application in \Cref{sec:appl-regr-disc}, we find
  that little efficiency is lost by using MSE-optimal smoothing parameters,
  relative to using $h_{+}$ and $h_{-}$ that minimize the CI
  length~\eqref{flci_length_eq}. Interestingly, we find that smoothing
  parameters that minimize the CI length actually oversmooth slightly relative
  to the MSE optimal smoothing parameters. We generalize these findings in an
  asymptotic setting in \citet{ArKo15}.
\end{remark}
\begin{remark}\label{remark:covariates}
  Often, a set of covariates $z_i$ will be available that does not depend on the
  treatment, but that may be correlated with the outcome variable $y_i$. If the
  parameter of interest is still the average treatment effect for units with
  $x_i=0$, one can simply ignore these covariates. Alternatively, to gain
  additional precision, as suggested in \citet{ccft16}, one can run a local
  polynomial regression, but with the covariates added linearly. In
  \Cref{covariate_section_append}, we show that this approach is
  near-optimal if one places smoothness assumptions on the conditional mean of
  $\tilde{y}_i$ given $x_i$, where $\tilde{y}_i$ is the outcome with the effect
  of $z_i$ partialled out. If one is interested in the treatment effect as a
  function of $z$ (with $x$ still set to zero), one can use our general
  framework by considering the model $y_i=f(x_i,z_i)+u_i$, specifying a
  smoothness class for $f$, and constructing CIs for
  $\lim_{x\downarrow 0}f(x,z)-\lim_{x\uparrow 0}f(x,z)$ for different values of
  $z$. See \Cref{covariate_section_append} for details.
\end{remark}

A final consideration in implementing these CIs in practice is the choice of the
smoothness constants $C$ and $p$. The choice of $p$ depends on the order of the
derivative the researcher wishes to bound. Since much of empirical practice in
RD is justified by asymptotic MSE optimality results for
$\mathcal{F}_{RDT,2}(C)$ (in particular, this class justifies the use of local
linear estimators), we recommend $p=2$ as a default choice. For $C$,
generalizations of the non-adaptivity results described in
\Cref{optimal_cis_subsec} show that the researcher must choose $C$ a priori,
rather than attempting to use the data to choose $C$. To assess the sensitivity
of the results to different smoothness assumptions on $f$, we recommend
considering a range of plausible choices for $C$. We implement this approach for
our empirical application in \Cref{sec:appl-regr-disc}.

\section{General characterization of optimal procedures}%
\label{sec:general_results}

We consider the following setup and notation, much of which follows
\citet{donoho94}. We observe data $Y$ of the form
\begin{equation}\label{eq:donoho-model}
  Y=Kf+\sigma\varepsilon
\end{equation}
where $f$ is known to lie in a convex subset $\mathcal{F}$ of a vector space,
and $K:\mathcal{F}\to \mathcal{Y}$ is a linear operator between $\mathcal{F}$
and a Hilbert space $\mathcal{Y}$. We denote the inner product on
$\mathcal{Y}$ by $\langle \cdot,\cdot\rangle$, and the norm by $\|\cdot \|$. The
error $\varepsilon$ is standard Gaussian with respect to this inner
product: for any $g\in\mathcal{Y}$, $\langle \varepsilon,g\rangle$ is normal
with $E\langle \varepsilon,g\rangle=0$ and
$\text{var}\left(\langle \varepsilon, g\rangle\right)=\|g\|^2$. We are
interested in constructing a confidence set for a linear functional $Lf$.

The RD model~\eqref{eq:fixed_design_eq} fits into this setup by setting
$Y=(y_{1}/\sigma(x_1),\ldots,\allowbreak y_{n}/\sigma(x_n))'$,
$\mathcal{Y}=\mathbb{R}^n$,
$Kf=(f(x_1)/\sigma(x_1),\ldots,\allowbreak f(x_n)/\sigma(x_n))'$,
$Lf=\lim_{x\downarrow 0}f(x)-\lim_{x\uparrow 0}f(x)$ and $\langle x,y\rangle$
given by the Euclidean inner product $x'y$. As we discuss in detail in
\Cref{sec:ap:special_cases}, our setup covers a number of other
important models, including average treatment effects under unconfoundedness,
the partly linear model, constraints on the sign or magnitude of parameters
in the linear regression model, and other parametric models.

\subsection{Performance criteria}\label{sec:perf-crit-class}

Let us now define the performance criteria that we use to evaluate confidence
sets for $Lf$. A set $\mathcal{C}=\mathcal{C}(Y)$ is called a
$100\cdot (1-\alpha)\%$ confidence set for $Lf$ if
$\inf_{f\in\mathcal{F}} P_f\left(Lf\in\mathcal{C}\right)\ge 1-\alpha$. We denote
the collection of all $100\cdot (1-\alpha)\%$ confidence sets by
$\mathcal{I}_{\alpha}$.

We can compare performance of confidence sets at a particular $f\in\mathcal{F}$
using expected length, $E_{f}\lambda(\mathcal{C})$, where $\lambda$ is Lebesgue
measure. Allowing confidence sets to have arbitrary form may make them difficult
to interpret or even compute. One way of avoiding this is to restrict attention
to confidence sets that take the form of a fixed-length confidence interval
(CI), an interval of the form $[\hat L-\chi,\hat L+\chi]$ for some estimate
$\hat L$ and nonrandom $\chi$ (for instance, in the RD
model~\eqref{eq:fixed_design_eq}, $\chi$ may depend on the running variable
$x_{i}$ and $\sigma^{2}(x_{i})$, but not on $y_{i}$). Let
\begin{equation*}
  \Rfl{\alpha}(\hat L)=\min\big\{\chi\colon \inf_{f\in\mathcal{F}}
  P_f\big(\abs{\hat L-Lf}\le \chi \big)\ge 1-\alpha \big\}
\end{equation*}
denote the half-length of the shortest fixed-length $100\cdot(1-\alpha)\%$ CI
centered around an estimator $\hat L$. Fixed-length CIs are easy to compare: one
simply prefers the one with the shortest half-length. On the other hand, their
length cannot ``adapt'' to reflect greater precision for different functions
$f\in\mathcal{F}$. To address this concern, in \Cref{sec:two-sided-cis_main}, we
compare the length of fixed-length CIs to sharp bounds on the optimal expected
length $\inf_{\mathcal{C}\in\mathcal{I}_{\alpha}}E_{f}(\mathcal{C})$.

If $\mathcal{C}$ is restricted to take the form of a one-sided CI
$\hor{\hat c,\infty}$, we cannot use expected length as a criterion. We
therefore measure performance at a particular parameter $f$ using the $\beta$th
quantile of their excess length $Lf-\hat c$, which we denote by
$q_{f,\beta}(Lf-\hat c)$. To measure performance globally over some set
$\mathcal{G}$, we use the maximum $\beta$th quantile of the excess length,
\begin{equation}\label{eq:excess_length_eq}
  \Rlower{\beta}(\hat c,\mathcal{G})
  =\sup_{g\in\mathcal{G}} q_{g,\beta}(Lg-\hat c).
\end{equation}
If $\mathcal{G}=\mathcal{F}$, minimizing $\Rlower{\beta}(\hat c,\mathcal{F})$
over one-sided CIs in the set $\mathcal{I}_{\alpha}$ gives minimax excess
length. If $\mathcal{G}\subset\mathcal{F}$ is a class of smoother functions,
minimizing $\Rlower{\beta}(\hat c,\mathcal{G})$ yields CIs that direct power:
they achieve good performance when $f$ is smooth, while maintaining coverage
over all of $\mathcal{F}$. A CI that achieves good performance over multiple
classes $\mathcal{G}$ is said to be ``adaptive'' over these classes. In
\Cref{sec:onesided-cis}, we give sharp bounds on~\eqref{eq:excess_length_eq} for
a single class $\mathcal{G}$, which gives a benchmark for adapting over multiple
classes \citep[cf.][]{CaLo04}.

\subsection{Affine estimators and optimal bias-variance
  tradeoff}\label{affine_estimators_sec}

Many popular estimators are linear functions of the outcome variable $Y$, and we
will see below that optimal or near-optimal CIs are based on estimators of this
form. In the general framework~\eqref{eq:donoho-model}, linear estimators take
the form $\langle w,Y\rangle$ for some non-random $w\in\mathcal{Y}$, which
simplifies to~\eqref{eq:rd-linear-estimator} in the RD model. It will be
convenient to allow for a recentering by some constant $a\in\mathbb{R}$, which
leads to an affine estimator $\hat L=a+\langle w,Y\rangle$.

For any estimator $\hat L$, let
$\maxbias_{\mathcal{G}}(\hat{L})=\sup_{f\in\mathcal{G}}E_{f}(\hat{L}-Lf)$ and
$\minbias_{\mathcal{G}}(\hat{L})=\inf_{f\in\mathcal{G}}E_{f}(\hat{L}-Lf)$. An
affine estimator $\hat L=a+\langle w,Y\rangle$ follows a normal distribution
with mean $E_f \hat L=a+\langle w, Kf\rangle$ and variance
$\text{var}(\hat L)=\|w\|^2\sigma^2$, which does not depend on $f$. Thus, the
set of possible distributions for $\hat L-Lf$ as $f$ varies over a given convex
set $\mathcal{G}$ is given by the set of normal distributions with variance
$\|w\|^2\sigma^2$ and mean between $\minbias_{\mathcal{G}}(\hat L)$ and
$\maxbias_{\mathcal{G}}(\hat{L})$. It follows that a one-sided CI based on an
affine estimator $\hat L$ is given by
\begin{equation}\label{affine_oci}
  \hor{\hat c,\infty}\qquad
  \hat c=\hat L-\maxbias_{\mathcal{F}}(\hat L)-\sd(\hat L)z_{1-\alpha},
\end{equation}
with $z_{1-\alpha}$ denoting the $1-\alpha$ quantile of a standard normal
distribution, and that its worst-case $\beta$th quantile excess length over a
convex class $\mathcal{G}$ is
\begin{align}\label{affine_oci_qbeta}
  \Rlower{\beta}(\hat c,\mathcal{G})
  =\maxbias_{\mathcal{F}}(\hat L)-\minbias_{\mathcal{G}}(\hat L)+\sd(\hat L)(z_{1-\alpha}+z_\beta).
\end{align}
The shortest fixed-length CI centered at the affine estimator $\hat L$ is given
by
\begin{equation}\label{affine_flci}
  \hat L\pm \chi_\alpha(\hat L),
  \qquad
  \chi_\alpha(\hat L)=\cv_{\alpha}\left(\frac{
  \max\{|\maxbias_{\mathcal{F}}(\hat L)|,|\minbias_{\mathcal{F}}(\hat L)|\}}{
  \sd(\hat L)}\right)
  \cdot \sd(\hat L),
\end{equation}
where $\cv_\alpha(t)$ is the $1-\alpha$ quantile of the absolute value of a
$\mathcal{N}(t,1)$ random variable, as tabulated in Table~\ref{tab:cv-b}.

The fact that optimal CIs turn out to be based on affine estimators reduces the
derivation of optimal CIs to bias-variance calculations: since the performance
of CIs based on affine estimators depends only on the variance and worst-case
bias, one simply minimizes worst-case bias subject to a bound on variance, and
then trades off bias and variance in a way that is optimal for the given
criterion. The main tool for doing this is the ordered modulus of continuity
between $\mathcal{F}$ and $\mathcal{G}$ \citep{CaLo04},
\begin{equation*}
  \omega(\delta;\mathcal{F},\mathcal{G})
  =\sup \left\{Lg-Lf\colon \|K(g-f)\|\le
    \delta, f\in\mathcal{F},g\in\mathcal{G}\right\}
\end{equation*}
for any sets $\mathcal{F}$ and $\mathcal{G}$ with a non-empty intersection (so
that the set over which the supremum is taken is non-empty). When
$\mathcal{G}=\mathcal{F}$, $\omega(\delta;\mathcal{F},\mathcal{F})$ is the
(single-class) modulus of continuity over $\mathcal{F}$ \citep{DoLi91iii}, and
we denote it by $\omega(\delta;\mathcal{F})$. The ordered modulus
$\omega(\cdot;\mathcal{F},\mathcal{G})$ is concave, which implies that the
superdifferential at $\delta$ (the set of slopes of tangent lines at
$(\delta,\omega(\delta;\mathcal{F},\mathcal{G}))$) is nonempty for any
$\delta>0$. Throughout the paper, we let
$\omega'(\delta;\mathcal{F},\mathcal{G})$ denote an (arbitrary unless otherwise
stated) element in this set. Typically, $\omega(\cdot;\mathcal{F},\mathcal{G})$
is differentiable, in which case $\omega'(\delta;\mathcal{F},\mathcal{G})$ is
defined uniquely as the derivative at $\delta$. We use
$g^{*}_{\delta,\mathcal{F},\mathcal{G}}$ and
$f^{*}_{\delta,\mathcal{F},\mathcal{G}}$ to denote a solution to the ordered
modulus problem (assuming it exists), and
$f^*_{M,\delta,\mathcal{F},\mathcal{G}}=(f^*_{\delta,\mathcal{F},\mathcal{G}}+g^*_{\delta,\mathcal{F},\mathcal{G}})/2$
to denote the midpoint.\footnote{See \Cref{translation_invariance_sec}
  for sufficient conditions for differentiability and a discussion of the
  non-differentiable case. Regarding existence of a solution to the modulus
  problem, we verify this directly for our RD application in
  \Cref{rd_modulus_solution_sec}; see also \citet{donoho94}, Lemma 2 for
  a general set of sufficient conditions.}

We will show that optimal decision rules will in general depend on the data $Y$
through an affine estimator of the form
\begin{equation}\label{eq:Ldelta_eq}
  \hat L_{\delta,\mathcal{F},\mathcal{G}}
  =Lf_{M,\delta,\mathcal{F},\mathcal{G}}^{*}+\frac{\omega'(\delta;\mathcal{F},\mathcal{G})}{\delta}
  \left\langle K(g^{*}_{\delta,\mathcal{F},\mathcal{G}}-f^{*}_{\delta,\mathcal{F},\mathcal{G}}),Y-Kf^{*}_{M,\delta,\mathcal{F},\mathcal{G}} \right\rangle,
\end{equation}
with $\delta$ and $\mathcal{G}$ depending on the optimality criterion. When
$\mathcal{F}=\mathcal{G}$, we denote the estimator
$\hat{L}_{\delta,\mathcal{F},\mathcal{F}}$ by $\hat{L}_{\delta,\mathcal{F}}$.
When the sets $\mathcal{F}$ and $\mathcal{G}$ are clear from the context, we use
$\omega(\delta)$, $\hat{L}_{\delta}$, $f_{\delta}^{*}$, $g^{*}_{\delta}$ and
$f^{*}_{M,\delta}$ in place of $\omega(\delta;\mathcal{F},\mathcal{G})$,
$\hat{L}_{\delta,\mathcal{F},\mathcal{G}}$,
$f_{\delta,\mathcal{F},\mathcal{G}}^{*}$,
$g^{*}_{\delta,\mathcal{F},\mathcal{G}}$ and
$f^{*}_{M,\delta,\mathcal{F},\mathcal{G}}$ to avoid notational clutter.

As we show in \Cref{th:max_bias_lemma} in the Appendix, a useful property of
$\hat{L}_{\delta,\mathcal{F},\mathcal{G}}$ is that its maximum bias over
$\mathcal{F}$ and minimum bias over $\mathcal{G}$ are attained at
$f^{*}_{\delta}$ and $g^{*}_{\delta}$, respectively, and are given by
\begin{equation}\label{eq:maxbias-minbias}
  \maxbias_{\mathcal{F}}( \hat{L}_{\delta,\mathcal{F},\mathcal{G}})
  =-\minbias_{\mathcal{G}}( \hat{L}_{\delta,\mathcal{F},\mathcal{G}})
  =\frac{1}{2}\left(\omega(\delta;\mathcal{F},\mathcal{G})
    -\delta\omega'(\delta;\mathcal{F},\mathcal{G})\right).
\end{equation}
Its standard deviation equals
$\sd(\hat{L}_{\delta,\mathcal{F},\mathcal{G}})=\sigma\omega'(\delta;\mathcal{F},\mathcal{G})$,
and doesn't depend on $f$. As remarked by \citet{CaLo04nonconvex}, no estimator
can simultaneously achieve lower maximum bias over $\mathcal{F}$, higher minimum
bias over $\mathcal{G}$, and lower variance than the estimators in the class
$\{\hat{L}_{\delta,\mathcal{F},\mathcal{G}}\}_{\delta>0}$.
Estimators~\eqref{eq:Ldelta_eq} can thus be used to optimally trade off various
levels of bias and variance.

A condition that will play a central role in bounding the gains from directing power
at smooth functions is \emph{centrosymmetry}.
We say that a class $\mathcal{F}$ is \emph{centrosymmetric} if
$f\in\mathcal{F}\Longrightarrow -f\in\mathcal{F}$. Under centrosymmetry, the
functions that solve the single-class modulus problem can be taken to satisfy
$g^{*}_{\delta}=-f^{*}_{\delta}$, and the modulus is given by
\begin{equation}\label{eq:centrosymmetric_modulus_eq}
  \omega(\delta;\mathcal{F})= \sup \left\{2Lf\colon \|K f\|\le \delta/2,
    f\in\mathcal{F}\right\}.
\end{equation}
Since $f_{\delta}^*=-g_{\delta}^*$, $f^*_{M,\delta}$ is the zero function and
$\hat L_{\delta,\mathcal{F}}$ is linear:
\begin{equation}\label{eq:L-delta-centrosymmetry}
  \hat{L}_{\delta,\mathcal{F}}
  =\frac{2\omega'(\delta;\mathcal{F})}{\delta}
  \langle Kg^*_{\delta},Y\rangle.
\end{equation}
In the RD model~\eqref{eq:fixed_design_eq} the class
$\mathcal{F}_{RDT,p}(C)$ is centrosymmetric, and the estimator
$\hat L_{\delta,\mathcal{F}_{RDT,p}(C)}$ takes the form $\hat L_{h_+,h_-}$ given
in~\eqref{eq:rd-linear-estimator} for a certain class of weights $w_+(x,h_+)$
and $w_-(x,h_-)$, with the smoothing parameters $h_+$ and $h_-$ both determined
by $\delta$ (see \Cref{sec:addit-deta-rd}).

\subsection{Optimal one-sided CIs}\label{sec:onesided-cis}

Given $\beta$, a one-sided CI that minimizes~\eqref{eq:excess_length_eq} among
all one-sided CIs with level $1-\alpha$ is based on
$\hat{L}_{\delta_{\beta};\mathcal{F},\mathcal{G}}$, where
$\delta_{\beta}=\sigma(z_{\beta}+z_{1-\alpha})$.

\begin{theorem}\label{th:one_side_adapt_thm}
  Let $\mathcal{F}$ and $\mathcal{G}$ be convex with
  $\mathcal{G}\subseteq\mathcal{F}$, and suppose that $f^*_\delta$ and
  $g^*_\delta$ achieve the ordered modulus at $\delta$ with
  $\|K(f^*_\delta-g^*_\delta)\|=\delta$. Let
  \begin{equation*}
    \hat c_{\alpha,\delta,\mathcal{F},\mathcal{G}}
    =\hat{L}_{\delta,\mathcal{F},\mathcal{G}}-
    \maxbias_{\mathcal{F}}(\hat{L}_{\delta,\mathcal{F},\mathcal{G}})-z_{1-\alpha}\sigma
    \omega'(\delta;\mathcal{F},\mathcal{G}).
  \end{equation*}
  Then $\hat{c}_{\alpha,\delta,\mathcal{F},\mathcal{G}}$ minimizes
  $\Rlower{\beta}(\hat c,\mathcal{G})$ for
  $\beta=\Phi(\delta/\sigma-z_{1-\alpha})$ among all one-sided $1-\alpha$ CIs,
  where $\Phi$ denotes the standard normal cdf. The minimum coverage is taken at
  $f^{*}_{\delta}$ and equals $1-\alpha$. All quantiles of excess length are
  maximized at $g^*_\delta$. The worst case $\beta$th quantile of excess length
  is
  $\Rlower{\beta}(\hat{c}_{\alpha,\delta,\mathcal{F},\mathcal{G}},\mathcal{G})=
  \omega(\delta;\mathcal{F},\mathcal{G})$.
\end{theorem}

Since the worst-case bias of $\hat L_{\delta,\mathcal{F},\mathcal{G}}$ is given
by~\eqref{eq:maxbias-minbias}, and its standard deviation equals
$\sigma\omega'(\delta;\mathcal{F},\mathcal{G})$, it can be seen that
$\hat{c}_{\alpha,\delta,\mathcal{F},\mathcal{G}}$ takes the form given
in~\eqref{affine_oci}, and its worst-case excess length
follows~\eqref{affine_oci_qbeta}. The assumption that the modulus is achieved
with $\|K(f^*_\delta-g^*_\delta)\|=\delta$ rules out degenerate cases: if
$\|K(f^*_\delta-g^*_\delta)\|<\delta$, then relaxing this constraint does not
increase the modulus, which means that
$\omega'(\delta;\mathcal{F},\mathcal{G})=0$ and the optimal CI does not depend
on the data.

Implementing the CI from \Cref{th:one_side_adapt_thm} requires the researcher to
choose a quantile $\beta$ to optimize, and to choose the set $\mathcal{G}$.
There are two natural choices for $\beta$. If the objective is to optimize the
performance of the CI ``on average'', then optimizing the median excess length
($\beta=0.5$) is a natural choice. Since for any CI $\hor{\hat{c},\infty}$ such
that $\hat{c}$ is affine in the data $Y$, the median and expected excess lengths
coincide, and since $\hat{c}_{\alpha,\delta,\mathcal{F},\mathcal{G}}$ is affine
in the data, setting $\beta=0.5$ also has the advantage that it minimizes the
expected excess length among affine CIs. Alternatively, if the CI is
being computed as part of a power analysis, then setting $\beta=0.8$ is natural,
as, under conditions given in \Cref{translation_invariance_sec}, it
translates directly to statements about 80\% power, a standard benchmark in such
analyses \citep{cohen88}.

For the set $\mathcal{G}$, there are two leading choices. First, setting
$\mathcal{G}=\mathcal{F}$ yields minimax CIs:
\begin{corollary}[One-sided minimax CIs]\label{th:one_sided_minimax_thm}
  Let $\mathcal{F}$ be convex, and suppose that $f^{*}_{\delta}$ and
  $g^{*}_{\delta}$ achieve the single-class modulus at $\delta$ with
  $\|K(f^*_\delta-g^*_\delta)\|=\delta$. Let
  \begin{equation*}
    \hat c_{\alpha,\delta,\mathcal{F}} =
    \hat{L}_{\delta,\mathcal{F}}-\frac{1}{2}\left(\omega(\delta;\mathcal{F})-\delta
    \omega'(\delta;\mathcal{F}) \right) -z_{1-\alpha}\sigma
    \omega'(\delta;\mathcal{F}).
  \end{equation*}
  Then, for $\beta=\Phi(\delta/\sigma-z_{1-\alpha})$,
  $\hat{c}_{\alpha,\delta,\mathcal{F}}$ minimizes the maximum $\beta$th quantile
  of excess length among all $1-\alpha$ CIs for $Lf$. The minimax excess length
  is given by $\omega(\delta;\mathcal{F})$.
\end{corollary}
The minimax criterion may be considered overly pessimistic: it focuses on
controlling the excess length under the least favorable function. This leads to
the second possible choice for $\mathcal{G}$, a smaller convex class of smoother
functions $\mathcal{G}\subset\mathcal{F}$. The resulting CIs will then achieve
the best possible performance when $f$ is smooth, while maintaining coverage
over all of $\mathcal{F}$. Unfortunately, there is little scope for improvement
for such a CI when $\mathcal{F}$ is centrosymmetric. In particular, suppose that
$g^{*}_{\delta,\mathcal{F},\mathcal{G}}$ is ``sufficiently smooth'' relative to
$\mathcal{F}$, in the sense that
\begin{equation}\label{eq:f-g-in-F}
  f-g^{*}_{\delta,\mathcal{F},\mathcal{G}}\in\mathcal{F}\quad
  \text{for all $f\in\mathcal{F}$.}
\end{equation}
Since $\mathcal{F}$ is centrosymmetric, this condition is equivalent to the
requirement that the sets
$\{f-g^{*}_{\delta,\mathcal{F},\mathcal{G}}\colon f\in\mathcal{F}\}$ and
$\mathcal{F}$ are the same.\footnote{We thank a referee for pointing this out.}
For instance,~\eqref{eq:f-g-in-F} holds if $\mathcal{G}$ contains the zero
function only. In the RD model~\eqref{eq:fixed_design_eq} with
$\mathcal{F}=\mathcal{F}_{RDT,p}(C)$,~\eqref{eq:f-g-in-F} holds if
$\mathcal{G}=\mathcal{F}_{RDT,p}(0)$, the class of piecewise polynomial
functions.

\begin{corollary}\label{th:centrosymmetric_adaptation_corollary}
  Let $\mathcal{F}$ be centrosymmetric, and let
  $\mathcal{G}\subseteq\mathcal{F}$ be any convex set such that the solution to
  the ordered modulus problem exists and satisfies~\eqref{eq:f-g-in-F} with
  $\|K(f^*_{\delta_\beta}-g^*_{\delta_\beta})\|=\delta_\beta$, where
  $\delta_{\beta}=\sigma(z_{\beta}+z_{1-\alpha})$. Then the one-sided CI
  $\hat{c}_{\alpha,\delta_{\beta},\mathcal{F}}$ that is minimax for the
  $\beta$th quantile also optimizes
  $\Rlower{\tilde{\beta}}(\hat{c};\mathcal{G})$, where
  $\tilde{\beta}=\Phi((z_{\beta}-z_{1-\alpha})/2)$. In particular,
  $\hat{c}_{\alpha,\delta_{\beta},\mathcal{F}}$ optimizes
  $\Rlower{\tilde{\beta}}(\hat c;\left\{0\right\})$. Moreover, the efficiency of
  $\hat{c}_{\alpha,\delta_{\beta},\mathcal{F}}$ for the $\beta$th quantile of
  maximum excess length over $\mathcal{G}$ is given by
  \begin{equation}\label{eq:minimax-onesidedCI-efficiency}
    \frac{\inf_{\hat{c}\colon \hor{\hat{c},\infty}\in\mathcal{I}_{\alpha}}
      \Rlower{\beta}(\hat{c}, \mathcal{G})}{
      \Rlower{\beta}(\hat{c}_{\alpha,\delta_{\beta},\mathcal{F}}, \mathcal{G})}=
    \frac{ \omega(\delta_{\beta};\mathcal{F},\mathcal{G})}{
      \Rlower{\beta}(\hat{c}_{\alpha,\delta_{\beta},\mathcal{F}}, \mathcal{G})}
    = \frac{\omega(2\delta_{\beta};\mathcal{F})}{\omega(\delta_{\beta};
      \mathcal{F})+\delta_{\beta}\omega'(\delta_{\beta};\mathcal{F})}.
  \end{equation}
\end{corollary}
The first part of \Cref{th:centrosymmetric_adaptation_corollary} states that
minimax CIs that optimize a particular quantile $\beta$ will also minimize the
maximum excess length over $\mathcal{G}$ at a different quantile
$\tilde{\beta}$. For instance, a CI that is minimax for median excess length
among 95\% CIs also optimizes $\Phi(-z_{0.95}/2)\approx 0.205$ quantile under
the zero function. Vice versa, the CI that optimizes median excess length under
the zero function is minimax for the $\Phi(2z_{0.5}+z_{0.95})=0.95$ quantile.

The second part of~\Cref{th:centrosymmetric_adaptation_corollary} gives the
exact cost of optimizing the ``wrong'' quantile $\tilde{\beta}$. Since the
one-class modulus is concave, $\delta \omega'(\delta)\leq \omega(\delta)$, and
we can lower bound the efficiency of
$\hat{c}_{\alpha,\delta_{\beta},\mathcal{F}}$ given
in~\eqref{eq:minimax-onesidedCI-efficiency} by
$\omega(2\delta_{\beta})/(2\omega(\delta_{\beta}))\geq 1/2$. Typically, the
efficiency is much higher. In particular, in the regression
model~\eqref{eq:fixed_design_eq}, the one-class modulus satisfies
\begin{equation}\label{eq:asymptotoc-modulus}
  \omega(\delta;\mathcal{F})=n^{-r/2}A\delta^{r}(1+o(1))
\end{equation}
for many choices of $\mathcal{F}$ and $L$, as $n\to \infty$ for some constant
$A$, where $r/2$ is the rate of convergence of the minimax root MSE\@. This is
the case under regularity conditions in the RD model with $r=2p/(2p+1)$ by
Lemma~\ref{rd_uniform_modulus_convergence_lemma} (see \citealp{DoLo92}, for
other cases where~\eqref{eq:asymptotoc-modulus} holds). In this case,
\eqref{eq:minimax-onesidedCI-efficiency} evaluates to
$\frac{2^{r}}{1+r}(1+o(1))$, so that the asymptotic efficiency depends only on
$r$. Figure~\ref{fig:asymptotic-adaptivity-bounds} plots the asymptotic
efficiency as a function of $r$. Since adapting to the zero function easier than
adapting to any set $\mathcal{G}$ that includes it, if $\mathcal{F}$ is convex
and centrosymmetric, ``directing power'' yields very little gain in excess
length no matter how optimistic one is about where to direct it.

This result places a severe bound on the scope for adaptivity in settings in
which $\mathcal{F}$ is convex and centrosymmetric: any CI that performs better
than the minimax CI by more than the ratio
in~\eqref{eq:minimax-onesidedCI-efficiency} must fail to control coverage at
some $f\in\mathcal{F}$.

\subsection{Two-sided CIs}\label{sec:two-sided-cis_main}

A fixed-length CI based on $\hat{L}_{\delta,\mathcal{F}}$ can be computed by
plugging its worst-case bias~\eqref{eq:maxbias-minbias}
into~\eqref{affine_flci},\footnote{We assume that
  $\omega'(\delta;\mathcal{F})=\sd(\hat{L}_{\delta,\mathcal{F}})/\sigma\neq 0$.
  Otherwise, the estimator $\hat{L}_{\delta,\mathcal{F}}$ doesn't depend on the
  data, and the only valid fixed-length CI around it is the trivial CI that
  reports the whole parameter space for $Lf$.}
\begin{equation*}
  \hat{L}_{\delta,\mathcal{F}}\pm \chi_\alpha(\hat{L}_{\delta,\mathcal{F}}),
  \qquad
  \chi_\alpha(\hat{L}_{\delta,\mathcal{F}})=\cv_{\alpha}\left(\textstyle\frac{
      \omega(\delta;\mathcal{F})}{
      2\sigma\omega'(\delta;\mathcal{F})}
    -\frac{
      \delta }{
      2\sigma}
  \right)
  \cdot \sigma\omega'(\delta;\mathcal{F}).
\end{equation*}
The optimal $\delta$ minimizes the half-length,
$\delta_{\chi}=\argmin_{\delta>0}\chi_{\alpha}(\hat{L}_{\delta,\mathcal{F}})$.
It follows from \citet{donoho94} that this CI is the shortest possible in the
class of fixed-length CIs based on affine estimators. Just as with minimax
one-sided CIs, one may worry that since its length is driven by the least
favorable functions, restricting attention to fixed-length CIs may be costly
when the true $f$ is smoother. The next result characterizes confidence sets
that optimize expected length at a single function $g$, and thus bounds
 the possible performance gain.
\begin{theorem}\label{th:two_sided_adaptation_to_function_thm}
  Let $g\in\mathcal{F}$, and assume that a minimizer $f_{{L_0}}$ of $\|K(g-f)\|$
  subject to $Lf={L_0}$ and $f\in\mathcal{F}$ exists for all
  ${L_0}\in\mathbb{R}$. Then the confidence set $\mathcal{C}_{g}$ that minimizes
  $E_g\lambda(\mathcal{C})$ subject to $\mathcal{C}\in\mathcal{I}_{\alpha}$
  inverts the family of tests $\phi_{L_0}$ that reject for large values of
  $\langle K(g- f_{L_0}),Y\rangle$ with critical value given by the $1-\alpha$
  quantile under $f_{L_0}$. Its expected length is
  \begin{equation*}
    E_{g}[\lambda(\mathcal{C}_{g})]=
    (1-\alpha)E\left[(\omega(\sigma(z_{1-\alpha}-Z);\mathcal{F},\left\{g\right\})+
      \omega(\sigma(z_{1-\alpha}-Z);\left\{g\right\},\mathcal{F}))\mid Z\leq
      z_{1-\alpha} \right],
  \end{equation*}
  where $Z$ is a standard normal random variable.
\end{theorem}
This result solves the problem of ``adaptation to a function'' posed by
\citet{cai_adaptive_2013}, who obtain bounds for this problem if $\mathcal{C}$
is required to be an interval. The theorem uses the observation in
\citet{pratt61} that minimum expected length CIs are obtained by inverting a
family of uniformly most powerful tests of $H_{0}\colon Lf={L_0}$ and
$f\in\mathcal{F}$ against $H_{1}\colon f=g$, which, as shown in the proof, is
given by $\phi_{L_{0}}$; the expression for the expected length of
$\mathcal{C}_{g}$ follows by computing the power of these tests. The assumption
on the existence of the minimizer $f_{L_{0}}$ means that $Lf$ is unbounded over
$\mathcal{F}$, and it is made to simplify the statement; a truncated version of
the same formula holds when $\mathcal{F}$ places a bound on $Lf$.

Directing power at a single function is seldom desirable in practice.
\Cref{th:two_sided_adaptation_to_function_thm} is very useful, however, in
bounding the efficiency of other procedures. In particular, suppose
$f-g\in\mathcal{F}$ for all $f$, so that~\eqref{eq:f-g-in-F} holds with
$\mathcal{G}=\left\{g\right\}$ (such as when $g$ is the zero function), and that
$\mathcal{F}$ is centrosymmetric. Then, by arguments in the proof of
Corollary~\ref{th:centrosymmetric_adaptation_corollary},
$\omega(\delta;\mathcal{F},\left\{g\right\})=
\omega(\delta;\left\{g\right\},\mathcal{F})=
\frac{1}{2}\omega(2\delta;\mathcal{F})$, which yields:
\begin{corollary}\label{th:centrosymmetric_adaptation_twosided}
  Consider the setup in \Cref{th:two_sided_adaptation_to_function_thm} with the
  additional assumption that $\mathcal{F}$ is centrosymmetric and $g$ satisfies
  $f-g\in\mathcal{F}$ for all $f$. Then the efficiency of the fixed-length CI
  around $\hat{L}_{\delta_{\chi},\mathcal{F}}$ at $g$ relative to all confidence
  sets is
  \begin{equation}\label{eq:finite-sample-flci-efficiency}
    \frac{\inf_{\mathcal{C}\in\mathcal{I}_{\alpha}}
      E_{g}\lambda(\mathcal{C}(Y))}{
      2\chi_{\alpha}(\hat{L}_{\delta_{\chi},\mathcal{F}})
    }
    =\frac{(1-\alpha)E\left[\omega(2\sigma(z_{1-\alpha}-Z);\mathcal{F})\mid Z\leq
        z_{1-\alpha} \right]}{
      2\cv_{\alpha}\left(
        \frac{ \omega(\delta_{\chi};\mathcal{F})
        }{2\sigma \omega'(\delta_{\chi};\mathcal{F})}
        -\frac{ \delta_{\chi}   }{2\sigma}
      \right)
      \cdot \sigma \omega'(\delta_{\chi};\mathcal{F})}.
  \end{equation}
\end{corollary}
The efficiency ratio~\eqref{eq:finite-sample-flci-efficiency} can easily be
computed in particular applications, and we do so in the empirical application
in \Cref{sec:appl-regr-disc}. When the one-class modulus
satisfies~\eqref{eq:asymptotoc-modulus}, then, as in the case of one-sided CIs,
the asymptotic efficiency of the fixed-length CI around
$\hat{L}_{\delta_{\chi}}$ can be shown to depend only on $r$ and $\alpha$, and
we plot it in Figure~\ref{fig:asymptotic-adaptivity-bounds} for $\alpha=0.05$
(see Theorem~\ref{rd_optimal_estimator_thm_main_text} for the formula). When
$r=1$ (parametric rate of convergence) and $\alpha=0.05$, the asymptotic
efficiency equals $84.99\%$, as in the normal mean example in \citet[Section
5]{pratt61}.

Just like with minimax one-sided CIs, this result places a severe bound on the
scope for improvement over fixed-length CIs when $\mathcal{F}$ is
centrosymmetric. It strengthens the finding in \citet{low97} and \citet{CaLo04},
who derive bounds on the expected length of random length $1-\alpha$ CIs. Their
bounds imply that when $\mathcal{F}$ is constrained only by bounds on a
derivative, the expected length of any CI in $\mathcal{I}_{\alpha}$ must shrink
at the minimax rate $n^{-r/2}$ for any $g$ in the interior of
$\mathcal{F}$.\footnote{One can use
  Theorem~\ref{th:two_sided_adaptation_to_function_thm} to show that this result
  holds even if we don't require $\mathcal{C}$ to take the form of an interval.
  For example, in the RD model with $\mathcal{F}=\mathcal{F}_{RDT,p}(C)$ and
  $g\in\mathcal{F}_{RDT,p}(C_{g})$, $C_{g}<C$, the result follows from lower
  bounding $E_{g}[\lambda(\mathcal{C}_{g})]$ using
  $\omega(\delta;\mathcal{F},\{g\})+\omega(\delta;\{g\},\mathcal{F})\geq
  \omega(2\delta,\mathcal{F}_{RDT,p}(C-C_{g}))$.}
Figure~\ref{fig:asymptotic-adaptivity-bounds} shows that for smooth functions
$g$, this remains true whenever $\mathcal{F}$ is centrosymmetric, even if we
don't require $\mathcal{C}$ to take the form of an interval. Importantly, the
figure also shows that not only is the rate the same as the minimax rate, the
constant must be close to that for fixed-length CIs. Since adapting to a single
function $g$ is easier than adapting to any class $\mathcal{G}$ that includes
it, this result effectively rules out adaptation to subclasses of $\mathcal{F}$
that contain smooth functions.

\section{Empirical illustration}%
\label{sec:appl-regr-disc}

In this section, we illustrate the theoretical results in an RD application
using a dataset from \citet{lee08}. The dataset contains 6,558 observations on
elections to the US House of Representatives between 1946 and 1998. The running
variable $x_{i}\in [-100,100]$ is the Democratic margin of victory (in
percentages) in election $i$. The outcome variable $y_{i}\in [0,100]$ is the
Democratic vote share (in percentages) in the next election. Given the inherent
uncertainty in final vote counts, the party that wins is essentially randomized
in elections that are decided by a narrow margin, so that the RD parameter $Lf$
measures the incumbency advantage for Democrats for elections decided by a
narrow margin---the impact of being the current incumbent party in a
congressional district on the vote share in the next election.

We consider inference under the Taylor class $\mathcal{F}_{RDT,p}(C)$, with
$p=2$. We report results for the optimal estimators and CIs, as well as CIs
based on local linear estimators, using the formulas described in
\Cref{sec:pract-impl} (which follow from the general results in
\Cref{sec:general_results}). We use the preliminary estimates
$\hat{\sigma}^{2}_{+}(x)=12.6^{2}$ and $\hat{\sigma}^{2}_{-}(x)=10.8^{2}$ in
Step~\ref{item:step1}, which are based on residuals form a local linear
regression with bandwidth selected using the \citet{ik12restud} selector. In
Step~\ref{item:step4}, we use the nearest-neighbor variance estimator with
$J=3$.

Let us briefly discuss the interpretation of the smoothness constant $C$ in this
application. By definition of the class $\mathcal{F}_{RDT,2}(C)$, $C$ determines
how large the approximation error can be if we approximate the regression
functions $f_{+}$ and $f_{-}$ on either side of the cutoff by a linear Taylor
approximation at the cutoff: the approximation error is no greater than
$Cx^{2}$. One way of gauging the magnitude of this approximation error is to
look at its effect on prediction error when using the Taylor approximation to
predict the vote share in the next election, and the margin in the previous
election was $x_0$. If one uses the Taylor approximation, the prediction MSE is
at most $C^2x_0^4+\sigma^2(x_0)$, whereas using the true conditional mean to
predict the vote share would lead to prediction MSE $\sigma^2(x_0)$. Thus, using
the true conditional mean leads to a MSE reduction in this prediction problem by
a factor of at most $C^2x_0^4/(C^2x_0^4+\sigma^2(x_0))$. If $C=0.05$ for
instance, this implies MSE reductions of at most $13.6\%$ at $x_{0}=10\%$, and
$71.5\%$ at $x_{0}=20\%$, assuming that $\sigma^{2}(x_{0})$ equals our estimate
of $12.6^{2}$. To the extent that researchers agree that the vote share in the
next election varies smoothly enough with the margin of victory in the current
election to make such large reductions in MSE unlikely, $C=0.05$ is quite a
conservative choice.

Our adaptivity bounds imply that one cannot use data-driven methods to tighten
our CIs, by say, estimating $C$. It is, however, possible to lower-bound the
value of $C$. We derive a simple estimate of this lower bound in
\Cref{sec:estim-lower-bound}, which in the \citeauthor{lee08} data
yields the lower bound estimate $0.017$. As detailed in the appendix, the lower
bound estimate can also be used in a model specification test to check whether a
given chosen value of $C$ is too low. To examine sensitivity of the results to
different choices of $C$, we present the results for the range
$C\in[0.0002, 0.1]$ that, by the argument in the preceding paragraph, includes
most plausible values.

\subsection{Optimal and near-optimal confidence intervals}%
\label{sec:optim-infer-proc}

The top panel in Figure~\ref{fig:lee-minimax} plots the optimal one- and
two-sided CIs defined in \Cref{sec:simple-example}, as well as estimates based
on minimizing the worst-case MSE (see Remark~\ref{MSE_bandwidth_remark}). The
estimates vary between 5.8\% and 7.4\% for $C\geq 0.005$, which is close to the
original Lee estimate of 7.7\% that was based on a global fourth degree
polynomial. Interestingly, the lower and upper limits $\hat{c}_{u}$ and
$\hat{c}_{\ell}$ of the one-sided CIs $\hor{\hat{c}_{\ell},\infty}$ and
$\hol{-\infty,\hat{c}_{u}}$ are not always within the corresponding limits for
the two-sided CIs. The reason for this is that for any given $C$, the optimal
smoothing parameters $h_{+}$ and $h_{-}$ are smaller for one-sided CIs than for
two-sided fixed-length CIs. Thus, when the point estimate decreases with the
amount of smoothing as is the case for low values of $C$, then one-sided CIs are
effectively centered around a lower estimate, which explains why at first the
one-sided CI limits are both below the two-sided limits. This reverses once the
point estimate starts increasing with the amount of smoothing. Furthermore, the
optimal smoothing parameters for the minimax MSE estimator are slightly
\emph{smaller} than those for fixed-length CIs throughout the entire range of
$C$s, albeit by a small amount. This matches the asymptotic predictions in
\citet{ArKo15}.

As we discussed in Remark~\ref{MSE_bandwidth_remark}, it may be desirable to
report an estimate with good MSE, with a CI centered at this estimate (without
reoptimizing the smoothing parameters). The bottom panel in
Figure~\ref{fig:lee-minimax} gives CIs with the smoothing parameters chosen so
that the $\hat{L}_{h_{+},h_{-}}$ minimizes the maximum MSE\@. The limits of the
one-sided CIs are now contained within the two-sided CIs, as they are both based
on the same estimator, although they are less than
$(z_{1-\alpha/2}-z_{1-\alpha})\sd(\hat{L}_{h_{+},h_{-}})$ apart as would be the
case if $\hat{L}_{h_{+},h_{-}}$ were unbiased. Finally, Figure~\ref{fig:lee-ll}
considers CIs based on local linear estimators with triangular kernel; these CIs
are very close to the optimal CIs in Figure~\ref{fig:lee-minimax}.

\subsection{Efficiency comparisons and bounds on adaptation}%
\label{sec:bounds-adaptation}

We now consider the relative efficiency of the different CIs reported in
Figures~\ref{fig:lee-minimax} and~\ref{fig:lee-ll}. To keep the efficiency
comparisons meaningful, we assume that the variance is homoskedastic on either
side of the cutoff, and equal to the initial estimates.

First, comparing half-length and excess length of CIs based on choosing
$h_{+},h_{-}$ to minimize the MSE to that of CIs based on optimally chosen
$h_{+}$ and $h_{-}$, we find that over the range of $C$'s considered, for both
optimal and local linear estimators, two-sided CIs based on MSE-optimal
estimators are at least 99.9\% efficient, and one-sided CIs are at least 97.7\%
efficient. These results are in line with the asymptotic results in
\citet{ArKo15}, which imply that the asymptotic efficiency of two-sided
fixed-length CIs is 99.9\%, and it is 98.0\% for one-sided CIs.

Second, comparing half-length and excess length of the CIs based on local linear
estimates to that of CIs based on optimal estimators, we find that one- and
two-sided CIs based on local linear estimators with triangular kernel are at
least 96.9\% efficient. This is very close to the asymptotic efficiency result
in \citet{ArKo15} that the local linear estimator with a triangular kernel is
97.2\% efficient, independently of the performance criterion.

Third, since the class $\mathcal{F}_{RDT,2}(C)$ is centrosymmetric, we can use
Corollaries~\ref{th:centrosymmetric_adaptation_corollary}
and~\ref{th:centrosymmetric_adaptation_twosided} to bound the scope for
adaptation to the class of piecewise linear functions,
$\mathcal{G}=\mathcal{F}_{RDT,2}(0)$. We find that the relative efficiency of
CIs that minimax the $0.8$ quantile is between 96\% and 97.4\%, and the
efficiency of fixed-length two-sided CIs at any $g\in\mathcal{G}$ is between
95.5\% and 95.9\% for the range of $C$'s considered. This is very close to the
asymptotic efficiency predictions, 96.7\% and 95.7\%, respectively, implied by
Figure~\ref{fig:asymptotic-adaptivity-bounds} (with $r=4/5$). Thus, one cannot
avoid choosing $C$ a priori.

\newpage
\begin{appendices}
  \crefalias{section}{appsec}
  \crefalias{subsection}{appsubsec}

  \section{Proofs for main results}\label{sec:ap:main-proofs}

  This section contains proofs of the results in \Cref{sec:general_results}.
  \Cref{sec:ap:auxiliary_lemmas_sec} contains auxiliary lemmas used in the
  proofs. The proofs of the results in \Cref{sec:general_results} are given in
  the remainder of the section. Proofs of \Cref{th:one_sided_minimax_thm,th:centrosymmetric_adaptation_twosided} follow immediately from the
  theorems and arguments in the main text, and their proofs are omitted.
  We assume throughout this section that the sets $\mathcal{F}$ and $\mathcal{G}$ are convex.

  Before proceeding, we recall that $\omega'(\delta;\mathcal{F},\mathcal{G})$
  was defined in \Cref{sec:general_results} to be an arbitrary element of the
  superdifferential. We denote this set by
  \begin{equation*}
    \partial \omega(\delta;\mathcal{F},\mathcal{G}) =\left\{d\colon \text{for all
    }\eta>0, \omega(\eta;\mathcal{F},\mathcal{G}) \le
    \omega(\delta;\mathcal{F},\mathcal{G})+d(\eta-\delta)\right\}.
  \end{equation*}
  It is nonempty since $\omega(\cdot;\mathcal{F},\mathcal{G})$ is concave---if
  $f^*_\delta,g^*_\delta$ attain the modulus at $\delta$ and similarly for
  $\tilde\delta$, then, for $\lambda\in[0,1]$,
  $f_\lambda=\lambda f^*_\delta+(1-\lambda)f^*_{\tilde \delta}$ and
  $g_\lambda=\lambda g^*_\delta+(1-\lambda)g^*_{\tilde \delta}$ satisfy
  $\|K(g_\lambda-f_\lambda)\|\le \lambda \delta+(1-\lambda)\tilde \delta$ so
  that
  $\omega(\lambda\delta+(1-\lambda)\tilde\delta)\ge
  Lg_\lambda-Lf_\lambda=\lambda
  \omega(\delta)+(1-\lambda)\omega(\tilde\delta)$.

  The definition of $\hat L_{\delta,\mathcal{F},\mathcal{G}}$
  in~\eqref{eq:Ldelta_eq} depends on the choice of
  $\omega'(\delta;\mathcal{F},\mathcal{G})\in \partial
  \omega(\delta;\mathcal{F},\mathcal{G})$ and
  $f^*_{\delta,\mathcal{F},\mathcal{G}},g^*_{\delta,\mathcal{F},\mathcal{G}}$.
  As we explain in \Cref{translation_invariance_sec},
  Theorem~\ref{th:one_side_adapt_thm} holds for any choice of
  $\omega'(\delta;\mathcal{F},\mathcal{G})$ so long as the same element is used
  in the definition of the estimator and worst-case bias formula.
  Regarding the choice of the particular solution
  $f^*_{\delta,\mathcal{F},\mathcal{G}},g^*_{\delta,\mathcal{F},\mathcal{G}}$
  used to construct the estimator and CIs, it turns out that, under the
  conditions of Theorem~\ref{th:one_side_adapt_thm}, the choice does not affect
  the definition of $\hat L_{\delta,\mathcal{F},\mathcal{G}}$ or the CIs based
  on it, as we now explain. If $(f^*_0,g^*_0)$ and $(f^*_1,g^*_1)$ solve the
  modulus problem with $K(g^*_0-f^*_0)\ne K(g^*_1-f^*_1)$, a strict convex
  combination $(f_\lambda,g_\lambda)$ will satisfy
  $\|K(f_\lambda-g_\lambda)\|\le \delta-\eta$ for some $\eta>0$, which implies
  $\omega(\delta-\eta;\mathcal{F},\mathcal{G})=L(g_\lambda-f_\lambda)=\omega(\delta;\mathcal{F},\mathcal{G})$.
  Since the modulus is nondecreasing, this implies that it is constant in a
  neighborhood of $\delta$, so that
  $\partial \omega(\delta;\mathcal{F},\mathcal{G})=\{0\}$. Thus, either
  $K(g^*_\delta-f^*_\delta)$ is defined uniquely or
  $\partial \omega(\delta;\mathcal{F},\mathcal{G})=\{0\}$. In either case,
  $\omega'(\delta;\mathcal{F},\mathcal{G}) \cdot K(f^*_\delta-g^*_\delta)$ is
  defined uniquely up to the choice of
  $\omega'(\delta;\mathcal{F},\mathcal{G})$, which means that, for any two
  estimators $\hat L_0$ and $\hat L_1$ that satisfy the definition of
  $\hat L_{\delta,\mathcal{F},\mathcal{G}}$ with the same choice of
  $\omega'(\delta;\mathcal{F},\mathcal{G})$, we must have $\hat L_1=\hat L_0+a$
  for some constant $a$. The bias formula~\eqref{eq:maxbias-minbias}, which
  follows from Lemma~\ref{th:max_bias_lemma} below, then implies that $a=0$.
  Similarly, the CIs $\hor{\hat c_{\alpha,\mathcal{F},\mathcal{G}},\infty}$ and
  $\hat L_{\delta,\mathcal{F},\mathcal{G}}\pm
  \chi_\alpha(\hat{L}_{\delta,\mathcal{F},\mathcal{G}})$ are defined uniquely up
  to the choice of $\omega'(\delta;\mathcal{F},\mathcal{G})$.

  \subsection{Auxiliary lemmas}%
  \label{sec:ap:auxiliary_lemmas_sec}

  The following lemma extends Lemma 4 in \citet{donoho94} to the two class
  modulus \citep[see also Theorem 2 in][for a similar result in the Gaussian
  white noise model]{CaLo04nonconvex}. The proof is essentially the same as for
  the single class case.

\begin{lemma}\label{th:max_bias_lemma}
  Let $\mathcal{F}$ and $\mathcal{G}$ be convex sets and let $f^*$ and $g^*$
  solve the optimization problem for $\omega(\delta_0;\mathcal{F},\mathcal{G})$
  with $\|K(f^*-g^*)\|=\delta_0$, and let
  $d\in \partial\omega(\delta_0;\mathcal{F},\mathcal{G})$. Then, for all
  $f\in\mathcal{F}$ and $g\in\mathcal{G}$,
  \begin{align}\label{between_class_bias_ineq}
    Lg-Lg^*\le d \frac{\langle K(g^*-f^*),K(g-g^*)\rangle}{\|K(g^*-f^*)\|}
    \text{ and } Lf-Lf^*\ge d \frac{\langle
    K(g^*-f^*),K(f-f^*)\rangle}{\|K(g^*-f^*)\|}.
  \end{align}
  In particular, $\hat L_{\delta,\mathcal{F},\mathcal{G}}$ achieves maximum bias
  over $\mathcal{F}$ at $f^*$ and minimum bias over $\mathcal{G}$ at $g^*$.
\end{lemma}
\begin{proof}
  Denote the ordered modulus $\omega(\delta;\mathcal{F},\mathcal{G})$ by
  $\omega(\delta)$. Suppose that the first inequality
  in~\eqref{between_class_bias_ineq} does not hold for some $g$. Then, for some
  $\varepsilon>0$,
  \begin{equation}\label{contradiction_ineq}
    Lg-Lg^*> (d+\varepsilon) \frac{\langle
    K(g^*-f^*),K(g-g^*)\rangle}{\|K(g^*-f^*)\|}.
  \end{equation}
  Let $g_\lambda=(1-\lambda) g^*+\lambda g$. Since
  $g_\lambda-g^*=\lambda (g-g^*)$, we have
  $\lambda
  L(g-g^*)=Lg_\lambda-Lf^*-L(g^*-f^*)=Lg_\lambda-Lf^*-\omega(\delta_0)$.
  Furthermore, since $g_\lambda\in\mathcal{G}$ by convexity,
  $Lg_\lambda-Lf^*\le \omega(\|K(g_\lambda-f^*)\|)$ so
  multiplying~\eqref{contradiction_ineq} by $\lambda$ gives
  \begin{equation}\label{omega_Klambdagf_bound_eq}
    \omega(\|K(g_\lambda-f^*)\|)-\omega(\delta_0)
    \ge \lambda L(g-g^*)
    >\lambda(d+\varepsilon) \frac{\langle
      K(g^*-f^*),K(g-g^*)\rangle}{\|K(g^*-f^*)\|}.
  \end{equation}
  Note that
  \begin{equation}\label{lambda_derivative_eq}
    \frac{d}{d\lambda}\|K(g_\lambda-f^*)\|\bigg|_{\lambda=0}
    =\frac{1}{2}\frac{\frac{d}{d\lambda}\|K(g_\lambda-f^*)\|^2\bigg|_{\lambda=0}}
    {\|K(g^*-f^*)\|} =\frac{\langle K(g^*-f^*), K(g-g^*)\rangle}{\|K(g^*-f^*)\|}
  \end{equation}
  so that
  $\|K(g_\lambda-f^*)\|=\delta_0+\lambda \frac{\langle K(g^*-f^*),
    K(g-g^*)\rangle}{\|K(g^*-f^*)\|}+o(\lambda)$. Combining this
  with~\eqref{omega_Klambdagf_bound_eq}, we have
  \begin{equation*}
    \omega\left(\delta_0+\lambda \frac{\langle K(g^*-f^*),
    K(g-g^*)\rangle}{\|K(g^*-f^*)\|}+o(\lambda)\right)
    -\omega(\delta_0)
    >\lambda(d+\varepsilon) \frac{\langle
    K(g^*-f^*),K(g-g^*)\rangle}{\|K(g^*-f^*)\|},
  \end{equation*}
  which is a contradiction unless $\langle K(g^*-f^*),K(g-g^*)\rangle=0$.

  If $\langle K(g^*-f^*),K(g-g^*)\rangle=0$, then~\eqref{contradiction_ineq}
  gives $Lg-Lg^*>0$, which, by the first inequality
  in~\eqref{omega_Klambdagf_bound_eq} implies
  $\omega(\|K(g_\lambda-f^*)\|)-\omega(\delta_0)\ge \lambda c$ where
  $c=Lg-Lg^*>0$. But in this case~\eqref{lambda_derivative_eq} implies
  $\|K(g_\lambda-f^*)\|=\delta_0+o(\lambda)$, again giving a contradiction. This
  proves the first inequality, and a symmetric argument applies to the
  inequality involving $Lf-Lf^*$, thereby giving the first result.

  Now consider the test statistic $\hat L_{\delta,\mathcal{F},\mathcal{G}}$.
  Under $g\in\mathcal{G}$, the bias of this statistic is equal to a constant
  that does not depend on $g$ plus
  \begin{equation*}
    d\frac{\langle K(g^*-f^*),K(g-g^*)\rangle} {\|K(g^*-f^*)\|}-(Lg-Lg^*).
  \end{equation*}
  It follows from~\eqref{between_class_bias_ineq} that this is minimized over
  $g\in\mathcal{G}$ by taking $g=g^*$. Similarly, the maximum bias over
  $\mathcal{F}$ is taken at $f^*$.
\end{proof}

The next lemma is used in the proof of
\Cref{th:two_sided_adaptation_to_function_thm}.

\begin{lemma}\label{convex_testing_lemma}
  Let $\tilde{\mathcal{F}}$ and $\tilde{\mathcal{G}}$ be convex sets, and
  suppose that $f^*$ and $g^*$ minimize $\|K(f-g)\|$ over
  $f\in\tilde{\mathcal{F}}$ and $g\in\tilde{\mathcal{G}}$. Then, for any level
  $\alpha$, the minimax test of $H_0:\tilde{\mathcal{F}}$ vs
  $H_1:\tilde{\mathcal{G}}$ is given by the Neyman-Pearson test of $f^*$ vs
  $g^*$. It rejects when $\langle K(f^*-g^*),Y\rangle$ is greater than its
  $1-\alpha$ quantile under $f^*$. The minimum power of this test over
  $\tilde{\mathcal{G}}$ is taken at $g^*$.
\end{lemma}
\begin{proof}
  The result is immediate from results stated in Section 2.4.3 in
  \citet{InSu03}, since the sets $\{Kf\colon f\in\tilde{\mathcal{F}}\}$ and
  $\{Kg\colon g\in\tilde{\mathcal{G}}\}$ are convex.
\end{proof}

\subsection{Proof of Theorem~\ref{th:one_side_adapt_thm}}%
\label{sec:ap:one_side_adapt_proof_sec}

For ease of notation in this proof, let $f^*=f^*_\delta$ and $g^*=g^*_\delta$
denote the functions that solve the modulus problem with
$\|K(f^*-g^*)\|=\delta$, and let
$d=\omega'(\delta;\mathcal{F},\mathcal{G})\in\partial
\omega(\delta;\mathcal{F},\mathcal{G})$ so that, plugging the worst-case bias
formula~\eqref{eq:maxbias-minbias} into the definition of $\hat c_\alpha$, we
have
\begin{equation*}\label{eq:hat-c-adapt-thm}
  \hat c_\alpha =\hat c_{\alpha,\delta,\mathcal{F},\mathcal{G}} =Lf^*
  +d\frac{\langle K(g^*-f^*), Y\rangle} {\|K(g^*-f^*)\|} -d\frac{\langle
  K(g^*-f^*),K f^*\rangle} {\|K(g^*-f^*)\|} -z_{1-\alpha}\sigma d.
\end{equation*}
Note that $\hat c_\alpha=\hat L_{\delta,\mathcal{F},\mathcal{G}}+a$ for $a$
chosen so that the $1-\alpha$ quantile of $\hat c_\alpha-Lf^*$ under $f^*$ is
zero. Thus, it follows from Lemma~\ref{th:max_bias_lemma} that
$\hor{\hat{c}_\alpha,\infty}$ is a valid $1-\alpha$ CI for $Lf$ over
$\mathcal{F}$, and that all quantiles of excess coverage $Lg-\hat{c}_\alpha$ are
maximized over $\mathcal{G}$ at $g^*$. In particular,
$\Rlower{\beta}(\hat{c}_\alpha;\mathcal{G})=q_{g^*,\beta}(Lg^*-\hat c_\alpha)$.
To calculate this quantile, note that, under $g^*$, $Lg^*-\hat c_\alpha$ is
normal with variance $d^2\sigma^2$ and mean
\begin{equation*}
  Lg^*-Lf^* -d\frac{\langle K(g^*-f^*),K (g^*-f^{*})\rangle} {\|K(g^*-f^*)\|}
  +z_{1-\alpha}\sigma
  d =\omega(\delta;\mathcal{F},\mathcal{G})+d(z_{1-\alpha}\sigma-\delta).
\end{equation*}
The probability that this normal variable is less than or equal to
$\omega(\delta;\mathcal{F},\mathcal{G})$ is given by the probability that a
normal variable with mean $d (z_{1-\alpha}\sigma-\delta)$ and variance
$d^2\sigma^2$ is less than or equal to zero, which is
$\Phi(\delta/\sigma-z_{1-\alpha})=\beta$. Thus,
$\Rlower{\beta}(\hat{c}_\alpha;\mathcal{G})=\omega(\delta;\mathcal{F},\mathcal{G})$
as claimed.

It remains to show that no other $1-\alpha$ CI can strictly improve on this.
Suppose that some other $1-\alpha$ CI $\hor{\tilde c,\infty}$ obtained
$\Rlower{\beta}(\tilde{c};\mathcal{G})<
\Rlower{\beta}(\hat{c}_\alpha;\mathcal{G})=\omega(\delta;\mathcal{F},\mathcal{G})$.
Then the $\beta$ quantile of excess length at $g^*$ would be strictly less than
$\omega(\delta;\mathcal{F},\mathcal{G})$, so that, for some $\eta>0$,
\begin{equation*}
  P_{g^*}(Lg^*-\tilde c\le \omega(\delta;\mathcal{F},\mathcal{G})-\eta)\ge
  \beta.
\end{equation*}
Let $\tilde f$ be given by a convex combination between $g^*$ and $f^*$ such
that $Lg^*-L\tilde f=\omega(\delta;\mathcal{F};\mathcal{G})-\eta/2$. Then the
above display gives
\begin{equation*}
  P_{g^*}(\tilde c> L\tilde f) \ge P_{g^*}(\tilde c\ge L\tilde f+\eta/2)
  =P_{g^*}(Lg^*-\tilde c\le Lg^*-L\tilde f-\eta/2) \ge \beta.
\end{equation*}
But this would imply that the test that rejects when $\tilde c>L\tilde f$ is
level $\alpha$ for $H_0:\tilde f$ and has power $\beta$ at $g^*$. This can be
seen to be impossible by calculating the power of the Neyman-Pearson test of
$\tilde f$ vs $g^*$, since $\beta$ is the power of the Neyman-Pearson test of
$f^*$ vs $g^*$, and $\tilde f$ is a strict convex combination of these
functions.

\subsection{Proof of
  Corollary~\ref{th:centrosymmetric_adaptation_corollary}}\label{centrosymmetric_adaptation_corollary_proof_sec}

Under~\eqref{eq:f-g-in-F}, if $f_{\delta,\mathcal{F},\mathcal{G}}^{*}$ and
$g^{*}_{\delta,\mathcal{F},\mathcal{G}}$ solve the modulus problem
$\omega(\delta,\mathcal{F},\mathcal{G})$, then
$f_{\delta,\mathcal{F},\mathcal{G}}^{*}-g^{*}_{\delta,\mathcal{F},\mathcal{G}}$
and $0$ (the zero function) solve $\omega(\delta;\mathcal{F},\left\{0\right\})$.
Thus,
$\omega(\delta;\mathcal{F},\mathcal{G})=\omega(\delta;\mathcal{F},\left\{0\right\})$,
and the estimators $\hat L_{\delta,\mathcal{F},\mathcal{G}}$ and
$\hat L_{\delta,\mathcal{F},\{0\}}$ and the corresponding CIs are equal up to
the choice of the element in the superdifferential. It therefore suffices to
prove the result for $\mathcal{G}=\{0\}$.

We have
\begin{equation*}
  \omega(\delta;\mathcal{F},\left\{0\right\})=
  \sup\left\{-Lf\colon
    \norm{Kf}\leq\delta,f\in\mathcal{F} \right\}=
  \frac{1}{2}\omega(2\delta;\mathcal{F}),
\end{equation*}
where the last equality obtains because under centrosymmetry, maximizing
$-Lf=L(-f)$ and maximizing $Lf$ are equivalent, so that the maximization problem
is equivalent to~\eqref{eq:centrosymmetric_modulus_eq}. Furthermore, we can take
$g^{*}_{2\delta,\mathcal{F}},f^{*}_{2\delta,\mathcal{F}}$ to satisfy
$g^{*}_{2\delta,\mathcal{F}}=-f^{*}_{2\delta,\mathcal{F}}$ with
$f^{*}_{2\delta,\mathcal{F}}$ solving the above optimization problem, so that
$g^{*}_{\delta,\mathcal{F},\{0\}}-f^{*}_{\delta,\mathcal{F},\{0\}}=-f^{*}_{\delta,\mathcal{F},\{0\}}=-f^{*}_{2\delta,\mathcal{F}}=\frac{1}{2}(g^{*}_{2\delta,\mathcal{F}}-f^{*}_{2\delta,\mathcal{F}})$.
Thus, $\hat{L}_{\delta,\mathcal{F},\{0\}}$ and $\hat{L}_{2\delta,\mathcal{F}}$
are equal up to a constant, which implies
$\hat{c}_{\alpha,\delta,\mathcal{F},\left\{0\right\}}=\hat{c}_{\alpha,2\delta,\mathcal{F}}$.
This proves the first part of the corollary. The second part of the corollary
follows by plugging $\minbias_{\{0\}}(\hat L_{\delta_\beta,\mathcal{F}})=0$ and
the formulas for $\maxbias_{\mathcal{F}}(\hat L_{\delta_\beta,\mathcal{F}})$ and
$\sd(\hat L_{\delta_\beta,\mathcal{F}})$ given in \Cref{affine_estimators_sec}
into the expression~\eqref{affine_oci_qbeta} to obtain
$\Rlower{\beta}(\hat{c}_{\alpha,\delta_\beta,\mathcal{F}},\{0\})=(\omega(\delta_{\beta};\mathcal{F})+\delta_{\beta}\omega'(\delta_{\beta};\mathcal{F}))/2$.

\subsection{Proof of \Cref{th:two_sided_adaptation_to_function_thm}}%
\label{sec:ap:two_sided_adaptation_to_function_proof_sec}
Following \citet{pratt61}, note that, for any confidence set $\mathcal{C}$ for
$\vartheta=Lf$, we have
\begin{equation*}
  E_g\lambda(\mathcal{C}) =E_g\int (1-\phi_{\mathcal{C}}(\vartheta))\, d\vartheta
  =\int E_g(1-\phi_{\mathcal{C}}(\vartheta))\, d\vartheta
\end{equation*}
by Fubini's theorem, where
$\phi_{\mathcal{C}}(\vartheta)=\1{\vartheta\notin\mathcal{C}}$. Thus, the CI
that minimizes this inverts the family of most powerful tests of
$H_0\colon Lf=\vartheta,f\in\mathcal{F}$ against $H_{1}\colon f=g$. By
\Cref{convex_testing_lemma} since the sets
$\{f\colon Lf=\vartheta,f\in\mathcal{F}\}$ and $\{g\}$ are convex, the least
favorable function $f_{\vartheta}$ minimize $\norm{K(g-f)}$ subject to
$Lf=\vartheta$, which gives the first part of the theorem.

To derive the expression for expected length, note that if $Lg\leq \vartheta$,
then the minimization problem is equivalent to solving the inverse ordered
modulus problem $\omega^{-1}(\vartheta-Lg;\left\{g\right\},\mathcal{F})$, and if
$Lg\geq \vartheta$, it is equivalent to solving
$\omega^{-1}(Lg-\vartheta;\mathcal{F},\left\{g\right\})$. This follows because
if the ordered modulus $\omega(\delta;\mathcal{F},\left\{g\right\})$ is attained
at some $f^{*}_{\delta}$ and $g$, then the inequality $\norm{K(f-g)}\leq \delta$
must be binding: otherwise a convex combination of $\tilde{f}$ and
$f^{*}_{\delta}$, where $\tilde{f}$ is such that
$L(g-f^{*}_{\delta})<L(g-\tilde{f})$ would achieve a strictly larger value, and
similarly for $\omega(\delta;\left\{g\right\},\mathcal{F})$. Such $\tilde{f}$
always exists since by the assumption that $f_{\vartheta}$ exists for all
$\vartheta$. The above argument assumes that
$\vartheta-Lg\ge \omega(0;\{g\},\mathcal{F})$ so that $\vartheta-Lg$ is in the
range of the modulus; if $0\le \vartheta-Lg\le \omega(0;\{g\},\mathcal{F})$,
then $\|K(f_\vartheta-g)\|=0$ so the minimization problem is still equivalent to
the inverse modulus if we define the inverse to be $0$ in this case (and
similarly for $0\le Lg-\vartheta\le \omega(0;\mathcal{F},\{g\})$).

Next, it follows from the proof of \Cref{th:one_side_adapt_thm} that the power
of the test $\phi_{\vartheta}$ at $g$ is given by
$\Phi(\delta_{\vartheta}/\sigma-z_{1-\alpha})$, where
$\delta_\vartheta=\|f_\vartheta-g\|$. Therefore,
\begin{equation*}
    E_{g}[\lambda(\mathcal{C}_{g}(Y))]=
    \int \Phi\left(z_{1-\alpha}-\frac{\delta_{\vartheta}}{\sigma}\right)\,\dd \vartheta=\iint \1{\delta_{\vartheta}\leq \sigma(z_{1-\alpha}-z)} \,\dd
    \vartheta\,\dd \Phi(z),
\end{equation*}
where the second equality swaps the order of integration. Splitting the inner
integral, using fact that
$\delta_{\vartheta}=\omega^{-1}(Lg-\vartheta;\mathcal{F},\left\{g\right\})$ for
$\vartheta\leq Lg$ and
$\delta_{\vartheta}=\omega^{-1}(\vartheta-Lg;\left\{g\right\},\mathcal{F})$ for
$\vartheta\geq Lg$, and taking a modulus on both sides of the inequality of the
integrand then yields
\begin{equation*}
  \begin{split}
    E_{g}[\lambda(\mathcal{C}_{g}(Y))] &=\iint_{\vartheta\leq Lg}
    \1{Lg-\vartheta\leq
      \omega\left(\sigma(z_{1-\alpha}-z);\mathcal{F},\left\{g\right\}\right)}
    \1{z\leq z_{1-\alpha}} \,\dd \vartheta
    \,\dd \Phi(z)\\
    &\qquad + \iint_{\vartheta>Lg} \1{\vartheta-Lg\leq
      \omega\left(\sigma(z_{1-\alpha}-z);\left\{g\right\},\mathcal{F}\right)}
    \1{z\leq z_{1-\alpha}} \,\dd
    \vartheta\,\dd \Phi(z)\\
    &=
    (1-\alpha)E\left[(\omega(\sigma(z_{1-\alpha}-Z);\mathcal{F},\left\{g\right\})+
      \omega(\sigma(z_{1-\alpha}-Z);\left\{g\right\},\mathcal{F}))\mid Z\leq
      z_{1-\alpha} \right],
  \end{split}
\end{equation*}
where $Z$ is standard normal, which yields the result.

\section{Extension to RD with covariates}\label{covariate_section_append}

This section discusses extensions to the RD setup when we have available a set
of covariates $z_i$ that are independent of the treatment. If the object of
interest is still the average treatment effect at $x=0$, then ignoring the
additional covariates will still lead to a valid CI\@. However, one may want to
use the information that $z_i$ is independent of treatment to gain precision. We
discuss this in \Cref{unconditional_te_section}. Alternatively, one may want to
estimate the treatment effect at $x=0$ conditional on different values of $z$,
which leads to a different approach, discussed in \Cref{conditional_te_section}.

\subsection{Using covariates to improve precision}\label{unconditional_te_section}

As argued by \citet{ccft16}, if $z_i$ is independent of treatment, the conditional mean of $z_i$ given the running variable $x_i$ should be smooth near the cutoff.  We can fit this into our setup using the model
\begin{align*}
\begin{array}{c}
  y_i=h_y(x_i)+u_i,  \\
  z_i=h_z(x_i)+v_i,
\end{array}
\quad
  \begin{pmatrix}
    u_i \\ v_i
      \end{pmatrix}
  \sim\mathcal{N}\left(0,\Sigma(x_i)\right),\,
  h_y\in\mathcal{H}_y, \, h_z\in\mathcal{H}_z,
\end{align*}
where $\mathcal{H}_y$ and $\mathcal{H}_z$ are convex smoothness classes, and we
treat $\Sigma(\cdot)$ as known. We incorporate the constraint that $z_i$ is
independent of treatment by choosing a class $\mathcal{H}_z$ such that
$\lim_{x\downarrow 0}h_z(x)-\lim_{x\uparrow 0}h_z(x)=0$ for all
$h_z\in\mathcal{H}_z$. For example, we can take
$\mathcal{H}_y=\mathcal{F}_{RDT,p}(C_y)$ and
$\mathcal{H}_z=\mathcal{F}_{RDT,p}(C_z)\cap \{h\colon \lim_{x\downarrow
  0}h_z(x)-\lim_{x\uparrow 0}h_z(x)=0\}$ for some constants $C_y$ and $C_z$.

Using our general results, one can compute optimal CIs and bounds for adaptation.  For example, our adaptation bounds show that, when $\mathcal{H}_y$ and $\mathcal{H}_z$ are centrosymmetric, there are severe limitations to adapting to the smoothness constant for either class.  Thus, CIs that take into account the covariates $z_i$ will have to depend explicitly on the smoothness constant that $h_z$ is assumed to satisfy.

In the remainder of this section, we consider a particular smoothness class, and we construct CIs that are optimal or near-optimal when $\Sigma(x)$ is constant as well as feasible versions of these CIs that are valid when $\Sigma(x)$ is unknown and may not be constant.  Given $\Sigma$, let $\Sigma_{22}$ denote the bottom-right $d_z\times d_z$ submatrix of $\Sigma$ and let $\Sigma_{21}$ denote the bottom-left $d_z\times d_1$ submatrix of $\Sigma$, where $d_z$ is the dimension of $z_i$.  Let
$\tilde y_i=y_i-z_i'\Sigma_{22}^{-1}\Sigma_{21}$ so that
\begin{align*}
  \tilde y_i=
  h_y(x_i)-h_z(y_i)'\Sigma_{22}^{-1}\Sigma_{21}+u_i-v_i'\Sigma_{22}^{-1}\Sigma_{21}
  =\tilde h_y(x_i)+\tilde u_i
\end{align*}
where $\tilde h_y(x_i)=h_y(x_i)-h_z(y_i)'\Sigma_{22}^{-1}\Sigma_{21}$ and
$\tilde u_i=u_i-v_i'\Sigma_{22}^{-1}\Sigma_{21}$. Note also that
$\lim_{x\downarrow 0}\tilde h_y(x)-\lim_{x\uparrow
  0}\tilde{h}_y(x)=\lim_{x\downarrow 0} h_y(x)-\lim_{x\uparrow 0} h_y(x)$, so
that the RD parameter for $\tilde h_y$ is the same as the RD parameter for
$h_y$. Suppose that we model the smoothness of $\tilde h_y$ directly, and take
the parameter space for $(\tilde h_y,h_z)$ to be
$\mathcal{F}_{RDT,p}(\tilde C)\times \mathcal{H}_z$. Since $\tilde u_i$ is
independent of $v_i$ and the RD parameter depends only on $\tilde h_y$, it can
be seen that minimax optimal estimators and CIs can be formed by ignoring the
$z_i$'s after this transformation is made. Thus, one can proceed as in
\Cref{sec:pract-impl} with $\tilde y_i$ in place of $y_i$.\footnote{If one
  places smoothness assumptions on $h_y$ rather than $\tilde h_y$ by taking
  $\mathcal{H}_y=\mathcal{F}_{RDT,p}(C_y)$ and
  $\mathcal{H}_z=\mathcal{F}_{RDT,p}(C_z)\cap \{h\colon \lim_{x\downarrow
    0}h_z(x)-\lim_{x\uparrow 0}h_z(x)=0\}$, then
  $\tilde h_y\in \mathcal{F}_{RDT,p}(C_y+C_z\iota'\Sigma_{22}^{-1}\Sigma_{21})$
  where $\iota$ is a vector of ones. It follows that the CIs discussed here will
  be valid for $\tilde C\ge C_y+C_z\iota'\Sigma_{22}^{-1}\Sigma_{21}$. However,
  the resulting parameter space for $(\tilde h_y,h_z)$ will be different (in
  particular, it will not take the form $\mathcal{H}_y\times \mathcal{H}_z$), so
  that optimal estimators will be different for this class.}

To make this procedure feasible, we need an estimate of
$\Sigma_{22}^{-1}\Sigma_{21}$. We propose the estimates
$\hat \Sigma_{22}=\frac{1}{nh}\sum_{i=1}^n \hat v_i\hat v_i' k(x_i/h)$ and
$\hat \Sigma_{21}=\frac{1}{nh}\sum_{i=1}^n \hat v_i y_i k(x_i/h)$ where
$\hat v_i$ is the residual from the local polynomial regression of $z_i$ on a
$p$th order polynomial of $x_i$ and its interaction with $\1{x_i>0}$, with
weight $k(x_i/h)$. To form CIs, one proceeds as in \Cref{sec:pract-impl} with
$\tilde y_i=y_i-z_i'\hat \Sigma_{22}^{-1}\hat \Sigma_{21}$ in place of $y_i$ and
$\tilde C$ playing the role of $C$. A simple calculation shows that, if one uses
the local polynomial weights~\eqref{eq:lp-weights}, with the same kernel and
bandwidth used to estimate $\Sigma$, the resulting CIs will be centered at a
local polynomial estimate where $z_i$ is included as a regressor in the local
polynomial regression. This corresponds exactly to an estimator proposed by
\citet{ccft16}. Thus, our relative efficiency results can be used to show that
this estimator is close to optimal under these assumptions.

\subsection{Estimating the treatment effect conditional on $z_i=z$}\label{conditional_te_section}

If one is interested in how the treatment effect at $x=0$ varies with $z$, one
can use the model $y_i=f(x_i,z_i)+u_i$ where $f$ is placed in a smoothness class
and the object of interest is
$L_{z}f =\lim_{x\downarrow 0}f(x,z)-\lim_{x\uparrow 0}f(x,z)$ for different
values of $z$. This fits into our general framework once one fixes the point $z$
at which $L_{z}f$ is evaluated, and one can use our results to obtain CIs for
different values of $z$. A natural smoothness class is to place a bound on the
$p$th order multivariate Taylor approximation of $f(x,z)\1{x>0}$ and
$f(x,z)\1{x<0}$ at $x=0$ and $z$ equal to the value of interest. The analysis of
optimal and near optimal estimators then follows from a generalization of the
results described in \Cref{sec:pract-impl}. In particular, one can use
multivariate local polynomial estimators (with worst-case bias computed using a
generalization of the calculations in \Cref{rd_bias_sec}), or optimal weights
can be computed by generalizing the calculations in
\Cref{rd_modulus_solution_sec}.

Estimating the treatment effect conditional on different values of $z$ can be a
useful way of exploring treatment effect heterogeneity. However, unless one
places some additional parametric structure on $f(x,z)$, the resulting estimates
will suffer from imprecision when the dimension of $z$ is moderate due to the
curse of dimensionality.

\end{appendices}

\newpage
\bibliography{np-testing-library}

\begin{table}[p]
  \centering
  \begin{tabular}{@{}llll@{}}
    & \multicolumn{3}{c}{$\alpha$} \\
    \cmidrule(rl){2-4}
    $b$ &    0.01&    0.05&     0.1\\
    \midrule
    0.0& 2.576& 1.960& 1.645\\
    0.1& 2.589& 1.970& 1.653\\
    0.2& 2.626& 1.999& 1.677\\
    0.3& 2.683& 2.045& 1.717\\
    0.4& 2.757& 2.107& 1.772\\
    0.5& 2.842& 2.181& 1.839\\
    0.6& 2.934& 2.265& 1.916\\
    0.7& 3.030& 2.356& 2.001\\
    0.8& 3.128& 2.450& 2.093\\
    0.9& 3.227& 2.548& 2.187\\
    1.0& 3.327& 2.646& 2.284\\
    1.5& 3.826& 3.145& 2.782\\
    2.0& 4.326& 3.645& 3.282
  \end{tabular}
  \caption{Critical values $\cv_{\alpha}(b)$ for selected confidence levels and
    values of maximum absolute bias $b$. For $b\ge 2$,
    $\cv_{\alpha}(b)\approx b+z_{1-\alpha}$ up to 3 decimal places for these
    values of $\alpha$.}\label{tab:cv-b}
\end{table}

\clearpage

\begin{figure}[p]
  \centering \input{lipschitz-rd-small.tex}
  \caption{The least favorable null and alternative functions $f^{*}$ and
    $g^{*}$ from Equation~\eqref{eq:fg-star} in \Cref{optimal_cis_subsec}.
  }\label{fig:lipschitz-lf}
\end{figure}
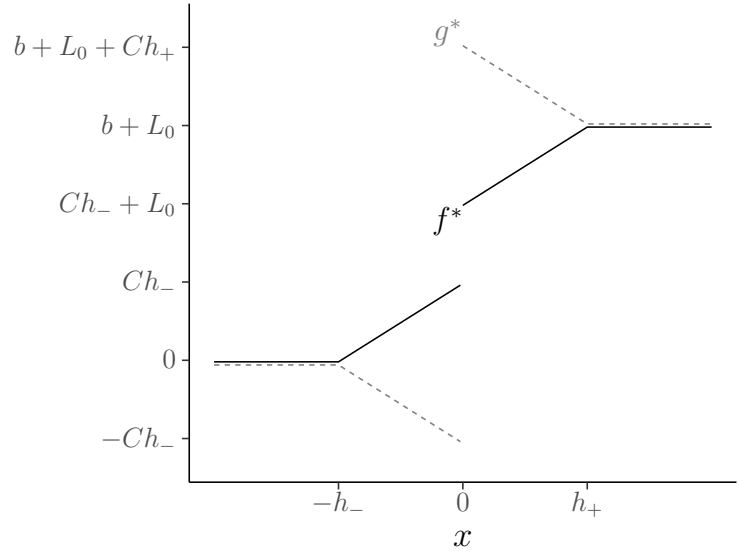

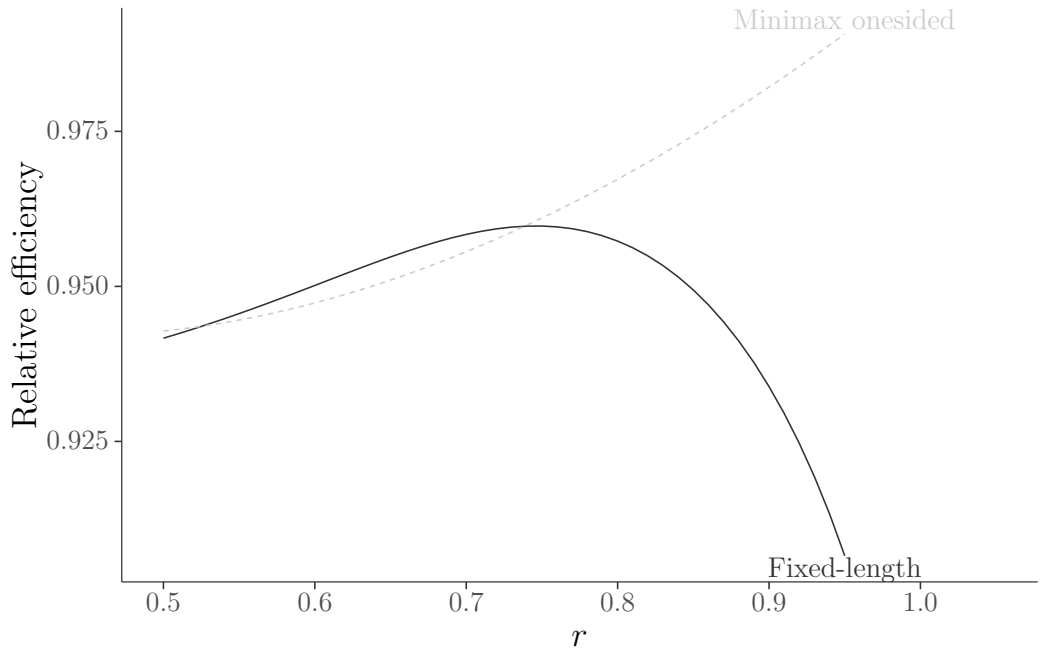
\begin{figure}[p]
  \centering \input{adaptivity-bounds-bw.tex}
  \caption{Asymptotic efficiency bounds for one-sided and fixed-length CIs as
    function of the optimal rate of convergence $r$ under centrosymmetry.
    Minimax one-sided refers to ratio of $\beta$-quantile of excess length of
    CIs that direct power at smooth functions relative to minimax one-sided CIs
    given in~\eqref{eq:minimax-onesidedCI-efficiency}. Shortest fixed-length
    refers the ratio of expected length of CIs that direct power at a given
    smooth function relative to shortest fixed-length affine CIs given
    in~Theorem~\ref{rd_optimal_estimator_thm_main_text}.}\label{fig:asymptotic-adaptivity-bounds}
\end{figure}

\begin{figure}[p]
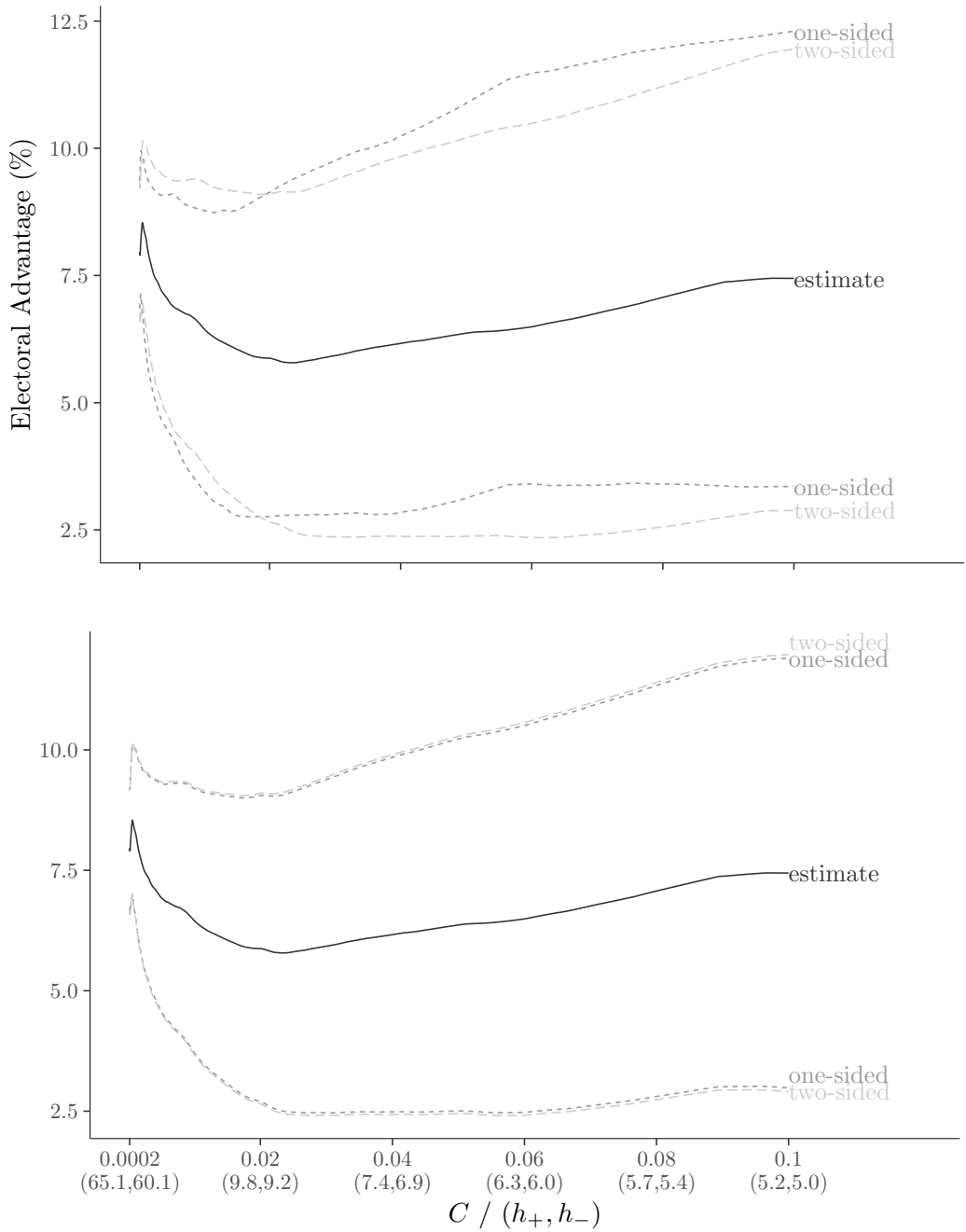

  \centering%
  {\footnotesize%
    \input{lee-opt-bw.tex}
    \input{lee-opt-estimation-bw.tex}}
  \caption{\citet{lee08} RD example. Top panel displays minimax MSE estimator
    (estimator), and lower and upper limits of minimax one-sided confidence
    intervals for 0.8 quantile (one-sided), and fixed-length CIs (two-sided) as
    function of smoothness $C$. Bottom panel displays one-and two-sided CIs
    around the minimax MSE estimator. $h_{+},h_{-}$ correspond to the optimal
    smoothness parameters for the minimax MSE estimator.}\label{fig:lee-minimax}
\end{figure}

\begin{figure}[p]
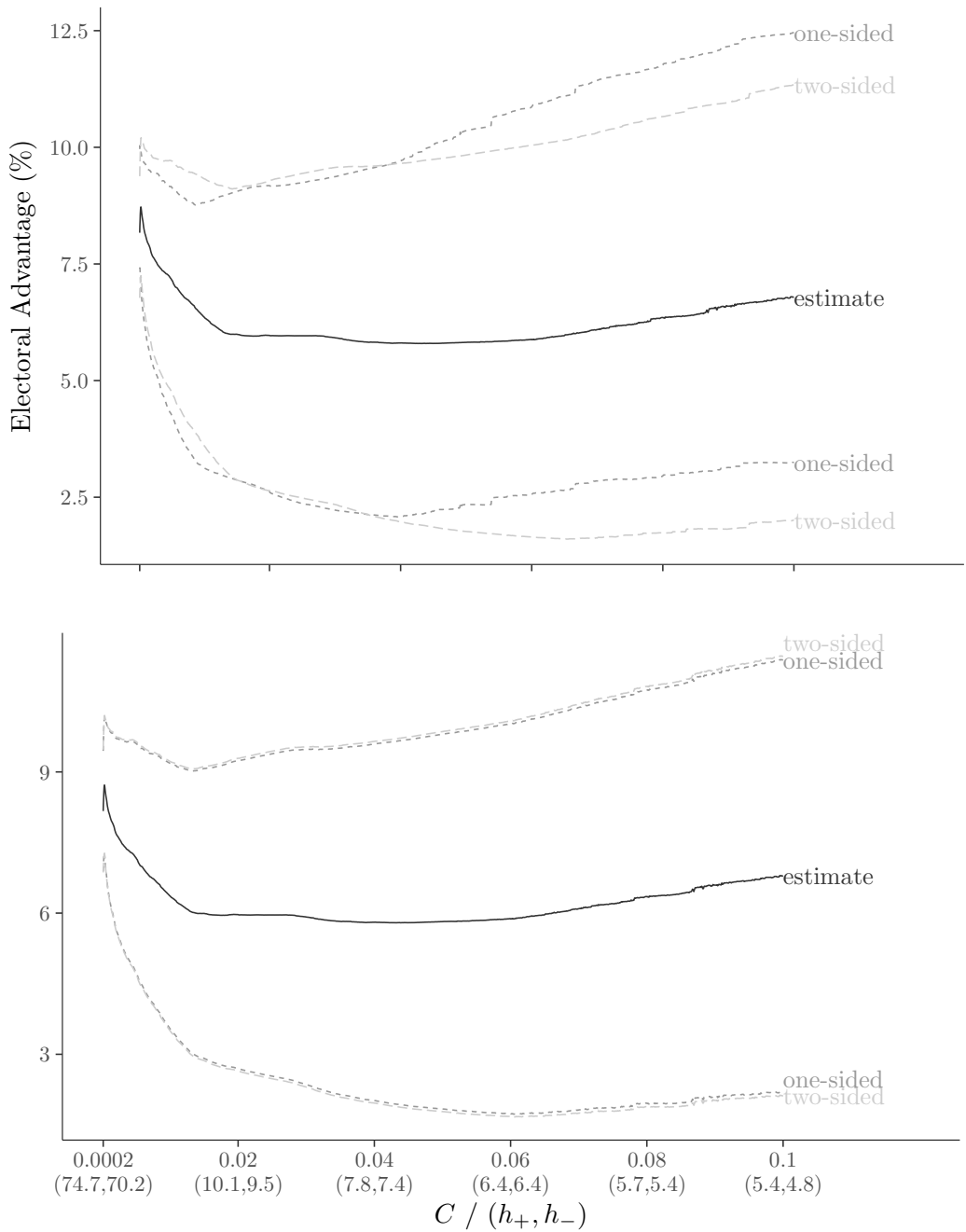

  \centering%
  \footnotesize
  \input{lee-ll-bw.tex}
  \input{lee-ll-estimation-bw.tex}
  \caption{\citet{lee08} RD example: local linear regression with triangular
    kernel. Top panel displays estimator based on minimax MSE bandwidths
    (estimator), lower and upper limits of one-sided CIs with bandwidths that
    are minimax for 0.8 quantile of excess length (one-sided), and shortest
    fixed-length CIs (two-sided) as function of smoothness $C$. Bottom panel
    displays one-and two-sided CIs around and estimator based on minimax MSE
    bandwidths. $h_{+},h_{-}$ correspond to the minimax MSE
    bandwidths.}\label{fig:lee-ll}
\end{figure}

\end{document}


\maketitle

This supplement provides appendices not included in the main text.
\Cref{sec:comp-with-other} compares our approach with other methods, and includes a Monte Carlo study.
\Cref{sec:ap:details} contains details for the results in
\Cref{sec:general_results} not included in the main text.
\Cref{sec:addit-deta-rd} contains details for the RD application.
\Cref{sec:sap:unknown_error} considers feasible versions of the procedures in
\Cref{sec:general_results} in the case with unknown error distribution
and derives their asymptotic efficiency. \Cref{sec:sap:asym_efficiency_bounds}
gives some auxiliary results used for relative asymptotic efficiency
comparisons. \Cref{sec:sap:rd_asym} gives the proof of
Theorem~\ref{rd_optimal_estimator_thm_main_text}.

\begin{appendices}
\crefalias{section}{sappsec}
\crefalias{subsection}{sappsubsec}
\addtocounter{section}{2}

  \section{Comparison with other methods}\label{sec:comp-with-other}

  This section compares the CIs developed in this paper to other approaches to
  inference in the RD application. We consider two popular approaches. The first
  approach is to form a nominal $100\cdot (1-\alpha)\%$ CI by adding and
  subtracting the $1-\alpha/2$ quantile of the $\mathcal{N}(0,1)$ distribution
  times the standard error, thereby ignoring any bias. We refer to these CIs as
  ``conventional.'' The second approach is the robust bias correction (RBC)
  method studied by \citet{cct14}, which subtracts an estimate of the bias, and
  then takes into account the estimation error in this bias correction in
  forming the interval.

  The coverage of these CIs will depend on the smoothness class $\mathcal{F}$ as
  well as the choice of bandwidth. Since CIs reported in applied work are
  typically based on local linear estimators, with relative efficiency results
  for minimax MSE in the class $\mathcal{F}_{T,2}(C,\mathbb{R}_+)$ for
  estimation of $f(0)$ due to \citet{cfm97} often cited as justification, we
  focus on the class $\mathcal{F}_{RDT,2}(C)$ when computing coverage (in
  \Cref{sec:monte-carlo-evidence}, we consider classes that also impose bounds
  on smoothness away from the discontinuity point rather than just placing
  bounds on the error of the Taylor approximation around the discontinuity
  point). If the bandwidth choice is non-random, then finite sample coverage can
  be computed exactly when errors are normal with known variance.\footnote{The
    resulting coverage calculations hold in an asymptotic sense with unknown
    error distribution in the same way that, for example, coverage calculations
    in \citet{StockYogo2005} are valid in an asymptotic sense in the
    instrumental variables setting.} We take this approach in
  \Cref{sec:exact-coverage}. If a data-driven bandwidth is used, computing
  finite sample coverage exactly becomes computationally prohibitive. We examine
  the coverage and relative efficiency of CIs with data driven bandwidths in a
  Monte Carlo study in \Cref{sec:monte-carlo-evidence}.

  \subsection{Exact coverage with nonrandom bandwidth}\label{sec:exact-coverage}

  For a given CI, we examine coverage in the classes $\mathcal{F}_{RDT,2}(C)$ by
  asking ``what is the largest value of $C$ for which this CI has good
  coverage?'' Since the conventional CI ignores bias, there will always be some
  undercoverage, so we formalize this by finding the largest value of $C$ such
  that a nominal $95\%$ CI has true coverage $90\%$. This calculation is easily
  done using the formulas in \Cref{affine_estimators_sec}: the
  conventional approach uses the critical value $z_{0.975}=\cv_{0.05}(0)$ to
  construct a nominal $95\%$ CI, while a valid $90\%$ CI uses
  $\cv_{0.1}(\maxbias_{\mathcal{F}_{RDT,2}(C)}(\hat{L})/\text{se}(\hat{L}))$
  (where $\hat L$ denotes the estimator and $\text{se}(\hat{L})$ denotes its
  standard error), so we equate these two critical values and solve for $C$.

  The resulting value of $C$ for which undercoverage is controlled will depend
  on the bandwidth. To provide a simple numerical comparison to commonly used
  procedures, we consider the (data-dependent) \citet[IK]{ik12restud} bandwidth
  $\hat{h}_{IK}$ in the context of the \citeauthor{lee08} application considered
  in \Cref{sec:appl-regr-disc}, but treat it as if it were fixed a
  priori. The IK bandwidth selector leads to $\hat{h}_{IK}=29.4$ for local
  linear regression with the triangular kernel. The conventional two-sided CI
  based on this bandwidth is given by $7.99\pm 1.71$. Treating the bandwidth as
  nonrandom, it achieves coverage of at least 90\% over $\mathcal{F}_{RDT,2}(C)$
  as long as $C\leq C_{\text{conv}}=0.0018$. This is a rather low value, lower
  than the lower bound estimate on $C$ from
  \Cref{sec:estim-lower-bound}. It implies that even when $x=20\%$, the
  prediction error based on a linear Taylor approximation to $f$ can be reduced
  by less than 1\% by using the true conditional expectation.

  As an alternative to the conventional approach, one can use the robust-bias
  correction method studied in \citet{cct14}. \citet{cct14} show that if the
  pilot bandwidth and the kernel used by the bias estimator equal those used by
  the local linear estimator of $Lf$, this method is equivalent to running a
  quadratic instead of a linear local regression, and then using the usual CI\@.
  In the \citeauthor{lee08} application with IK bandwidth, this delivers the CI
  $6.68\pm 2.52$, increasing the half-length substantially relative to the
  conventional CI\@. The maximum smoothness parameter under which these CIs have
  coverage at least 90\% is given by $C_{RBC}=0.0023>C_{\text{conv}}$. By way of
  comparison, the optimal 95\% fixed-length CIs at $C_{RBC}$ leads to a much
  narrower CI given by $7.70\pm 2.11$.

  While the CCT CI maintains good coverage for a larger smoothness constant than
  the conventional CI, both constants are rather small (equivalently, coverage
  is bad for moderate values of $C$). This is an artifact of the large realized
  value of $\hat h_{IK}$: the CCT CI essentially ``undersmooths'' relative to a
  given bandwidth by making the bias-standard deviation ratio smaller. Since
  $\hat h_{IK}$ is large to begin with, the amount of undersmoothing is not
  enough to make the procedure robust to moderate values of $C$. In fact, the IK
  bandwidth is generally quite sensitive to tuning parameter choices: we show in
  a Monte Carlo study in \Cref{sec:monte-carlo-evidence} that the CCT
  implementation of the IK bandwidth yields smaller bandwidths and achieves good
  coverage over a much larger set of functions, at the cost of larger length. In
  finite samples, the tuning parameters drive the maximum bias of the estimator,
  and hence its coverage properties, even though under standard pointwise
  asymptotics, the tuning parameters shouldn't affect coverage.

  In contrast, if one performs the CCT procedure starting from a minimax MSE
  optimal bandwidth based on a known smoothness constant $C$, the asymptotic
  coverage will be quite good (above $94\%$), although the CCT CI ends up being
  about $30\%$ longer than the optimal CI \citep[see][]{ArKo15}. Thus, while
  using a data driven bandwidth selector such as IK for inference can lead to
  severe undercoverage for smoothness classes used in RD (even if one
  undersmooths or bias-corrects as in CCT), procedures such as RBC can have good
  coverage if based on an appropriate bandwidth choice that is fixed ex ante.

  \subsection{Monte Carlo evidence with random
    bandwidth}\label{sec:monte-carlo-evidence}

  Corollaries~\ref{th:centrosymmetric_adaptation_corollary}
  and~\ref{th:centrosymmetric_adaptation_twosided} imply that confidence
  intervals based on data-driven bandwidths must either undercover or else
  cannot be shorter than fixed-length CIs that assume worst-case smoothness. We
  now illustrate this implication with a Monte Carlo study.

  We consider the RD setup from \Cref{sec:simple-example}. To help
  separate the difficulty in constructing CIs for $Lf$ due to unknown smoothness
  of $f$ from that due to irregular design points or heteroskedasticity, for all
  designs below, the distribution of $x_{i}$ is uniform on $[-1,1]$, and $u_{i}$
  is independent of $x_{i}$, distributed $\mathcal{N}(0,\sigma^{2})$. The sample
  size is $n=500$ in each case.

  For $\sigma^{2}$, we consider two values, $\sigma^{2}=0.1295$, and
  $\sigma^{2}=4\times 0.1295=0.518$. We consider conditional mean functions $f$
  that lie in the smoothness class
  \begin{equation*}
    \mathcal{F}_{RDH,2}(C)= \left \{f_+-f_-\colon
    f_+\in\mathcal{F}_{H,2}(C;\mathbb{R}_+),\;
    f_{-}\in\mathcal{F}_{H,2}(C;\mathbb{R}_{-}) \right\},
  \end{equation*}
  where $\mathcal{F}_{H,p}(C;\mathcal{X})$ is the second-order Hölder class, the
  closure of twice-differentiable functions with second derivative bounded by
  $2C$, uniformly over $\mathcal{X}$:
  \begin{equation*}
    \mathcal{F}_{H,p}(C;\mathcal{X}) =\left\{f\colon \left|
    f'(x_{1})-f'(x_{2})\right|\le 2C|x_{1}-x_{2}|\text{ all
    }x_{1},x_{2}\in\mathcal{X}\right\}.
  \end{equation*}
  Unlike the class $\mathcal{F}_{RDT,2}(C)$, the class $\mathcal{F}_{RDH,2}(C)$ also
  imposes smoothness away from the cutoff, so that
  $\mathcal{F}_{RDH,2}(C)\subseteq \mathcal{F}_{RDT,2}(C)$. Imposing smoothness
  away from the cutoff is natural in many empirical applications. We consider
  $C=1$ and $C=3$, and for each $C$, we consider 4 different shapes for $f$. In
  each case, $f$ is odd, $f_{+}=-f_{-}$. In Designs 1 through 3, $f_{+}$ is
  given by a quadratic spline with two knots, at $b_{1}$ and $b_{2}$,
  \begin{equation*}
    f_{+}(x)=\1{x\geq 0}\cdot C\left(
      x^{2}-2(x-b_{1})_{+}^{2}+2(x-b_{2})_{+}^{2}\right).
  \end{equation*}
  In Design 1 the knots are given by $(b_{1},b_{2})=(0.45,0.75)$, in Design 2 by
  $(0.25,0.65)$, and in Design 3 by $(0.4,0.9)$. The function $f_{+}(x)$ is
  plotted in Figure~\ref{fig:reg-fkt} for $C=1$. For $C=3$, the function $f$ is
  identical up to scale. It is clear from the figure that although locally to
  the cutoff, the functions are identical, they differ away from the cutoff (for
  $\abs{x}\geq 0.25$), which, as we demonstrate below, affects the performance
  of data-driven methods. Finally, in Design 4, we consider $f(x)=0$ to allow us
  to compare the performance of CIs when $f$ is as smooth as possible.

  We consider four methods for constructing CIs based on data-driven bandwidths,
  and two fixed-length CIs. All CIs are based on local polynomial regressions
  with a triangular kernel. The variance estimators used to construct the CIs
  are based on the nearest-neighbor method described in
  Remark~\ref{remark:nn-variance}. The results based on Eicker-Huber-White
  variance estimators are very similar and not reported here.

  The first two methods correspond to conventional CIs based on local linear
  regression described in \Cref{sec:exact-coverage}. The first CI uses
  \citet[IK]{ik12restud} bandwidth selector $\hat{h}_{IK}$, and the second CI
  uses a bandwidth selector proposed in \citet[CCT]{cct14}, $\hat{h}_{CCT}$. The
  third CI uses the robust bias correction (RBC) studied in CCT, with both the
  pilot and the main bandwidth given by $\hat{h}_{IK}$ (the main estimate is
  based on local linear regression, and the bias correction is based on local
  quadratic regression), so that the bandwidth ratio is given by $\rho=1$. The
  fourth CI is also based on RBC, but with the main and pilot bandwidth
  potentially different and given by the \citet{cct14} bandwidth selectors.
  Finally, we consider two fixed-length CIs with uniform coverage under the
  class $\mathcal{F}_{RDH,2}(C)$, with $C=1,3$, and bandwidth chosen to minimize
  their half-length. Their construction is similar to the CIs considered in
  \Cref{sec:pract-impl}, except they use the fact that under
  $\mathcal{F}_{RDH,2}(C)$, the maximum bias for local linear estimators based
  on a fixed bandwidth is attained at $g^{*}(x)=Cx^{2}\1{x\geq 0}-Cx^{2}\1{x<0}$
  (see \citealp{ArKo15}, for derivation).

  The results are reported in Table~\ref{tab:mc2-C1} for $C=1$
  and~\ref{tab:mc2-C3} for $C=3$. One can see from the tables that CIs based on
  $\hat{h}_{IK}$ may undercover severely even at the higher level of smoothness,
  $C=1$. In particular, the coverage of conventional CIs based on $\hat{h}_{IK}$
  is as low as 10.1\% for 95\% nominal CIs in Design 1, and the coverage of RBC
  CIs is as low as 64.4\%, again in Design 1. The undercoverage is even more
  severe when $C=3$.

  In contrast, CIs based on the CCT bandwidth selector perform much better in
  terms of coverage under $C=1$, with coverage over 90\% for all designs. These
  CIs only start undercovering once $C=3$, with 80.7\% coverage in Design 3 for
  conventional CIs, and 86.2\% coverage for RBC CIs. The cost for the good coverage
  properties, as can be seen from the tables, is that the CIs are longer,
  sometimes much longer than optimal fixed-length CIs.

  As discussed in \Cref{sec:exact-coverage}, the dramatically different coverage
  properties of the CIs based on the IK and CCT bandwidths illustrates the point
  that the coverage of CIs based on data-driven bandwidths is governed by the
  tuning parameters used in defining the bandwidth selector. These results can
  also be interpreted as showing the limits of procedures that try to ``estimate
  $C$'' from the data. In particular, we show in \citet{ArKo15} that for
  inference at a point based on local linear regression under the second-order
  Hölder class, in large samples the MSE-optimal bandwidth (see
  Remark~\ref{MSE_bandwidth_remark}) differs from the usual (infeasible)
  bandwidth minimizing the large-sample MSE under pointwise asymptotics only in
  that it replaces $f''(0)$ with $C$. Thus, plug-in rules that estimate the
  infeasible pointwise bandwidth by plugging in an estimate of $f''(0)$ can be
  interpreted as data-driven bandwidths that try to estimate $C$ from the data.
  Since the IK and CCT bandwidths are plug-in rules, to the extent that one can
  interpret them as trying to ``estimate $C$'' from the data, these simulation
  results also illustrate the point that attempts to estimate $C$ from the data
  cannot improve upon FLCIs (one can show that if these procedures were
  successful at estimating $C$, conventional CIs with 95\% nominal level based
  on them should have coverage no less than 92.1\% in large samples).

  To assess sensitivity of these results to the normality and homoskedasticity
  of the errors, we also considered Designs 1--4 with heteroskedastic and
  log-normal errors. The results (not reported here) are similar in the sense
  that if a particular method achieved close to 95\% coverage under normal
  homoskedastic errors, the coverage remained good under alternative error
  distributions. If a particular method undercovered in a given design, the
  amount of undercoverage could be more or less severe, depending on the form of
  heteroskedasticity. In particular, fixed-length CIs with $C=3$ achieve
  excellent coverage for all designs and all error distributions considered.

\section{Additional details for \Cref{sec:general_results}}\label{sec:ap:details}

This section contains details for the results in \Cref{sec:general_results} not
included in the main text.

\subsection{Special cases}%
\label{sec:ap:special_cases}

In addition to regression discontinuity, the regression
model~\eqref{eq:fixed_design_eq} covers several other important models,
including inference at a point ($Lf=f(x_0)$ with $x_0$ given) and average
treatment effects under unconfoundedness (with
$Lf=\frac{1}{n}\sum_{i=1}^{n}(f(w_{i},1)-f(w_{i},0))$ where
$x_{i}=(w_{i}', d_{i})'$, $d_{i}$ is a treatment indicator and $w_{i}$ are
controls).

The setup~\eqref{eq:donoho-model} can also be used to study the linear
regression model with restricted parameter space. For simplicity, consider
the case with homoskedastic errors,
\begin{equation}\label{eq:linear_reg_model}
  Y=X\theta+\sigma\varepsilon,
  \quad \varepsilon\sim\mathcal{N}(0,I_n),
\end{equation}
where $X$ is a fixed $n\times k$ design matrix and $\sigma$ is known. This fits
into our framework with $f=\theta$, $X$ playing the role of $K$, taking
$\theta\in\mathbb{R}^k$ to $X\theta\in\mathbb{R}^n$, and
$\mathcal{Y}=\mathbb{R}^n$ with the Euclidean inner product
$\langle x,y\rangle=x'y$. We are interested in a linear functional
$L\theta=\ell'\theta$ where $\ell\in \mathbb{R}^k$. We consider this model in
previous version of this paper \citep{ArKo15optimal}.
Furthermore,~\eqref{eq:donoho-model} covers the multivariate normal location
model $\hat\theta\sim\mathcal{N}(\theta,\Sigma)$, which obtains as a
limiting experiment of regular parametric models. Our finite-sample results
could thus be extended to local asymptotic results in regular parametric models
with restricted parameter spaces.

In addition to the regression models~\eqref{eq:fixed_design_eq}
and~\eqref{eq:linear_reg_model}, the setup~\eqref{eq:donoho-model} includes
other nonparametric and semiparametric regression models such as the partly
linear model (where $f$ takes the form $g(w_1)+\gamma'w_2$, and we are
interested in a linear functional of $g$ or $\gamma$). It also includes the
Gaussian white noise model, which can be obtained as a limiting model for
nonparametric density estimation \citep[see][]{nussbaum_asymptotic_1996} as well
as nonparametric regression with fixed or random regressors
\citep[see][]{BrLo96,reiss08}. These white noise equivalence results imply that
our finite-sample results translate to asymptotic results in problems such as
inference at a point in density estimation or regression with random regressors.
We refer the reader to \citet[Section 9]{donoho94} for details of these and
other models that fit into the general setup~\eqref{eq:donoho-model}.

\subsection{Derivative of the modulus}\label{translation_invariance_sec}

The class of optimal estimators $\hat L_{\delta,\mathcal{F},\mathcal{G}}$
involves the superdifferential of the modulus. In the case where the modulus is
differentiable, the superdifferential is a singleton, so that
$\hat L_{\delta,\mathcal{F},\mathcal{G}}$ is defined uniquely. In this section,
we introduce a condition that guarantees differentiability and leads to a
formula for the derivative. We also briefly discuss the case where the modulus
is not differentiable.
\begin{definition}[Translation Invariance]
  The function class $\mathcal{F}$ is translation invariant if there exists a
  function $\iota\in\mathcal{F}$ such that $L\iota=1$ and
  $f+c\iota\in\mathcal{F}$ for all $c\in\mathbb{R}$ and $f\in\mathcal{F}$.
\end{definition}
Translation invariance will hold in most cases where the parameter of interest
$Lf$ is unrestricted. For example, if $Lf=f(0)$, it will hold with $\iota(x)=1$
if $\mathcal{F}$ places monotonicity restrictions and/or restrictions on the
derivatives of $f$. Under translation invariance, the modulus is differentiable,
and we obtain an explicit expression for its derivative:
\begin{lemma}\label{translation_invariant_derivative_lemma}
  Let $f^*$ and $g^*$ solve the modulus problem with
  $\delta_0=\|K(g^*-f^*)\|>0$, and suppose that $f^*+c \iota\in\mathcal{F}$ for
  all $c$ in a neighborhood of zero, where $L\iota=1$. Then the modulus is
  differentiable at $\delta_{0}$ with
  $\omega'(\delta_{0};\mathcal{F},\mathcal{G})=\delta_{0}/ \langle K\iota,K
  (g^{*}_{\delta_{0}}-f^{*}_{\delta_{0}})\rangle$.
\end{lemma}
\begin{proof}
  Let $d\in \partial\omega(\delta_0;\mathcal{F},\mathcal{G})$ and let
  $f_c=f^*-c\iota$. Let $\eta$ be small enough so that $f_c\in\mathcal{F}$ for
  $|c|\le \eta$. Then, for $|c|\le \eta$,
  \begin{equation*}
    L(g^*-f^*)
    +d\left[\|K(g^*-f_c)\|-\delta_0\right]
    \ge \omega(\|K(g^*-f_c)\|;\mathcal{F},\mathcal{G})
    \ge L(g^*-f_c)=L(g^*-f^*)+c
  \end{equation*}
  where the first inequality follows from the definition of the
  superdifferential and the second inequality follows from the definition of the
  modulus. Since the left-hand side of the above display is greater than or
  equal to the right-hand side for $\abs{c}\leq \eta$,
  and the two sides are equal at $c=0$, the derivatives of both sides with
  respect to $c$ must be equal. Since
  \begin{equation*}
    \frac{d\|K(g^*-f_c)\|}{d c}\bigg|_{c=0} =\frac{\frac{d}{d
        c}\|K(g^*-f_c)\|^2\big|_{c=0}}{2\delta_0} =\frac{\langle
      K(g^*-f^*),K\iota\rangle}{\delta_0},
  \end{equation*}
  result follows.
\end{proof}
The explicit expression for $\omega'(\delta;\mathcal{F},\mathcal{G})$ is useful
in simplifying the expressions~\eqref{eq:Ldelta_eq}
and~\eqref{eq:centrosymmetric_modulus_eq} for the optimal estimators.

Translation invariance leads to a direct relation between optimal CIs and tests.
In general, it can be seen from Lemma~\ref{convex_testing_lemma} that the test
that rejects $L_0$ when
$L_0\notin \hor{\hat{c}_{\alpha,\delta,\mathcal{F},\mathcal{G}},\infty}$ is
minimax for $H_0:Lf\le L_0$ and $f\in\mathcal{F}$ against
$H_1:Lf\ge L_0+\omega(\delta;\mathcal{F},\mathcal{G})$ and $f\in\mathcal{G}$,
where $L_0=Lf_{\delta}^*$. If both $\mathcal{F}$ and $\mathcal{G}$ are
translation invariant, $f_{\delta}^*+c\iota$ and $g_{\delta}^{*}+c\iota$ achieve
the ordered modulus for any $c\in\mathbb{R}$, so that, varying $c$, this test
can be seen to be minimax for any $L_0$. Thus, under translation invariance, the
CI in Theorem~\ref{th:one_sided_minimax_thm} inverts minimax one sided tests
with distance to the null given by $\omega(\delta)$ (in general, the test based
on the CI in Theorem~\ref{th:one_sided_minimax_thm} is minimax only when
$L_0=Lf_{\delta}^*$).

If the modulus is not differentiable at some $\delta$, the CIs defined in
\Cref{sec:onesided-cis,sec:two-sided-cis_main} are valid with
$\omega'(\delta,\mathcal{F},\mathcal{G})$ given by any element of the
superdifferential, so long as the same element of the superdifferential is used
throughout the formula (in particular, the same element used in the
estimator~\eqref{eq:Ldelta_eq} must be used in the worst-case bias
formula~\eqref{eq:maxbias-minbias}). For the one-sided CI,
Theorem~\ref{th:one_side_adapt_thm} applies regardless of which element of the
superdifferential is used. In the two-sided case, when computing the optimal
fixed-length affine CI described in \Cref{sec:two-sided-cis_main}, the only
additional detail in the case where the modulus is not everywhere differentiable
is that one optimizes the half-length over both $\delta$ and over elements in
the superdifferential.

\section{Additional details for RD}\label{sec:addit-deta-rd}

This section gives additional details for the RD application. \Cref{rd_bias_sec}
derives the worst-case bias formula given in~\eqref{eq:rd-worst-case-bias}.
\Cref{rd_modulus_solution_sec} derives the optimal estimator and the solution to
the modulus problem. \Cref{sec:estim-lower-bound} discusses lower bounds for the
smoothness constant $C$. \Cref{sec:asympt-valid-optim} shows the asymptotic
validity of the feasible version of the estimator in which the variance is
estimated.

\subsection{Worst-case bias for linear estimators}\label{rd_bias_sec}

This section derives the worst-case bias formula~\eqref{eq:rd-worst-case-bias}
for linear estimators $\hat L_{h_+,h_-}$ defined
in~\eqref{eq:rd-linear-estimator} in \Cref{sec:pract-impl}. We require the
weights to satisfy $w_{+}(-x,h_{+})=w_{-}(x,h_{-})=0$ for $x\geq 0$ and
\begin{equation}\label{w_unbiased_condition}
  \begin{aligned}
  \sum_{i=1}^{n}w_+(x_i,h_{+})&=\sum_{i=1}^{n}w_{-}(x_i,h_{-})=1,  \\
  \sum_{i=1}^{n}x_{i}^{j}w_{-}(x_i,h_{-})&=\sum_{i=1}^{n}x_{i}^{j}w_{+}(x_i,h_{+})=0
                                           \text{ for }j=1,\dotsc,p-1.
  \end{aligned}
\end{equation}
Note that~\eqref{w_unbiased_condition} holds iff. $\hat L_{h_+,h_-}$ is unbiased
at all $f=f_++f_-$ where $f_+$ and $f_-$ are both polynomials of order $p-1$ or
less, which is necessary to ensure that the worst-case bias is finite. This
condition holds if $\hat{L}_{h_{+},h_{-}}$ is based on a local polynomial
estimator of order at least $p-1$.

We can write any function $f\in\mathcal{F}_{RDT,p}(C)$ as $f=f_++f_-$, where
\begin{align*}
f_+(x)&=[\sum_{j=0}^{p-1}f_+^{(j)}(0)x^j/j!+r_+(x)]\1{x\geq 0},&
f_{-}(x)&=[\sum_{j=0}^{p-1}f_-^{(j)}(0)x^j/j!+r_-(x)]\1{x<0},
\end{align*}
and the remainder terms $r_{+}$ and $r_{-}$ satisfy $|r_+(x)|\le C|x|^p$ and
$|r_-(x)|\le C|x|^p$. Under~\eqref{w_unbiased_condition}, we can therefore write
\begin{equation*}
\bias_{f}(\hat{L}_{h_{+},h_{-}})=  \sum_{i=1}^{n} w_+(x_i,h_+)r_+(x)
  -\sum_{i=1}^{n} w_-(x_i,h_+)r_-(x),
\end{equation*}
which is maximized subject to the conditions $|r_+(x)|\le C|x|^p$ and
$|r_-(x)|\le C|x|^p$ by taking
$r_+(x_i)=C|x_i|^p\cdot \text{sign}(w_+(x_i,h_+))$ and
$r_-(x_i)=-C|x_i|^p\cdot \text{sign}(w_-(x_i,h_-))$. This yields the worst-case
bias formula~\Cref{eq:rd-worst-case-bias}.

\subsection{Solution to the modulus problem and optimal estimators}\label{rd_modulus_solution_sec}

This section derives the form of the optimal estimators and CIs. To that end, we
first need to find functions $g^{*}_{\delta}$ and $f^{*}_{\delta}$ that solve
the modulus problem. Since the class $\mathcal{F}_{RDT,p}(C)$ is
centrosymmetric, $f^{*}_{\delta}=-g^{*}_{\delta}$, and the (single-class)
modulus of continuity $\omega(\delta;\mathcal{F}_{RDT,p}(C))$ is given by the
value of the problem
\begin{equation}\label{rd_modulus_eq}
  \sup_{f_{+}+f_{-}\in\mathcal{F}_{RDT,p}(C)}
  2(f_{+}(0)-f_{-}(0))
  \quad
  \text{st}
  \quad
  \sum_{i=1}^{n}\frac{f_{-}(x_{i})^{2}}{\sigma^{2}(x_{i})}+
  \sum_{i=1}^{n}\frac{f_{+}(x_{i})^{2}}{\sigma^{2}(x_{i})}\leq \delta^{2}/4.
\end{equation}
Let $g^*_{\delta,C}$ denote the (unique up to the values at the $x_{i}$s)
solution to this problem. This solution can be obtained using a simple
generalization of Theorem 1 of \citet{SaYl78}. To describe it, define
$g_{b,C}(x)=g_{+,b,C}(x)+g_{-,b,C}(x)$ by
\begin{align*}
  g_{+,b,C}(x)
  &= \left((b-b_{-}+\textstyle\sum_{j=1}^{p-1}d_{+,j}x^{j}-C|x|^{p})_{+}
    -(b-b_{-}+\textstyle\sum_{j=1}^{p-1}d_{+,j}x^{j}+C|x|^{p})_{-}\right)\1{x\geq
    0},\\
  g_{-,b,C}(x)
  &= -\left((b_{-}+\textstyle\sum_{j=1}^{p-1}d_{-}x^{j}-C|x|^{p})_{+}
    -(b_{-}+\textstyle\sum_{j=1}^{p-1}d_{-,j}x^{j}+C|x|^{p})_{-}\right)\1{x<0},
\end{align*}
where we use the notation $(t)_+=\max\{t,0\}$ and $(t)_-=-\min\{t,0\}$. The
solution is given by $g^*_{\delta,C}=g_{b(\delta),C}$ where the coefficients
$d_{+}=(d_{+,1},\dotsc,d_{-,p-1})$, $d_{-}=(d_{-,1},\dotsc,d_{-,p-1})$, and
$b(\delta)$ and $b_{-}$ solve a system of equations given below. To see that the
solution must take the form $g_{b,C}(x)$ for some $b,b_-,d_+,d_-$, note that any
function $f_+\in\mathcal{F}_{T,p}(C)$ can be written as
\begin{equation}\label{eq:fplus-form}
  f_+(x)=b_++\sum_{j=1}^{p-1}d_{+,j}x^j+r_+(x),\qquad |r_+(x)|\le C|x|^p.
\end{equation}
Given $b_+,d_+$, in order to minimize $|f_+(x_i)|$ simultaneously for all $i$,
it must be that
\begin{equation*}
  r_{+}(x)=\begin{cases}
    -C|x|^p & \text{if $b_++\sum_{j=1}^{p=1}d_{+,j}x^j\ge C|x|^p$,} \\
    -b_+-\sum_{j=1}^{p=1}d_{+,j}x^j
    & \text{if $\abs{b_++\sum_{j=1}^{p=1}d_{+,j}x^j}< C|x|^p$,}\\
C\abs{x}^{p}    &\text{if $b_++\sum_{j=1}^{p=1}d_{+,j}x^j\le -C|x|^p$}.
\end{cases}
\end{equation*}
This form of $r(x)$ is necessary for $f_{+}$ to solve~\eqref{rd_modulus_eq}:
otherwise, one could strictly decrease
$\sum_{i=1}^n[f_-(x_i)^2/\sigma^2(x_i)+f_+(x_i)^2/\sigma^2(x_i)]$, thereby
making this quantity strictly less than $\delta^2/4$. But this would allow for a
strictly larger value of $2(f_+(0)+f_-(0))$ by increasing $b_+$ and leaving
$d_+$ and $r_+$ the same. Plugging $r_{+}(x)$ from the above display
into~\eqref{eq:fplus-form} shows that $f_+(x)=g_{+,b,C}(x)$ for some $b_+,d_+$.
Similar arguments apply for $f_{-}$.

Setting up the Lagrangian for the problem with $f$ constrained to the class of
functions that take the form $g_{b,C}$ for some $b,b_-,d_+,d_-$, and taking
first order conditions with respect to $b_-$, $d_+$ and $d_-$ gives
\begin{align}
  0&=\sum_{i=1}^{n}\frac{g_{-,b,C}(x_{i})}{\sigma^{2}(x_{i})}
     \left(x_{i},\dotsc,x_{i}^{p-1}\right)',    \label{eq:rd:d-condition1}\\
  0&=\sum_{i=1}^{n}\frac{g_{+,b,C}(x_{i})}{\sigma^{2}(x_{i})}
     \left(x_{i},\dotsc,x_{i}^{p-1}\right)',
     \label{eq:rd:d-condition2}\\
  0&=\sum_{i=1}^{n}\frac{ g_{+,b,C}(x_{i})}{\sigma^{2}(x_{i})}
     +\sum_{i=1}^{n}\frac{g_{-,b,C}(x_{i})}{\sigma^{2}(x_{i})}. \label{eq:rd:one-sided-condition}
\end{align}
The constraint in~\eqref{rd_modulus_eq} must be binding at the optimum, which
gives the additional equation
\begin{equation}\label{eq:rd:delta-b}
  \delta^{2}/4=  \sum_{i=1}^{n}\frac{g_{b,C}(x_{i})^{2}}{
    \sigma^{2}(x_{i})}  =b\sum_{i=1}^{n}\frac{g_{+,b,C}(x_{i})}{
    \sigma^{2}(x_{i})}-C\sum_{i=1}^{n}
  \frac{\abs{g_{b,C}(x_{i})}|x_{i}|^{p}}{
    \sigma^{2}(x_{i})},
\end{equation}
where the second equality follows
from~\eqref{eq:rd:d-condition1}--\eqref{eq:rd:d-condition2}. Note also
that, since $g^*_{\delta,C}=g_{b(\delta),C}$ solves the modulus problem and
gives the modulus as $2b(\delta)$, it also gives the solution to the inverse
modulus problem
\begin{equation}\label{rd_inverse_modulus_eq}
  \frac{\omega^{-1}(2b;\mathcal{F}_{RDT,p}(C))^{2}}{4} =\inf_{f_{+}-f_{-}\in\mathcal{F}_{RDT,p}(C)}
  \sum_{i=1}^n\left( \frac{f_+^2(x_i)}{\sigma^2(x_i)} +
    \frac{f_-^2(x_i)}{\sigma^2(x_i)}\right) \text{ s.t. } 2(f_+(0)-f_-(0))\ge 2b
\end{equation}
for $b=b(\delta)$. Since the objective for the inverse modulus is strictly
convex, this shows that the solution is unique up to the values at the $x_i$s.

Using the fact that the class $\mathcal{F}_{RDT,p}(C)$ is translation invariant
as defined in \Cref{translation_invariance_sec} (we can take
$\iota(x)=c_{0}+\1{x\geq 0}$ for any $c_{0}$), so that the derivative of the
modulus is given by~\Cref{translation_invariant_derivative_lemma}, along
with~\eqref{eq:rd:one-sided-condition} implies that the class of estimators
$\hat{L}_{\delta}$ can be written as
\begin{equation}\label{Ldelta_rd_eq}
  \hat{L}_{\delta}=  \hat{L}_{\delta,\mathcal{F}_{RDT,p}(C)}
  =
  \frac{\sum_{i=1}^{n}g^{*}_{+,\delta,C}(x_{i})y_{i}/\sigma^{2}(x_{i})}{
    \sum_{i=1}^{n}g^{*}_{+,\delta,C}(x_{i})/\sigma^{2}(x_{i})}
  -\frac{\sum_{i=1}^{n}g^{*}_{-,\delta,C}(x_{i})y_{i}/\sigma^{2}(x_{i})}{
    \sum_{i=1}^{n}g^{*}_{-,\delta,C}(x_{i})/\sigma^{2}(x_{i})}.
\end{equation}
Note that Conditions~\eqref{eq:rd:d-condition1},~\eqref{eq:rd:d-condition2},
and~\eqref{eq:rd:one-sided-condition} are simply the conditions
(\ref{w_unbiased_condition}) applied to this class of estimators.

To write the estimator $\hat{L}_{\delta}$ in the
form~\eqref{eq:rd-linear-estimator}, let
$w_{-}(x_{i},h_{-})=g_{-,b,C}(x_{i})/\sum_{i=1}^{n}g_{-,b,C}(x_{i})$ and
$w_{+}(x_{i},h_{+})=g_{+,b,C}(x_{i})/\sum_{i=1}^{n}g_{+,b,C}(x_{i})$, where
$d_{+}$ and $d_{-}$ solve~\eqref{eq:rd:d-condition1}
and~\eqref{eq:rd:d-condition2} with $b-b_{-}=Ch_{+}^{p}$ and $b_{-}=Ch_{-}^{p}$.
Then $\hat L_\delta=\hat L_{h_+(\delta),h_-(\delta)}$ where $h_+(\delta)$ and
$h_-(\delta)$ are determined by the additional
conditions~\eqref{eq:rd:one-sided-condition} and~\eqref{eq:rd:delta-b}.

To find the optimal estimators as described in \Cref{sec:pract-impl}, one can
use the estimator $\hat{L}_{h_{+},h_{-}}$ and optimize $h_{+}$ and $h_{-}$ for
the given performance criterion, using the variance and worst-case bias formulas
given in that section. Since the optimal estimator $\hat L_{\delta}$ (with
$\delta$ determined by the performance criterion) takes this form for some $h_+$
and $h_-$, the resulting estimator and CI will be the same as the one obtained
by computing $\hat L_\delta$ with $\delta$ determined by solving the additional
equation that corresponds to the performance criterion of interest.

\subsection{Lower bound on $C$}%
\label{sec:estim-lower-bound}

While it is not possible to consistently estimate the smoothness constant $C$
from the data, it is possible to lower bound its value. Here we develop a simple
estimator and lower CI for this bound, focusing on the case
$f\in\mathcal{F}_{RDT,2}(C)$.

As noted in \Cref{rd_modulus_solution_sec}, we can write
$f_{+}(x)=f_{+}(0)+f_{+}'(0)x+r_{+}(x)$, where $\abs{r_{+}(x)}\leq Cx^{2}$. It
therefore follows that for any three points $0\leq x_{1}\leq x_{2}\leq x_{3}$,
\begin{equation*}
  \lambda f_{+}(x_{1})+  (1-\lambda)f_{+}(x_{3})-f_{+}(x_{2})=
  \lambda r_{+}(x_{1})+  (1-\lambda)r_{+}(x_{3})-r_{+}(x_{2}),
\end{equation*}
where $\lambda=(x_{3}-x_{2})/(x_{3}-x_{1})$. The left-hand side measures the
curvature of $f$ by comparing $f(x_{2})$ to an approximation based on linearly
interpolating between $f(x_{1})$ and $f(x_{3})$. Since
$\abs{r_{+}(x)}\leq Cx^{2}$, the right-hand side is bounded by
$C(\lambda x_{1}^{2}+(1-\lambda)x_{3}^{3}+x_{2}^{2})$. Taking averages of the
preceding display over intervals $I_{k}=\hor{a_{k-1},a_{k}}$ where
$a_{0}\leq a_{1}\leq a_{2}\leq a_{3}$ and applying this bound yields the lower
bound
\begin{equation*}
  C\geq \abs{\mu_{+}},\qquad
  \mu_{+}=
  \frac{\lambda E_{n,1}(f_{+}(x))+  (1-\lambda)
    E_{n,3}(f_{+}(x))-E_{n,2}(f_{+}(x))}{
    \lambda E_{n,1}(x^{2})+  (1-\lambda)
    E_{n,3}(x^{2})+E_{n,2}(x^{2})},
\end{equation*}
where we use the notation
$E_{n,k}(g(x) )= \sum_{i}\1{x_{i}\in I_{k}}g(x_{i})/n_{k}$,
$n_{k}=\sum_{i}\1{x_{i}\in I_{k}}g(x_{i})$ to denote sample average over
$I_{k}$. Replacing $E_{n,k}(f_{+}(x))$ with $E_{n,k}(y)$ yields the estimator of
$\mu_{+}$
\begin{equation*}
  Z=\frac{\lambda E_{n,1}(y)+  (1-\lambda)
    E_{n,3}(y)-E_{n,2}(y)}{
    \lambda E_{n,1}(x^{2})+  (1-\lambda)
    E_{n,3}(x^{2})+E_{n,2}(x^{2})}\sim\mathcal{N}\left(\mu_{+},\tau^{2}\right),
\end{equation*}
where
$\tau^{2}=\frac{\lambda^{2} E_{n,1}(\sigma^{2}(x))/n_{1}+ (1-\lambda)^{2}
  E_{n,3}(\sigma^{2}(x))/n_{3}-E_{n,2}(\sigma^{2}(x))/n_{2}}{( \lambda
  E_{n,1}(x^{2})+ (1-\lambda) E_{n,3}(x^{2})+E_{n,2}(x^{2}))^{2}}$. Inverting
tests of the hypotheses $H_{0}\colon \abs{\mu_{+}}\leq \mu_{0}$ against
$H_{1}\colon\abs{\mu_{+}}> \mu_{0}$ then yields a one-sided CI for
$\abs{\mu_{+}}$ of the form $\hor{\hat{\mu}_{+,\alpha},\infty}$, where
$\hat{\mu}_{+,\alpha}$ solves $\abs{Z/\tau}=\cv_{\alpha}(\mu/\tau)$, with the
convention that $\hat{\mu}_{+,\alpha}=0$ if $\abs{Z/\tau}\leq \cv_{\alpha}(0)$.
This CI can be used as a lower CI for $C$ in model specification checks.

Since unbiased estimates of the lower bound $\abs{\mu_{+}}$ do not exist,
following \citet{clr13}, we take $\hat{\mu}_{+,0.5}$ as an estimator of the
lower bound, which has the property that it's half-median unbiased in the sense
that $P(\abs{\mu_{+}}\leq \hat{\mu}_{+,0.5})\leq 0.5$. An analogous bound
obtains by considering intervals below the cutoff. We leave the question of
optimal choice of the intervals $I_{k}$ to future research. In the \citet{lee08}
application, we set $a_{0}=0$, and set the remaining interval endpoints $a_{k}$
such that each interval $I_{k}$ contains 200 observations. This yields estimates
$\hat{\mu}_{+,0.5}=0.008$ and $\hat{\mu}_{-,0.5}=0.017$.

\subsection{Asymptotic validity}%
\label{sec:asympt-valid-optim}

We now give a theorem showing asymptotic validity of CIs from
\Cref{sec:pract-impl} under an unknown error distribution. We consider uniform
validity over regression functions in $\mathcal{F}$ and error distributions in a
sequence $\mathcal{Q}_n$, and we index probability statements with
$f\in\mathcal{F}$ and $Q\in\mathcal{Q}_n$. We make the following assumptions on
the $x_i$s and the class of error distributions $\mathcal{Q}_n$.

\begin{assumption}\label{rd_xs_assump_main_text}
  For some $p_{X,+}(0)>0$ and $p_{X,-}(0)>0$, the sequence $\{x_i\}_{i=1}^{n}$
  satisfies
  $\frac{1}{n h_n}\sum_{i=1}^{n}m(x_i/h_n)\1{x_i\geq 0} \to
  p_{X,+}(0)\int_0^\infty m(u)\, du$ and
  $\frac{1}{n h_n}\sum_{i=1}^{n}m(x_i/h_n)\1{x_i<0}\to
  p_{X,-}(0)\int_{-\infty}^0 m(u)\, du$ for any bounded function $m$ with
  bounded support and any $h_n$ with
  $0<\liminf_n h_n n^{1/(2p+1)}\le \limsup_n h_n n^{1/(2p+1)}<\infty$.
\end{assumption}

\begin{assumption}\label{rd_errors_assump_main_text}
  For some $\sigma(x)$ with $\lim_{x\downarrow 0}\sigma(x)=\sigma_+(0)>0$ and
  $\lim_{x\uparrow 0}\sigma(x)=\sigma_-(0)>0$,
  \begin{itemize}
  \item[(i)] the $u_i$s are independent under any $Q\in\mathcal{Q}_n$ with
    $E_Q u_i=0$, $\text{var}_Q(u_i)=\sigma^2(x_i)$
  \item[(ii)] for some $\eta>0$, $E_Q|u_i|^{2+\eta}$ is bounded uniformly over
    $n$ and $Q\in\mathcal{Q}_n$.
  \end{itemize}
\end{assumption}

While the variance function $\sigma^2(x)$ is unknown, the definition of
$\mathcal{Q}_n$ is such that the variance function is the same for all
$Q\in\mathcal{Q}_n$. This is done for simplicity. One could consider uniformity
over classes $\mathcal{Q}_n$ that place only smoothness conditions on
$\sigma^2(x)$ at the cost of introducing additional notation and making the
optimality statements more cumbersome.

The estimators and CIs that we consider in the sequel are based on an estimate
$\hat\sigma(x)$ of the conditional variance in Step~\ref{item:step1} of the
procedure in \Cref{sec:pract-impl}. We make the following assumption on this
estimate.

\begin{assumption}\label{estimated_tilde_sigma_assump_main_text}
  The estimate $\hat\sigma(x)$ is given by
  $\hat\sigma(x)=\hat\sigma_+(0)\1{x\geq 0}+\hat\sigma_-(0)\1{x<0}$ where
  $\hat\sigma_+(0)$ and $\hat\sigma_-(0)$ are consistent for $\sigma_+(0)$ and
  $\sigma_-(0)$ uniformly over $f\in\mathcal{F}$ and $Q\in\mathcal{Q}_n$.
\end{assumption}

For asymptotic coverage, we consider uniformity over both $\mathcal{F}$ and
$\mathcal{Q}_n$. Thus, a confidence set $\mathcal{C}$ is said to have asymptotic
coverage at least $1-\alpha$ if
\begin{equation*}
  \liminf_{n\to\infty}\inf_{f\in\mathcal{F},Q\in\mathcal{Q}_n}
  P_{f,Q}\left(Lf\in\mathcal{C}\right)\ge 1-\alpha.
\end{equation*}

\begin{theorem}\label{rd_optimal_estimator_thm_main_text}
  Under
  Assumptions~\ref{rd_xs_assump_main_text},~\ref{rd_errors_assump_main_text}
  and~\ref{estimated_tilde_sigma_assump_main_text}, CIs given in
  \Cref{sec:pract-impl} based on $\hat{L}_{\delta}$ have asymptotic coverage at
  least $1-\alpha$. CIs based on local polynomial estimators have asymptotic
  coverage at least $1-\alpha$ so long as the kernel is bounded and uniformly
  continuous with bounded support and the bandwidths $h_+$ and $h_-$ satisfy
  $h_+ n^{1/(2p+1)}\to h_{+,\infty}$ and $h_- n^{1/(2p+1)}\to h_{-,\infty}$ for
  some $h_{+,\infty}>0$ and $h_{-,\infty}>0$.

  Let $\hat \chi$ denote the half-length of the optimal fixed-length CI based on
  $\hat\sigma(x)$. For $\chi_\infty$ given in \Cref{sec:sap:rd_asym}, the scaled half-length $n^{p/(2p+1)}\hat \chi$
  converges in probability to $\chi_{\infty}$ uniformly over $\mathcal{F}$ and
  $\mathcal{Q}_n$. If, in addition, each $\mathcal{Q}_n$ contains a distribution
  where the $u_i$s are normal, then for any sequence of confidence sets
  $\mathcal{C}$ with asymptotic coverage at least $1-\alpha$, we have the
  following bound on the asymptotic efficiency improvement at any
  $f\in\mathcal{F}_{RDT,p}(0)$
  \begin{equation*}
    \liminf_{n\to\infty}\sup_{Q\in\mathcal{Q}_n} \frac{n^{p/(2p+1)}E_{f,Q} \lambda(\mathcal{C})}
    {2 \chi_\infty}
    \ge
    \frac{(1-\alpha) 2^r E[(z_{1-\alpha}-Z)^r\mid Z\leq z_{1-\alpha}]}
    {2r\inf_{\delta>0}\cv_\alpha\left((\delta/2)(1/r-1)\right)\delta^{r-1}},
  \end{equation*}
  where $Z\sim\mathcal{N}(0,1)$ and $r=2p/(2p+1)$.

  Letting $\hat c_{\alpha,\delta}$ denote the lower endpoint of the one-sided CI
  corresponding to $\hat L_\delta$, the CI $\hor{\hat c_{\alpha,\delta},\infty}$
  has asymptotic coverage at least $1-\alpha$. If $\delta$ is chosen to minimax
  the $\beta$ quantile excess length, (i.e. $\delta=z_\beta+z_{1-\alpha}$),
  then, if each $\mathcal{Q}_n$ contains a distribution where the $u_i$s are
  normal, any other one-sided CI $\hor{\hat c,\infty}$ with asymptotic coverage
  at least $1-\alpha$ must satisfy the efficiency bound
  \begin{equation*}
    \liminf_{n\to\infty}
    \frac{\sup_{f\in\mathcal{F},Q\in\mathcal{Q}_n} q_{f,Q,\beta}\left(Lf-\hat c\right)}
    {\sup_{f\in\mathcal{F},Q\in\mathcal{Q}_n} q_{f,Q,\beta}\left(Lf-\hat c_{\alpha,\delta}\right)}\ge 1.
  \end{equation*}
  In addition, we have the following bound on the asymptotic efficiency
  improvement at any $f\in\mathcal{F}_{RDT,p}(0)$:
  \begin{equation*}
    \liminf_{n\to\infty}
    \frac{\sup_{Q\in\mathcal{Q}_n}q_{f,Q,\beta}\left(Lf-\hat c\right)}
    {\sup_{Q\in\mathcal{Q}_n} q_{f,Q,\beta}\left(Lf-\hat c_{\alpha,\delta}\right)}
    \ge \frac{2^{r}}{1+r}.
  \end{equation*}
\end{theorem}

The proof of Theorem~\ref{rd_optimal_estimator_thm_main_text} is given in
\Cref{sec:sap:rd_asym}. The asymptotic efficiency bounds
correspond to those in \Cref{sec:general_results}
under~\eqref{eq:asymptotoc-modulus} with $r=2p/(2p+1)$.

\section{Unknown Error Distribution}%
\label{sec:sap:unknown_error}

The Gaussian regression model~\eqref{eq:fixed_design_eq} makes the assumption of
normal i.i.d.\ errors with a known variance conditional on the $x_i$'s, which is
often unrealistic. This section considers a model that relaxes these assumptions
on the error distribution:
\begin{align}\label{unknown_error_eq}
 y_i=f(x_i)+u_i,\,
\{u_i\}_{i=1}^n\sim Q,\,
f\in\mathcal{F},\,
Q\in\mathcal{Q}_n
\end{align}
where $\mathcal{Q}_n$ denotes the set of possible joint distributions of
$\{u_i\}_{i=1}^n$ and, as before, $\{x_i\}_{i=1}^n$ is deterministic and
$\mathcal{F}$ is a convex set. We derive feasible versions of the optimal CIs in
\Cref{sec:general_results} and show their asymptotic validity (uniformly over
$\mathcal{F},\mathcal{Q}_n$) and asymptotic efficiency. As we discuss below, our
results hold even in cases where the limiting form of the optimal estimator is
unknown or may not exist, and where currently available methods for showing
asymptotic efficiency, such as equivalence with Gaussian white noise, break
down.

Since the distribution of the data $\{y_i\}_{i=1}^n$ now depends on both $f$ and
$Q$, we now index probability statements by both of these quantities: $P_{f,Q}$
denotes the distribution under $(f,Q)$ and similarly for $E_{f,Q}$. The coverage
requirements and definitions of minimax performance criteria in
\Cref{sec:general_results} are the same, but with infima and suprema over
functions $f$ now taken over both functions $f$ and error distributions
$Q\in\mathcal{Q}_n$. We will also consider asymptotic results. We use the
notation $Z_n\underset{\mathcal{F},\mathcal{Q}_n}{\overset{d}{\to}} \mathcal{L}$
to mean that $Z_n$ converges in distribution to $\mathcal{L}$ uniformly over
$f\in\mathcal{F}$ and $Q\in\mathcal{Q}_n$, and similarly for
$\underset{\mathcal{F},\mathcal{Q}_n}{\overset{p}{\to}}$.

If the variance function is unknown, the estimator $\hat L_\delta$ is
infeasible. However, we can form an estimate based on an estimate of the
variance function, or based on some candidate variance function. For a candidate
variance function $\tilde\sigma^2(\cdot)$, let
$K_{\tilde\sigma(\cdot),n}f= (f(x_1)/\tilde\sigma(x_1), \ldots,\allowbreak
f(x_n)/\tilde\sigma(x_n))'$, and let $\omega_{\tilde\sigma(\cdot),n}(\delta)$
denote the modulus of continuity defined with this choice of $K$. Let
$\hat{L}_{\delta,\tilde\sigma(\cdot)}=\hat{L}_{\delta,\mathcal{F},\mathcal{G},\tilde\sigma(\cdot)}$
denote the estimator defined in (\ref{eq:Ldelta_eq}) with this choice of $K$ and
$Y=(y_1/\tilde\sigma(x_1),\ldots,y_n/\tilde\sigma(x_n))'$, and let
$f^*_{\tilde\sigma(\cdot),\delta}$ and $g^*_{\tilde\sigma(\cdot),\delta}$ denote
the least favorable functions used in forming this estimate. We assume
throughout this section that $\mathcal{G}\subseteq\mathcal{F}$. More generally,
we will consider affine estimators, which, in this setting, take the form
\begin{equation}\label{affine_est_eq}
\hat L=a_n+\sum_{i=1}^{n}w_{i,n} y_i
\end{equation}
where $a_n$ and $w_{i,n}$ are a sequence and triangular array respectively.
For now, we assume that $a_n$ and $w_{i,n}$ are nonrandom, (which, in the case of the estimator $\hat L_{\delta,\tilde\sigma(\cdot)}$, requires that $\tilde\sigma(\cdot)$ and $\delta$ be nonrandom).
We provide conditions that allow for random $a_n$ and $w_{i,n}$ after stating our result for nonrandom weights.
For a class $\mathcal{G}$, the maximum and minimum bias are
\begin{align*}
\maxbias_{\mathcal{G}}(\hat L)
&=\sup_{f\in\mathcal{G}} \Big[a_n+\sum_{i=1}^{n}w_{i,n}f(x_i)-Lf\Big],
&
\minbias_{\mathcal{G}}(\hat L)
&=\inf_{f\in\mathcal{G}} \Big[a_n+\sum_{i=1}^{n}w_{i,n}f(x_i)-Lf\Big].
\end{align*}
By the arguments used to derive the formula~\eqref{eq:maxbias-minbias}, we have
\begin{equation*}
  \overline{\text{bias}}_{\mathcal{F}}(\hat{L}_{\delta,\mathcal{F},\mathcal{G},
    \tilde{\sigma}(\cdot)})=-\underline{\text{bias}}_{\mathcal{G}}(
  \hat{L}_{\delta,\mathcal{F},\mathcal{G},\tilde \sigma(\cdot)})=\frac{1}{2}
  (\omega_{n,\tilde\sigma(\cdot)}(\delta;\mathcal{F},\mathcal{G})-
  \delta\omega_{n,\tilde\sigma(\cdot)}'(\delta;\mathcal{F},\mathcal{G})).
\end{equation*}
This holds regardless of whether $\tilde\sigma(\cdot)$ is equal to the actual
variance function of the $u_i$'s. In our results below, we allow for infeasible
estimators in which $a_n$ and $w_{i,n}$ depend on $Q$ (for example, when the
unknown variance $\sigma_Q(x_i)=\text{var}_Q(y_i)$ is used to compute the
optimal weights), so that $\maxbias_{\mathcal{G}}(\hat L)$ and
$\minbias_{\mathcal{G}}(\hat L)$ may depend on $Q$. We leave this implicit in
our notation.

Let $s_{n,Q}$ denote the (constant over $f$) standard deviation of $\hat L$ under $Q$ and suppose that the uniform central limit theorem
\begin{align}\label{epsilon_clt_eq}
  \frac{\sum_{i=1}^n w_{i,n}u_i}{s_{n,Q}}
  \underset{\mathcal{F},\mathcal{Q}_n}{\overset{d}{\to}} \mathcal{N}(0,1)
\end{align}
holds.  To form a feasible CI, we will require an estimate $\widehat{\text{se}}_n$ of $s_{n,Q}$ satisfying
\begin{align}\label{se_convergence_eq}
\frac{\widehat{\text{se}}_n}{s_{n,Q}}
  \underset{\mathcal{F},\mathcal{Q}_n}{\overset{p}{\to}} 1.
\end{align}
The following theorem shows that using $\widehat{\text{se}}_n$ to form analogues
of the CIs treated in \Cref{sec:general_results} gives asymptotically valid CIs.

\begin{theorem}\label{high_level_asymptotic_thm}
  Let $\hat L$ be an estimator of the form (\ref{affine_est_eq}), and suppose
  that (\ref{epsilon_clt_eq}) and (\ref{se_convergence_eq}) hold. Let $\hat c=\hat L
  -\overline{\text{bias}}_{\mathcal{F}}(\hat L)
  -\widehat{\text{se}}_{n}z_{1-\alpha}$, and let
  $b=\max\{|\overline{\text{bias}}_{\mathcal{F}}(\hat{L})|,|
  \underline{\text{bias}}_{\mathcal{F}}(\hat L)|\}$. Then
\begin{equation}\label{asym_one_sided_coverage_eq}
\liminf_{n\to\infty}\inf_{f\in\mathcal{F},Q\in\mathcal{Q}_n}
P_{f,Q}\left(Lf\in \hor{\hat c,\infty}\right)\ge 1-\alpha
\end{equation}
and
\begin{equation}\label{asym_two_sided_coverage_eq}
\liminf_{n\to\infty}\inf_{f\in\mathcal{F},Q\in\mathcal{Q}_n}
P_{f,Q}\left(Lf\in \left\{\hat L
\pm \widehat{\text{se}}_n\cv_{\alpha}\left(b/\widehat{\text{se}}_n\right)\right\}
\right)\ge 1-\alpha.
\end{equation}
The worst-case $\beta$th quantile excess length of the one-sided CI over
$\mathcal{G}$ will satisfy
\begin{equation}\label{one_side_asym_minimax_eq}
\limsup_{n\to\infty}\sup_{Q\in\mathcal{Q}_n} \frac{\sup_{g\in\mathcal{G}}q_{g,Q,\beta}(Lg-\hat c)}
 {\overline{\text{bias}}_{\mathcal{F}}(\hat L)-\underline{\text{bias}}_{\mathcal{G}}(\hat L)+s_{n,Q}(z_{1-\alpha}+z_{\beta})}
\le 1
\end{equation}
and the length of the two-sided CI will satisfy
\begin{equation*}
\frac{\cv_{\alpha}\left(b/\widehat{\text{se}}_n\right)\widehat{\text{se}}_n}
 {\cv_{\alpha}\left(b/s_{n,Q}\right)s_{n,Q}}
\underset{\mathcal{F},\mathcal{Q}_n}{\overset{p}{\to}} 1.
\end{equation*}
Suppose, in addition, that
$\hat L=\hat L_{\delta,\mathcal{F},\mathcal{G},\tilde\sigma(\cdot)}$ with
$\tilde\sigma(\cdot)=\sigma_Q(\cdot)$ where $\sigma_Q^2(x_i)=\text{var}_Q(u_i)$
and, for each $n$, there exists a $Q_n\in\mathcal{Q}_n$ such that
$\{u_i\}_{i=1}^n$ are independent and normal under $Q_n$. Then no one-sided CI
satisfying (\ref{asym_one_sided_coverage_eq}) can satisfy
(\ref{one_side_asym_minimax_eq}) with the constant $1$ replaced by a strictly
smaller constant on the right-hand side.

\end{theorem}
\begin{proof}
  Let $Z_n=\sum_{i=1}^{n}w_{i,n}u_i/\widehat{\text{se}}_n$, and let $Z$ denote a
  standard normal random variable. To show asymptotic coverage of the one-sided
  CI, note that
\begin{equation*}
  P_{f,Q}\left(Lf\in\hor{\hat{c},\infty}\right)
  =P_{f,Q}\left(\widehat{\text{se}}_{n}z_{1-\alpha}\ge \hat L-Lf
    -\overline{\text{bias}}_{\mathcal{F}}(\hat L)\right)
  \ge P_{f,Q}\left(z_{1-\alpha}
    \ge Z_{n}\right)
\end{equation*}
using the fact that
$\overline{\text{bias}}_{\mathcal{F}}(\hat{L})+\sum_{i=1}^{n} w_{i,n}u_i\ge
\hat{L}-Lf$ for all $f\in\mathcal{F}$ by the definition of
$\overline{\text{bias}}_{\mathcal{F}}$. The right-hand side converges to
$1-\alpha$ uniformly over $f\in\mathcal{F}$ and $Q\in\mathcal{Q}_n$ by
(\ref{epsilon_clt_eq}) and (\ref{se_convergence_eq}).  For the two-sided CI, first note that
\begin{equation*}
  \left|\frac{\cv_{\alpha}\left(b/\widehat{\text{se}}_n\right)\widehat{\text{se}}_n}
    {\cv_{\alpha}\left(b/s_{n,Q}\right)s_{n,Q}} -1\right|
=\left|\frac{\cv_{\alpha}\left(b/\widehat{\text{se}}_n\right)-\cv_{\alpha}\left(b/s_{n,Q}\right)
+\cv_{\alpha}\left(b/s_{n,Q}\right)(1-s_{n,Q}/\widehat{\text{se}}_n)}
    {\cv_{\alpha}\left(b/s_{n,Q}\right)(s_{n,Q}/\widehat{\text{se}}_n)}\right|
\end{equation*}
which converges to zero uniformly over $f\in\mathcal{F},Q\in\mathcal{Q}_n$ since
$\cv_\alpha(t)$ is bounded from below and uniformly continuous with
respect to $t$.  Thus, $\frac{\cv_{\alpha}\left(b/\widehat{\text{se}}_n\right)\widehat{\text{se}}_n}
 {\cv_{\alpha}\left(b/s_{n,Q}\right)s_{n,Q}}
\underset{\mathcal{F},\mathcal{Q}_n}{\overset{p}{\to}} 1$ as claimed.  To show coverage of the two-sided CI, note that
\begin{equation*}
  P_{f,Q}\left(Lf\in \left\{\hat L \pm
      \cv_{\alpha}\left(b/\widehat{\text{se}}_n\right)\widehat{\text{se}}_n\right\}
  \right)
  =P_{f,Q}\left(\left|\tilde Z_n+r\right| \le
    \cv_{\alpha}\left(b/s_{n,Q}\right)\cdot c_n
  \right)
\end{equation*}
where
$c_n=\frac{\cv_{\alpha}\left(b/\widehat{\text{se}}_n\right)\widehat{\text{se}}_n}{\cv_{\alpha}\left(b/s_{n,Q}\right)s_{n,Q}}$,
$\tilde Z_n=\sum_{i=1}^{n}w_{i,n}u_i/s_{n,Q}$ and
$r=\left(a_n+\sum_{i=1}^n w_{i,n}f(x_i)-Lf\right)/s_{n,Q}$.
By~\eqref{epsilon_clt_eq} and the fact that
$c_n \underset{\mathcal{F},\mathcal{Q}_n}{\overset{p}{\to}} 1$, this is equal to
$P_{f,Q}\left(\left|Z+r\right| \le \cv_{\alpha}\left(b/s_{n,Q}\right)\right)$
(where $Z\sim \mathcal{N}(0,1)$) plus a term that converges to zero uniformly
over $f,Q$ (this can be seen by using the fact that convergence in distribution
to a continuous distribution implies uniform convergence of the cdfs; see Lemma
2.11 in \citealt{van_der_vaart_asymptotic_1998}). Since $|r|\le b/s_{n,Q}$,
(\ref{asym_two_sided_coverage_eq}) follows.

To show~\eqref{one_side_asym_minimax_eq}, note that,
\begin{multline*}
  Lg-\hat c =Lg -a_n-\sum_{i=1}^n w_{i,n}g(x_i)-\widehat{\text{se}}_{n} Z_{n}
  +\overline{\text{bias}}_{\mathcal{F}}(\hat L)
  +\widehat{\text{se}}_{n}z_{1-\alpha} \\
  \le \maxbias_{\mathcal{F}}(\hat L) -\minbias_{\mathcal{G}}(\hat L)
  +\widehat{\text{se}}_n(z_{1-\alpha}-Z_n)
\end{multline*}
for any $g\in\mathcal{G}$.
Thus,
\begin{multline*}
  \frac{Lg-\hat c}{\maxbias_{\mathcal{F}}(\hat L)
    -\minbias_{\mathcal{G}}(\hat{L})+s_{n,Q}(z_{1-\alpha}+z_{\beta})} -1 \le
  \frac{\widehat{\text{se}}_n(z_{1-\alpha}-Z_n)-s_{n,Q}(z_{1-\alpha}+z_{\beta})}
  {\maxbias_{\mathcal{F}}(\hat L) -\minbias_{\mathcal{G}}(\hat L)+s_{n,Q}(z_{1-\alpha}+z_{\beta})} \\
  =\frac{(\widehat{\text{se}}_n/s_{n,Q})\cdot
    (z_{1-\alpha}-Z_n)-(z_{1-\alpha}+z_{\beta})}
  {[\maxbias_{\mathcal{F}}(\hat{L})
    -\minbias_{\mathcal{G}}(\hat{L})]/s_{n,Q}+(z_{1-\alpha}+z_{\beta})}.
\end{multline*}
The $\beta$ quantile of the above display converges to $0$ uniformly over $f\in\mathcal{F}$ and $Q\in\mathcal{Q}_n$,
which gives the result.

For the last statement, let $\hor{\tilde c,\infty}$ be a sequence of CIs with
asymptotic coverage $1-\alpha$. Let $Q_n$ be the distribution from the
conditions in the theorem, in which the $u_i$'s are independent and normal.
Then, by Theorem~\ref{th:one_side_adapt_thm},
\begin{equation*}
\sup_{g\in\mathcal{F}}q_{f,Q_n,\beta}(\tilde c-Lg)
\ge \omega_{\sigma_{Q_n}(\cdot),n}(\tilde \delta_n),
\end{equation*}
where $\tilde \delta_n=z_{1-\alpha_n}+z_\beta$ and $1-\alpha_n$ is the coverage
of $\hor{\tilde c,\infty}$ over $\mathcal{F},\mathcal{Q}_n$. When
$\hat L=\hat L_{\delta,\mathcal{F},\mathcal{G},\sigma_{Q}(\cdot)}$, the
denominator in (\ref{one_side_asym_minimax_eq}) for $Q=Q_n$ is equal to
$\omega_{\sigma_{Q_n}(\cdot),n}(z_{1-\alpha}+z_\beta)$, which gives
\begin{equation*}
\frac{\sup_{g\in\mathcal{G}}q_{g,Q_n,\beta}(\hat c-Lg)}
 {\maxbias_{\mathcal{F}}(\hat L) -\minbias_{\mathcal{G}}(\hat L)+s_{n,Q_n}(z_{1-\alpha}+z_{\beta})}
\ge \frac{\omega_{\sigma_{Q_n}(\cdot),n}(z_{1-\alpha_n}+z_\beta)}
           {\omega_{\sigma_{Q_n}(\cdot),n}(z_{1-\alpha}+z_\beta)}.
\end{equation*}
If $\alpha_n\le \alpha$, then
$z_{1-\alpha_n}+z_\beta\ge z_{1-\alpha}-z_\beta$ so that the above display is greater than one by monotonicity of the modulus.  If not, then by concavity,
$\omega_{\sigma_{Q_n}(\cdot),n}(z_{1-\alpha_n}+z_\beta)\ge [\omega_{\sigma_{Q_n}(\cdot),n}(z_{1-\alpha}+z_\beta)/(z_{1-\alpha}+z_\beta)]\cdot (z_{1-\alpha_n}+z_\beta)$,
so the above display is bounded from below by $(z_{1-\alpha_n}+z_\beta)/(z_{1-\alpha}+z_\beta)$, and the $\liminf$ of this is at least one by the coverage requirement.
\end{proof}

The efficiency bounds in Theorem~\ref{high_level_asymptotic_thm} use the
assumption that the class of possible distributions contains a normal law, as is
often done in the literature on efficiency in nonparametric settings \citep[see,
e.g.,][pp. 205--206]{fan93}. We leave the topic of relaxing this assumption for
future research.

Theorem~\ref{high_level_asymptotic_thm} requires that a known candidate variance function $\tilde\sigma(\cdot)$ and a known $\delta$ be used when forming CIs based on the estimate $\hat L_\delta$.  However, the theorem does not require that the candidate variance function be correct in order to get asymptotic coverage, so long as the standard error $\widehat{\text{se}}_n$ is consistent.  If it turns out that $\tilde\sigma(\cdot)$ is indeed the correct variance function, then it follows from the last part of the theorem that the resulting CI is efficient.  In the special case where $\mathcal{F}$ imposes a (otherwise unconstrained) linear model, this corresponds to the common practice of using ordinary least squares with heteroskedasticity robust standard errors.

In some cases, one will want to use a data dependent $\tilde\sigma(\cdot)$ and
$\delta$ in order to get efficient estimates with unknown variance. The
asymptotic coverage and efficiency of the resulting CI can then be derived by
showing equivalence with the infeasible estimator
$\hat L_{\delta^*,\mathcal{F},\mathcal{G},\sigma_Q(\cdot)}$, where $\delta^*$ is
chosen according to the desired performance criterion. The following theorem
gives conditions for this asymptotic equivalence. We verify them for our
regression discontinuity example in \Cref{sec:sap:rd_asym}.

\begin{theorem}\label{high_level_asym_equivalence_thm}
  Suppose that $\hat{L}$ and $\widehat{\text{se}}_n$
  satisfy~\eqref{epsilon_clt_eq} and~\eqref{se_convergence_eq}. Let $\tilde{L}$
  and $\widetilde{\text{se}}_n$ be another estimator and standard error, and let
  $\widetilde{\overline{\text{bias}}}_n$ and
  $\widetilde{\underline{\text{bias}}}_n$ be (possibly data dependent)
  worst-case bias formulas for $\tilde L$ under $\mathcal{F}$. Suppose that
\begin{align*}
  \frac{\hat L-\tilde L}{s_{n,Q}}& \underset{\mathcal{F},\mathcal{Q}_n}{\overset{p}{\to}} 0,
  &
    \frac{\overline{\text{bias}}_{\mathcal{F}}(\hat L)-\widetilde{\overline{\text{bias}}}_n}{s_{n,Q}}
  &    \underset{\mathcal{F},\mathcal{Q}_n}{\overset{p}{\to}} 0,
  & \frac{\underline{\text{bias}}_{\mathcal{F}}(\hat L)-\widetilde{\underline{\text{bias}}}_n}{s_{n,Q}}
  &    \underset{\mathcal{F},\mathcal{Q}_n}{\overset{p}{\to}} 0,
  &
    \frac{\widehat{\text{se}}_n}{\widetilde{\text{se}}_n}
  &  \underset{\mathcal{F},\mathcal{Q}_n}{\overset{p}{\to}} 1.
\end{align*}
Let $\tilde c=\tilde L
  -\widetilde{\overline{\text{bias}}}_n
  -\widetilde{\text{se}}_{n}z_{1-\alpha}$,
and let
  $\tilde b=\max\{|\widetilde{\overline{\text{bias}}}_n|,|\widetilde{\underline{\text{bias}}}_n|\}$.
Then (\ref{asym_one_sided_coverage_eq}) and (\ref{asym_two_sided_coverage_eq}) hold with $\hat c$ replaced by $\tilde c$, $\hat L$ replaced by $\tilde L$, $b$ replaced by $\tilde b$ and $\widehat{\text{se}}_n$ replaced by $\widetilde{\text{se}}_n$.  Furthermore, the performance of the CIs is asymptotically equivalent in the sense that
\begin{equation*}
\frac{\sup_{Q\in \mathcal{Q}_n}\sup_{g\in\mathcal{G}}q_{g,Q,\beta}(\tilde c-Lg)}{\sup_{Q\in \mathcal{Q}_n}\sup_{g\in\mathcal{G}}q_{g,Q,\beta}(\hat c-Lg)} \to 1
\text{ and }
\frac{\cv_\alpha(b/\widehat{\text{se}}_n)\widehat{\text{se}}_n}{\cv_\alpha(\tilde b/\widetilde{\text{se}}_n)\widetilde{\text{se}}_n}\underset{\mathcal{F},\mathcal{Q}_n}{\overset{p}{\to}} 1.
\end{equation*}
\end{theorem}
\begin{proof}
By the conditions of the theorem, we have, for some $c_n$ that converges in probability to zero uniformly over $\mathcal{F},\mathcal{Q}_n$,
\begin{multline*}
\tilde c-Lf
=\tilde L-Lf-\widetilde{\overline{\text{bias}}}_n-\widetilde{\text{se}}_{n}z_{1-\alpha}
=\hat L-Lf-\maxbias_{\mathcal{F}}(\hat L)-s_{n,Q}z_{1-\alpha}+c_n s_{n,Q}  \\
\le \sum_{i=1}^{n}w_{i,n}u_i-s_{n,Q}z_{1-\alpha}+c_n s_{n,Q}.
\end{multline*}
Thus,
\begin{equation*}
  P_{f,Q}\left(Lf\in \hor{\tilde c,\infty}\right)
  =P_{f,Q}\left(0 \ge \tilde c-Lf\right)
  \ge P_{f,Q}\left(0\ge \frac{\sum_{i=1}^{n}w_{i,n}u_i}{s_{n,Q}}
    -z_{1-\alpha}+c_n\right),
\end{equation*}
which converges to $1-\alpha$ uniformly over $\mathcal{F},\mathcal{Q}_n$. By
Theorem~\ref{high_level_asymptotic_thm},
$\sup_{g\in\mathcal{G}}q_{g,Q,\beta}(\hat c-Lg)$ is bounded from below by a
constant times $s_{n,Q}$. Thus,
$\left|\frac{\sup_{Q\in\mathcal{Q}_n}\sup_{g\in\mathcal{G}}
    q_{g,Q,\beta}(\tilde{c}-Lg)}{\sup_{Q\in\mathcal{Q}_n}
    \sup_{g\in\mathcal{G}}q_{g,Q,\beta}(\hat c-Lg)}- 1\right|$ is bounded from
above by a constant times
\begin{equation*}
  \sup_{Q\in\mathcal{Q}_n}\sup_{g\in\mathcal{G}}
  \left|\frac{q_{g,Q,\beta}(\tilde c-Lg)-q_{g,Q,\beta}(\hat c-Lg)}{s_{n,Q}}\right|
  =\sup_{Q\in\mathcal{Q}_n}\sup_{g\in\mathcal{G}}
  \left|q_{g,Q,\beta}(\tilde c/s_{n,Q})-q_{g,Q,\beta}(\hat c/s_{n,Q})\right|,
\end{equation*}
which converges to zero since $(\tilde c-\hat c)/s_{n,Q}\underset{\mathcal{F},\mathcal{Q}_n}{\overset{p}{\to}} 0$.

The claim that $\frac{\cv_\alpha(b/\widehat{\text{se}}_n)\widehat{\text{se}}_n}{\cv_\alpha(\tilde b/\widetilde{\text{se}}_n)\widetilde{\text{se}}_n}\underset{\mathcal{F},\mathcal{Q}_n}{\overset{p}{\to}} 1$ follows using similar arguments to the proof of Theorem~\ref{high_level_asymptotic_thm}.
To show coverage of the two-sided CI, note that
\begin{equation*}
P_{f,Q}\left(Lf\in \left\{\tilde L \pm
      \cv_{\alpha}\left(\tilde b/\widetilde{\text{se}}_n\right)\widetilde{\text{se}}_n\right\}\right)
=P_{f,Q}\left(\frac{|\tilde L-Lf|}{s_{n,Q}}\le \cv_{\alpha}\left(b/s_{n,Q}\right)\cdot c_n\right),
\end{equation*}
where
$c_n=\frac{\cv_{\alpha}(\tilde{b} /
  \widetilde{\text{se}}_n)\widetilde{\text{se}}_n}{\cv_{\alpha}(b/s_{n,Q})
  s_{n,Q}} \underset{\mathcal{F},\mathcal{Q}_n}{\overset{p}{\to}} 1$. Since
$\frac{|\tilde L-Lf|}{s_{n,Q}}=|V_n+r|$ where
$r=(a_n+\sum_{i=1}^n w_{i,n}f(x_i)-Lf)/s_{n,Q}$ and
$V_n=\sum_{i=1}^n w_{i,n}u_i/s_{n,Q}
+(\tilde{L}-\hat{L})/s_{n,Q}\underset{\mathcal{F},\mathcal{Q}_n}{\overset{d}{\to}}
\mathcal{N}(0,1)$, the result follows from arguments in the proof of
Theorem~\ref{high_level_asymptotic_thm}.
\end{proof}

The results above give high-level conditions that can be applied to a wide range
of estimators and CIs. We now introduce an estimator and standard error formula
that give asymptotic coverage for essentially arbitrary functionals $L$ under
generic low level conditions on $\mathcal{F}$ and the $x_i$'s. The estimator is
based on a nonrandom guess for the variance function and, if this guess is
correct up to scale (e.g.\ if the researcher correctly guesses that the errors
are homoskedastic), the one-sided CI based on this estimator will be
asymptotically optimal for some quantile of excess length.

Let $\tilde\sigma(\cdot)$ be some nonrandom guess for the variance function bounded away from $0$ and $\infty$, and let $\delta>0$ be a deterministic constant specified by the researcher.  Let $\hat f$ be an estimator of $f$.  The variance of
$\hat L_{\delta,\tilde\sigma(\cdot)}$ under some $Q\in\mathcal{Q}_n$ is equal to
\begin{equation*}
\text{var}_Q(\hat L_{\delta,\tilde\sigma(\cdot),n})
=\left(\frac{\omega_{\tilde\sigma(\cdot),n}'(\delta)}{\delta}\right)^2\sum_{i=1}^n \frac{(g^*_{\tilde\sigma(\cdot),\delta}(x_i)-f^*_{\tilde\sigma(\cdot),\delta}(x_i))^2\sigma_Q^2(x_i)}{\tilde\sigma^4(x_i)}.
\end{equation*}
We consider the estimate
\begin{equation*}
\widehat{\text{se}}^2_{\delta,\tilde\sigma(\cdot),n}
=\left(\frac{\omega_{\tilde\sigma(\cdot),n}'(\delta)}{\delta}\right)^2\sum_{i=1}^n \frac{(g^*_{\tilde\sigma(\cdot),\delta}(x_i)-f^*_{\tilde\sigma(\cdot),\delta}(x_i))^2
( y_i-\hat f(x_i))^2}{\tilde\sigma^4(x_i)}.
\end{equation*}

Suppose that $f:\mathcal{X}\to\mathbb{R}$ where $\mathcal{X}$ is a metric space with metric $d_X$ such that
the functions $f^*_{\tilde\sigma(\cdot),\delta}$ and $g^*_{\tilde\sigma(\cdot),\delta}$ satisfy the uniform continuity condition
\begin{equation}\label{f_uniform_continuity_condition}
\sup_{n} \sup_{x,x'\colon d_X(x,x')\le \eta}
\max\left\{ \left|f^*_{\tilde\sigma(\cdot),\delta}(x)-f^*_{\tilde\sigma(\cdot),\delta}(x')\right|,
\left|g^*_{\tilde\sigma(\cdot),\delta}(x)-g^*_{\tilde\sigma(\cdot),\delta}(x')\right|\right\}
\le \overline g(\eta),
\end{equation}
where $\lim_{\eta\to 0} \overline g(\eta)=0$ and, for all $\eta>0$,
\begin{equation}\label{x_sequence_condition}
\min_{1\le i\le n} \sum_{j=1}^n I\left(d_X(x_j,x_i)\le \eta\right)\to\infty.
\end{equation}
We also assume that the estimator $\hat f$ used to form the variance estimate
satisfies the uniform convergence condition
\begin{equation}\label{fhat_uniform_convergence_condition}
\max_{1\le i\le n} |\hat f(x_i)-f(x_i)|
  \underset{\mathcal{F},\mathcal{Q}_n}{\overset{p}{\to}} 0.
\end{equation}
Finally, we impose conditions on the moments of the error distribution. Suppose
that there exist $K$ and $\eta>0$ such that, for all $n$, $Q\in\mathcal{Q}_n$,
the errors $\{u_i\}_{i=1}^n$ are independent with, for each $i$,
\begin{equation}\label{epsilon_clt_condition}
1/K\le \sigma_Q^2(x_i)\le K
\text{ and }
E_{Q}|u_i|^{2+\eta}\le K.
\end{equation}

In cases where function class $\mathcal{F}$ imposes smoothness on $f$,
(\ref{f_uniform_continuity_condition}) will often follow directly from the
definition of $\mathcal{F}$. For example, it holds for the Lipschitz class
$\{f\colon |f(x)-f(x')|\le C d_X(x,x')\}$. The condition
(\ref{x_sequence_condition}) will hold with probability one if the $x_i$'s are
sampled from a distribution with density bounded away from zero on a
sufficiently regular bounded support. The condition
(\ref{fhat_uniform_convergence_condition}) will hold under regularity conditions
for a variety of choices of $\hat f$. It is worth noting that smoothness
assumptions on $\mathcal{F}$ needed for this assumption are typically weaker
than those needed for asymptotic equivalence with Gaussian white noise. For
example, if $\mathcal{X}=\mathbb{R}^{k}$ with the Euclidean norm,
(\ref{f_uniform_continuity_condition}) will hold automatically for Hölder
classes with exponent less than or equal to 1, while equivalence with Gaussian
white noise requires that the exponent be greater than $k/2$
\citep[see][]{brown_asymptotic_1998}. Furthermore, we do not require any
explicit characterization of the limiting form of the optimal CI\@. In
particular, we do not require that the weights for the optimal estimator
converge to a limiting optimal kernel or efficient influence function.

The condition (\ref{epsilon_clt_condition}) is used to verify a Lindeberg
condition for the central limit theorem used to obtain (\ref{epsilon_clt_eq}),
which we do in the next lemma.
\begin{lemma}\label{lindeberg_clt_lemma}
  Let $Z_{n,i}$ be a triangular array of independent random variables and let
  $a_{n,j}$, $1\le j\le n$ be a triangular array of constants. Suppose that
  there exist constants $K$ and $\eta>0$ such that, for all $i$,
\begin{align*}
1/K\le \sigma_{n,i}^2 \le K
\text{ and }
E |Z_{n,i}|^{2+\eta} \le K
\end{align*}
where $\sigma_{n,i}^2=EZ_{n,i}^2$, and that
\begin{align*}
\lim_{n\to\infty}\frac{\max_{1\le j\le n} a_{n,j}^2}{\sum_{j=1}^n a_{n,j}^2}=0.
\end{align*}
Then
\begin{align*}
  \frac{\sum_{i=1}^n a_{n,i}Z_{n,i}}{\sqrt{\sum_{i=1}^{n}a_{n,i}^2\sigma_{n,i}^2}}
  \stackrel{d}{\to} \mathcal{N}(0,1).
\end{align*}
\end{lemma}
\begin{proof}
We verify the conditions of the Lindeberg-Feller theorem as stated on p. 116 in
\citet{durrett_probability:_1996}, with $X_{n,i}=a_{n,i}Z_{n,i}/\sqrt{\sum_{j=1}^n a_{n,j}^2\sigma_j^2}$.  To verify the Lindeberg condition, note that
\begin{align*}
&\sum_{i=1}^n E\left(|X_{n,m}|^2\1{|X_{n,m}|>\varepsilon}\right)
=\frac{\sum_{i=1}^n E\left[|a_{n,i}Z_{n,i}|^2I\left(|a_{n,i}Z_{n,i}|>\varepsilon \sqrt{\sum_{j=1}^n a_{n,j}^2\sigma_j^2}\right)\right]}
  {\sum_{i=1}^n a_{n,i}^2\sigma_{n,i}^2}  \\
&\le \frac{\sum_{i=1}^n E\left(|a_{n,i}Z_{n,i}|^{2+\eta}\right)}
  {\varepsilon^\eta \left(\sum_{i=1}^n a_{n,i}^2\sigma_{n,i}^2\right)^{1+\eta/2}}
\le \frac{K^{2+\eta/2}}{\varepsilon^\eta}\frac{\sum_{i=1}^n |a_{n,i}|^{2+\eta}}
  {\left(\sum_{i=1}^n a_{n,i}^2\right)^{1+\eta/2}}
\le \frac{K^{2+\eta/2}}{\varepsilon^\eta}
\left(\frac{\max_{1\le i\le n} a_{n,i}^2}
  {\sum_{i=1}^n a_{n,i}^2}\right)^{1+\eta/2}.
\end{align*}
This converges to zero under the conditions of the lemma.
\end{proof}

\begin{theorem}\label{generic_estimator_thm}
  Let $\hat L_{\delta,\tilde\sigma(\cdot)}$ and
  $\widehat{\text{se}}^2_{\delta,\tilde\sigma(\cdot),n}$ be defined above.
  Suppose that, for each $n$, $f^*_{\tilde\sigma(\cdot),\delta}$,
  $g^*_{\tilde\sigma(\cdot),\delta}$ achieve the modulus under
  $\tilde\sigma(\cdot)$ with
  $\|K_{\tilde\sigma(\cdot),n}(g^*_{\tilde\sigma(\cdot),\delta}-f^*_{\tilde\sigma(\cdot),\delta})\|=\delta$,
  and that (\ref{f_uniform_continuity_condition}) and
  (\ref{x_sequence_condition}) hold. Suppose the errors
  satisfy~\eqref{epsilon_clt_condition} and are independent over $i$ for all $n$
  and $Q\in\mathcal{Q}_n$. Then~\eqref{epsilon_clt_eq} holds. If, in addition,
  the estimator $\hat f$ satisfies (\ref{fhat_uniform_convergence_condition}),
  then~\eqref{se_convergence_eq} holds with $\widehat{\text{se}}_n$ given by
  $\widehat{\text{se}}_{\delta,\tilde\sigma(\cdot),n}$.
\end{theorem}
\begin{proof}
  Condition~\eqref{epsilon_clt_eq} will follow by applying
  Lemma~\ref{lindeberg_clt_lemma} to show convergence under arbitrary sequences
  $Q_n\in\mathcal{Q}_n$ so long as
\begin{equation*}
  \frac{\max_{1\le i\le n}
    (g^*_{\tilde\sigma(\cdot),\delta}(x_i)-f^*_{\tilde\sigma(\cdot),\delta}(x_i))^2/\tilde\sigma(x_i)^4}
  {\sum_{i=1}^n(f^*_{\tilde\sigma(\cdot),\delta}(x_i)
    -g^*_{\tilde\sigma(\cdot),\delta}(x_i))^2/\tilde\sigma(x_i)^4}
  \to 0.
\end{equation*}
Since the denominator is bounded from below by $\delta^2/\max_{1\le i\le
  n}\tilde\sigma^2(x_i)$, and $\tilde\sigma^2(x_i)$ is bounded away from $0$ and
$\infty$ over $i$, it suffices to show that $\max_{1\le i\le n}
(g^*_{\tilde\sigma(\cdot),\delta}(x_i)-f^*_{\tilde\sigma(\cdot),\delta}(x_i))^2\to
0$. To this end, suppose, to the contrary, that there exists some $c>0$ such
that $\max_{1\le i\le n}
(g^*_{\tilde\sigma(\cdot),\delta}(x_i)\allowbreak-f^*_{\tilde\sigma(\cdot),\delta}(x_i))^2>c^2$
infinitely often. Let $\eta$ be small enough so that $\overline g(\eta)\le c/4$.
Then, for $n$ such that this holds and $k_n$ achieving this maximum,
\begin{equation*}
\sum_{i=1}^n(g^*_{\tilde\sigma(\cdot),\delta}(x_i)-f^*_{\tilde\sigma(\cdot),\delta}(x_i))^2
\ge \sum_{i=1}^n (c-c/2)^2\1{d_X(x_i,x_{k_n})\le \eta}
\to \infty.
\end{equation*}
But this is a contradiction since
$\sum_{i=1}^n(g^*_{\tilde\sigma(\cdot),\delta}(x_i)-f^*_{\tilde\sigma(\cdot),\delta}(x_i))^2$
is bounded by a constant times
$\sum_{i=1}^n(g^*_{\tilde\sigma(\cdot),\delta}(x_i)-f^*_{\tilde\sigma(\cdot),\delta}(x_i))^2/\tilde\sigma^2(x_i)=\delta^2$.

To show convergence of
$\widehat{\text{se}}^2_{\delta,\tilde\sigma(\cdot),n}/\text{var}_Q(\hat L_{\delta,\tilde\sigma(\cdot)})$, note that
\begin{equation*}
\frac{\widehat{\text{se}}^2_{\delta,\tilde\sigma(\cdot),n}}{\text{var}_Q(\hat L_{\delta,\tilde\sigma(\cdot)})}-1
=\frac{\sum_{i=1}^n a_{n,i}\left[( y_i-\hat f(x_i))^2-\sigma_Q^2(x_i)\right]}
        {\sum_{i=1}^n a_{n,i}\sigma^2_Q(x_i)}
\end{equation*}
where $a_{n,i}=\frac{(g^*_{\tilde\sigma(\cdot),\delta}(x_i)-f^*_{\tilde\sigma(\cdot),\delta}(x_i))^2}{\tilde\sigma^4(x_i)}$.  Since the denominator is bounded from below by a constant times $\sum_{i=1}^n a_{n,i}\tilde\sigma^2(x_i)=\delta^2$, it suffices to show that the numerator, which can be written as
\begin{align*}
\sum_{i=1}^n a_{n,i}\left[u_i^2-\sigma_Q(x_i)^2\right]
+\sum_{i=1}^n a_{n,i}(f(x_i)-\hat f(x_i))^2
+2\sum_{i=1}^n a_{n,i} u_i(f(x_i)-\hat f(x_i)),
\end{align*}
converges in probability to zero uniformly over $f$ and $Q$. The second term is
bounded by a constant times $\max_{1\le i\le
  n}(f(x_i)-\hat{f}(x_i))^2\sum_{i=1}^{n}a_{n,i}\tilde\sigma^2(x_i)=\max_{1\le
  i\le n}(f(x_i)-\hat f(x_i))^2\delta^2$, which converges in probability to zero
uniformly over $f$ and $Q$ by assumption. Similarly, the last term is bounded by
$\max_{1\le i\le n}|f(x_i)-\hat f(x_i)|$ times $2\sum_{i=1}^n a_{n,i}|u_i|$, and
the expectation of the latter term is bounded uniformly over $\mathcal{F}$ and
$\mathcal{Q}$. Thus, the last term converges in probability to zero uniformly
over $f$ and $Q$ as well. For the first term in this display, an inequality of
\citet{von_bahr_inequalities_1965} shows that the expectation of the absolute
$1+\eta/2$ moment of this term is bounded by a constant times
\begin{align*}
  \sum_{i=1}^n
  a_{n,i}^{1+\eta/2}E_{Q}\left|u_i^2-\sigma_Q(x_i)^2\right|^{1+\eta/2} \le
  \left(\max_{1\le i\le n} a_{n,i}^{\eta/2}\right) \max_{1\le i\le
    n}E_Q\left|\varepsilon_i^2-\sigma_Q^2(x_i)\right|^{1+\eta/2}
  \sum_{i=1}^{n}a_{n,i},
\end{align*}
which converges to zero since $\max_{1\le i\le n} a_{n,i}\to 0$ as shown earlier
in the proof and $\sum_{i=1}^{n} a_{n,i}$ is bounded by a constant times
$\sum_{i=1}^{n} a_{n,i}\tilde\sigma^2(x_i)=\delta^2$.
\end{proof}

If the variance function used by the researcher is correct up to scale (for
example, if the variance function is known to be constant), the one-sided
confidence intervals in (\ref{generic_estimator_thm}) will be asymptotically
optimal for some level $\beta$, which depends on $\delta$ and the magnitude of
the true error variance relative to the one used by the researcher. We record
this as a corollary.

\begin{corollary}\label{generic_optimality_corollary}
  If, in addition to the conditions in Theorem~\ref{generic_estimator_thm},
  $\sigma_Q^2(x)=\sigma^2\cdot\tilde\sigma^2(x)$ for all $n$ and
  $Q\in\mathcal{Q}_n$, then, letting $\beta=\Phi(\delta/\sigma-z_{1-\alpha})$,
  no CI satisfying~\eqref{asym_one_sided_coverage_eq} can
  satisfy~\ref{one_side_asym_minimax_eq} with the constant $1$ replaced by a
  strictly smaller constant on the right-hand side.
\end{corollary}

\section{Asymptotics for the Modulus and Efficiency Bounds}%
\label{sec:sap:asym_efficiency_bounds}

As discussed in \Cref{sec:general_results}, asymptotic relative efficiency
comparisons can often be performed by calculating the limit of the scaled
modulus. Here, we state some lemmas that can be used to obtain asymptotic
efficiency bounds and limiting behavior of the value of $\delta$ that optimizes
a particular performance criterion. We use these results in the proof of
Theorem~\ref{rd_optimal_estimator_thm_main_text} in \Cref{sec:sap:rd_asym}.

Before stating these results, we recall the characterization of minimax affine
performance given in \citet{donoho94}. To describe the results, first consider
the normal model $Z\sim \mathcal{N}(\mu,1)$ where $\mu\in[-\tau,\tau]$. The
minimax affine mean squared error for this problem is
\begin{align*}
  \rho_A(\tau)=\min_{\text{$\delta(Y)$ affine}}\max_{\mu\in[-\tau,\tau]}
  E_\mu(\delta(Y)-\mu)^2.
\end{align*}
The solution is achieved by shrinking $Y$ toward $0$, namely
$\delta(Y)=c_{\rho}(\tau)Y$, with $c_{\rho}(\tau)=\tau^2/(1+\tau^2)$, which
gives $\rho_A(\tau)=\tau^2/(1+\tau^2)$. The length of the smallest fixed-length
affine $100\cdot (1-\alpha)\%$ confidence interval is
\begin{align*}
  \chi_{A,\alpha}(\tau)=\min \left\{\chi \colon \text{there exists $\delta(Y)$
      affine s.t. }
    \inf_{\mu\in[-\tau,\tau]}P_\mu(|\delta(Y)-\mu|\le \chi)\ge
    1-\alpha\right\}.
\end{align*}
The solution is achieved at some $\delta(Y)=c_{\chi}(\tau)Y$, and it is
characterized in \citet{drees99}.

Using these definitions, the minimax
affine root MSE is given by
\begin{equation*}
\sup_{\delta>0}\frac{\omega(\delta)}{\delta}
\sqrt{\rho_A\left(\frac{\delta}{2\sigma}\right)}\sigma,
\end{equation*}
and the MSE optimal estimate is given by $\hat L_{\delta,\chi}$ where $\chi$ maximizes the above display.
Similarly, the optimal fixed-length affine CI has half-length
\begin{equation*}
\sup_{\delta>0}\frac{\omega(\delta)}{\delta}
\chi_{A,\alpha}\left(\frac{\delta}{2\sigma}\right)\sigma,
\end{equation*}
and is centered at $\hat L_{\delta_\chi}$ where $\delta_\chi$ maximizes the
above display (it follows from our results and those of \citealt{donoho94} that
this leads to the same value of $\delta_\chi$ as the one obtained by minimizing
CI length as described in \Cref{sec:two-sided-cis_main}).

The results below give the limiting behavior of these quantities
as well as the bound on expected length in
Corollary~\ref{th:centrosymmetric_adaptation_twosided} under pointwise
convergence of a sequence of functions $\omega_n(\delta)$ that satisfy the
conditions of a modulus scaled by a sequence of constants.

\begin{lemma}\label{modulus_convergence_lemma}
  Let $\omega_n(\delta)$ be a sequence of concave nondecreasing nonnegative
  functions on $\hor{0,\infty}$ and let $\omega_\infty(\delta)$ be a concave
  nondecreasing function on $\hor{0,\infty}$ with range $\hor{0,\infty}$. Then
  the following are equivalent.
\begin{itemize}
\item[(i)] For all $\delta> 0$, $\lim_{n\to\infty} \omega_n(\delta)=\omega_\infty(\delta)$.

\item[(ii)] For all $b\in (0,\infty)$, $b$ is in the range of $\omega_n$ for large enough $n$, and $\lim_{n\to\infty} \omega_n^{-1}(b)=\omega_\infty^{-1}(b)$.

\item[(iii)] For any $\overline \delta>0$, $\lim_{n\to\infty}\sup_{\delta\in[0,\overline \delta]}|\omega_n(\delta)-\omega_\infty(\delta)|=0$.
\end{itemize}
\end{lemma}
\begin{proof}
  Clearly $(iii)\Longrightarrow (i)$. To show $(i)\Longrightarrow (iii)$, given
  $\varepsilon>0$, let $0<\delta_1<\delta_2<\dotsb <\delta_k=\overline\delta$ be
  such that $\omega(\delta_j)-\omega(\delta_{j-1})\le \varepsilon$ for each $j$.
  Then, using monotonicity of $\omega_n$ and $\omega_\infty$, we have
  $\sup_{\delta\in[0,\delta_1]}|\omega_n(\delta)-\omega_\infty(\delta)| \le
  \max\left\{|\omega_n(\delta_1)|,|\omega_n(0)-\omega_\infty(\delta_1)|\right\}
  \to \omega_\infty(\delta_1)$ and
\begin{align*}
&\sup_{\delta\in[\delta_{j-1},\delta_j]}|\omega_n(\delta)-\omega_\infty(\delta)|
\le \max\left\{|\omega_n(\delta_j)-\omega_\infty(\delta_{j-1})|,
  |\omega_n(\delta_{j-1})-\omega_\infty(\delta_{j})|\right\}  \\
&\to |\omega_\infty(\delta_{j-1})-\omega_\infty(\delta_{j})|\le \varepsilon.
\end{align*}
The result follows since $\varepsilon$ can be chosen arbitrarily small. To show
$(i)\Longrightarrow (ii)$, let $\delta_\ell$ and $\delta_u$ be such that
$\omega_\infty(\delta_\ell)<b<\omega_\infty(\delta_u)$. For large enough $n$, we
will have $\omega_n(\delta_\ell)<b<\omega_n(\delta_u)$ so that $b$ will be in
the range of $\omega_n$ and $\delta_\ell<\omega_n^{-1}(b)<\delta_u$. Since
$\omega_\infty$ is strictly increasing, $\delta_\ell$ and $\delta_u$ can be
chosen arbitrarily close to $\omega_\infty^{-1}(b)$, which gives the result. To
show $(ii)\Longrightarrow (i)$, let $b_\ell$ and $b_u$ be such that
$\omega_\infty^{-1}(b_\ell)<\delta<\omega_\infty^{-1}(b_u)$. Then, for large
enough $n$, $\omega_n^{-1}(b_\ell)<\delta<\omega_n^{-1}(b_u)$, so that
$b_\ell<\omega_n(\delta)<b_u$, and the result follows since $b_\ell$ and $b_u$
can be chosen arbitrarily close to $\omega_\infty(\delta)$ since
$\omega_\infty^{-1}$ is strictly increasing.
\end{proof}

\begin{lemma}\label{modulus_objective_convergence_lemma}
Suppose that the conditions of Lemma~\ref{modulus_convergence_lemma} hold with
$\lim_{\delta\to 0}\omega_\infty(\delta)=0$ and
$\lim_{\delta\to\infty}\omega_\infty(\delta)/\delta=0$.
Let $r$ be a nonnegative function with $0\le r(\delta/2)\le \overline r \min\{\delta,1\}$ for some $\overline r<\infty$.  Then
\begin{align*}
\lim_{n\to\infty}\sup_{\delta>0}\frac{\omega_n(\delta)}{\delta}r\left(\frac{\delta}{2}\right)
=\sup_{\delta>0}\frac{\omega_\infty(\delta)}{\delta}r\left(\frac{\delta}{2}\right).
\end{align*}
If, in addition $r$ is continuous, $\frac{\omega_\infty(\delta)}{\delta}r\left(\frac{\delta}{2}\right)$ has a unique maximizer $\delta^*$, and, for each $n$, $\delta_n$ maximizes $\frac{\omega_n(\delta)}{\delta}r\left(\frac{\delta}{2}\right)$, then $\delta_n\to\delta^*$ and $\omega_n(\delta_n)\to\omega_\infty(\delta^*)$.
In addition, for any $\sigma>0$ and $0<\alpha<1$ and $Z$ a standard normal variable,
\begin{align*}
\lim_{n\to \infty} (1-\alpha)E[\omega_n(2\sigma (z_{1-\alpha}-Z))|Z\le z_{1-\alpha}]
=(1-\alpha)E[\omega_\infty(2\sigma (z_{1-\alpha}-Z))|Z\le z_{1-\alpha}].
\end{align*}
\end{lemma}
\begin{proof}
We will show that the objective can be made arbitrarily small for $\delta$ outside of $[\underline \delta,\overline \delta]$ for $\underline\delta$ small enough and $\overline\delta$ large enough, and then use uniform convergence over $[\underline \delta,\overline \delta]$.  First, note that, if we choose $\underline\delta<1$, then, for $\delta\le \underline\delta$,
\begin{align*}
\frac{\omega_n(\delta)}{\delta}r\left(\frac{\delta}{2}\right)
\le \omega_n(\delta)\overline r
\le \omega_n(\underline\delta)\overline r\to \omega_{\infty}(\underline\delta),
\end{align*}
which can be made arbitrarily small by making $\underline\delta$ small.
Since $\omega_n(\delta)$ is concave and nonnegative, $\omega_n(\delta)/\delta$ is nonincreasing, so, for $\delta>\overline\delta$,
\begin{align*}
\frac{\omega_n(\delta)}{\delta}r\left(\frac{\delta}{2}\right)
\le \frac{\omega_n(\delta)}{\delta}\overline r
\le \frac{\omega_n(\overline\delta)}{\overline\delta}\overline r
\to \frac{\omega_\infty(\overline\delta)}{\overline\delta}\overline r,
\end{align*}
which can be made arbitrarily small by making $\overline\delta$ large. Applying
Lemma~\ref{modulus_convergence_lemma} to show convergence over
$[\underline\delta,\overline\delta]$ gives the first claim. The second claim
follows since $\underline\delta$ and $\overline\delta$ can be chosen so that
$\delta_n\in[\underline\delta,\overline\delta]$ for large enough $n$ (the
assumption that
$\frac{\omega_\infty(\delta)}{\delta}r\left(\frac{\delta}{2}\right)$ has a
unique maximizer means that it is not identically zero), and uniform convergence
to a continuous function with a unique maximizer on a compact set implies
convergence of the sequence of maximizers to the maximizer of the limiting
function.

For the last statement, note that, by positivity and concavity of $\omega_n$, we have, for large enough $n$,
$0\le \omega_n(\delta)\le \omega_n(1)\max\{\delta,1\}\le (\omega_\infty(1)+1)\max\{\delta,1\}$ for all $\delta>0$.
The result then follows from the dominated convergence theorem.
\end{proof}

\begin{lemma}\label{modulus_derivative_convergence_lemma}
  Let $\omega_n(\delta)$ be a sequence of nonnegative concave functions on
  $\hor{0,\infty}$ and let $\omega_\infty(\delta)$ be a nonnegative concave
  differentiable function on $\hor{0,\infty}$. Let $\delta_0>0$ and suppose that
  $\omega_n(\delta)\to\omega_\infty(\delta)$ for all $\delta$ in a neighborhood
  of $\delta_0$. Then, for any sequence $d_n\in \partial \omega_n(\delta_0)$, we
  have $d_n\to \omega_\infty'(\delta_0)$. In particular, if
  $\omega_n(\delta)\to\omega_\infty(\delta)$ in a neighborhood of $\delta_0$ and
  $2\delta_0$, then
$\frac{\omega_n(2\delta_0)}{\omega_n(\delta_0)+\delta_0 \omega'_n(\delta_0)}
\to \frac{\omega_\infty(2\delta_0)}{\omega_\infty(\delta_0)+\delta_0 \omega'_\infty(\delta_0)}$.
\end{lemma}
\begin{proof}
By concavity, for $\eta>0$ we have
$[\omega_n(\delta_0)-\omega_n(\delta_0-\eta)]/\eta
\ge d_n\ge [\omega_n(\delta_0+\eta)-\omega_n(\delta_0)]/\eta$.  For small enough $\eta$, the left and right-hand sides converge, so that
$[\omega_\infty(\delta_0)-\omega_\infty(\delta_0-\eta)]/\eta
\ge \limsup_n d_n\ge \liminf_n d_n\ge [\omega_\infty(\delta_0+\eta)-\omega_\infty(\delta_0)]/\eta$.  Taking the limit as $\eta\to 0$ gives the result.
\end{proof}

\section{Asymptotics for Regression Discontinuity}%
\label{sec:sap:rd_asym}

This section proves Theorem~\ref{rd_optimal_estimator_thm_main_text}. We first
give a general result for linear estimators under high-level conditions in
\Cref{sec:sap:rd_asym_general_results}. We then consider local polynomial
estimators in \Cref{sec:sap:rd_local_polynomial} and optimal estimators with a
plug-in variance estimate in \Cref{sec:sap:rd_opt_est}.
Theorem~\ref{rd_optimal_estimator_thm_main_text} follows immediately from the
results in these sections.

Throughout this section, we consider the RD setup where the error distribution
may be non-normal as in \Cref{sec:asympt-valid-optim}, using the conditions in
that section. We repeat these conditions here for convenience.

\begin{assumption}\label{rd_xs_assump}
  For some $p_{X,+}(0)>0$ and $p_{X,-}(0)>0$, the sequence $\{x_i\}_{i=1}^{n}$
  satisfies
  $\frac{1}{n h_n}\sum_{i=1}^{n}m(x_i/h_n)\1{x_i>0} \to p_{X,+}(0)\int_0^\infty
  m(u)\, du$ and
  $\frac{1}{n h_n}\sum_{i=1}^{n}m(x_i/h_n)\1{x_{i} < 0} \to
  p_{X,-}(0)\int_{-\infty}^0 m(u)\, du$ for any bounded function $m$ with
  bounded support and any $h_n$ with
  $0<\liminf_n h_n n^{1/(2p+1)}\le \limsup_n h_n n^{1/(2p+1)}<\infty$.
\end{assumption}

\begin{assumption}\label{rd_errors_assump}
For some $\sigma(x)$ with $\lim_{x\downarrow 0}\sigma(x)=\sigma_+(0)>0$
and $\lim_{x\uparrow 0}\sigma(x)=\sigma_-(0)>0$, we have
\begin{itemize}
\item[(i)] the $u_i$s are independent under any $Q\in\mathcal{Q}_n$ with $E_Q u_i=0$,
$\text{var}_Q(u_i)=\sigma^2(x_i)$

\item[(ii)] for some $\eta>0$, $E_Q|u_i|^{2+\eta}$ is bounded uniformly over $n$ and $Q\in\mathcal{Q}_n$.
\end{itemize}
\end{assumption}

Theorem~\ref{rd_optimal_estimator_thm_main_text} considers affine estimators
that are optimal under the assumption that the variance function is given by
$\hat\sigma_+\1{x>0}+\hat\sigma_-\1{x<0}$, which covers the plug-in optimal
affine estimators used in our application. Here, it will be convenient to
generalize this slightly by considering the class of affine estimators that are
optimal under a variance function $\tilde\sigma(x)$, which may be misspecified
or data-dependent, but which may take some other form.
We consider two possibilities for how $\tilde\sigma(\cdot)$ is calibrated.

\begin{assumption}\label{estimated_tilde_sigma_assump}
$\tilde\sigma(x)=\hat\sigma_+\1{x>0}+\hat\sigma_- \1{x<0}$
where
$\hat\sigma_+\underset{\mathcal{F},\mathcal{Q}_n}{\overset{p}{\to}} \tilde\sigma_+(0)>0$
and
$\hat\sigma_-\underset{\mathcal{F},\mathcal{Q}_n}{\overset{p}{\to}} \tilde\sigma_-(0)>0$.
\end{assumption}

\begin{assumption}\label{prespecified_tilde_sigma_assump}
$\tilde\sigma(x)$ is a deterministic function with $\lim_{x\downarrow 0}\tilde \sigma(x)=\tilde \sigma_-(0)>0$
and
$\lim_{x\uparrow 0}\tilde \sigma(x)=\tilde \sigma_+(0)>0$.
\end{assumption}

Assumption~\ref{estimated_tilde_sigma_assump} corresponds to the estimate of the
variance function used in the application. It generalizes
Assumption~\ref{estimated_tilde_sigma_assump_main_text} slightly by allowing
$\hat\sigma_+$ and $\hat\sigma_-$ to converge to something other than the left-
and right-hand limits of the true variance function.
Assumption~\ref{prespecified_tilde_sigma_assump} is used in deriving bounds
based on infeasible estimates that use the true variance function.

Note that, under Assumption~\ref{estimated_tilde_sigma_assump},
$\tilde\sigma_+(0)$ is defined as the probability limit of $\hat\sigma_+$ as
$n\to\infty$, and does not give the limit of $\tilde\sigma(x)$ as $x\downarrow
0$ (and similarly for $\tilde\sigma_-(0)$). We use this notation so that certain
limiting quantities can be defined in the same way under each of the
Assumptions~\ref{prespecified_tilde_sigma_assump}
and~\ref{estimated_tilde_sigma_assump}.

\subsection{General Results for Kernel Estimators}%
\label{sec:sap:rd_asym_general_results}

We first state results for affine estimators where the weights asymptotically take a kernel form.
We consider a sequence of estimators of the form
\begin{align*}
\hat L=\frac{\sum_{i=1}^n k_n^+(x_i/h_n)\1{x_i>0} y_i}
         {\sum_{i=1}^n k_n^+(x_i/h_n)\1{x_i>0}}
  -\frac{\sum_{i=1}^n k_n^-(x_i/h_n)\1{x_i<0} y_i}
         {\sum_{i=1}^n k_n^-(x_i/h_n)\1{x_i<0}}
\end{align*}
where $k_n^+$ and $k_n^-$ are sequences of kernels.
The assumption that the same bandwidth is used on each side of the discontinuity is a normalization: it can always be satisfied by redefining one of the kernels $k_n^+$ or $k_n^-$.
We make the following assumption on the sequence of kernels.

\begin{assumption}\label{rd_kern_assump}
The sequences of kernels and bandwidths $k^+_n$ and $h_n$ satisfy
\begin{itemize}
\item[(i)] $k^+_n$ has support bounded uniformly over $n$. For a bounded
  kernel $k^+$ with $\int k^+(u)\, du>0$, we have
  $\sup_x\left|k^+_n(x)-k^{+}(x)\right|\to 0$%

\item[(ii)]
  $\frac{1}{n h_n}\sum_{i=1}^n k^+_n(x_i/h_n)\1{x_i>0}
  (x_i,\ldots,x_i^{p-1})'=0$ for each $n$

\item[(iii)] $h_{n}n^{1/(2p+1)}\to h_{\infty}$ for some constant $0<h_{\infty}<\infty$,
\end{itemize}
and similarly for $k^-_n$ for some $k^-$.
\end{assumption}



Let
\begin{align*}
&\overline{\text{bias}}_n
=\frac{\sum_{i=1}^n |k_n^+(x_i/h_n)|\1{x_i>0}C|x_i|^p}
         {\sum_{i=1}^n k_n^+(x_i/h_n)\1{x_i>0}}
  +\frac{\sum_{i=1}^n |k_n^-(x_i/h_n)|\1{x_i<0} C|x_i|^p}
         {\sum_{i=1}^n k_n^-(x_i/h_n)\1{x_i<0}}  \\
&=Ch_n^{p}\left(\frac{\sum_{i=1}^n |k_n^+(x_i/h_n)|\1{x_i>0}|x_i/h_n|^p}
         {\sum_{i=1}^n k_n^+(x_i/h_n)\1{x_i>0}}
  +\frac{\sum_{i=1}^n |k_n^-(x_i/h_n)|\1{x_i<0}|x_i/h_n|^p}
         {\sum_{i=1}^n k_n^-(x_i/h_n)\1{x_i<0}}\right)
\end{align*}
and
\begin{align*}
v_n&=\frac{\sum_{i=1}^n k_n^+(x_i/h_n)^2\1{x_i>0}\sigma^2(x_i)}
         {\left[\sum_{i=1}^n k_n^+(x_i/h_n)\1{x_i>0}\right]^2}
  +\frac{\sum_{i=1}^n k_n^-(x_i/h_n)^2\1{x_i<0}\sigma^2(x_i)}
         {\left[\sum_{i=1}^n k_n^-(x_i/h_n)\1{x_i<0}\right]^2}  \\
&=\frac{1}{nh_n}
\left(\frac{\frac{1}{nh_n}\sum_{i=1}^n k_n^+(x_i/h_n)^2\1{x_i>0}\sigma^2(x_i)}
         {\left[\frac{1}{nh_n}\sum_{i=1}^n k_n^+(x_i/h_n)\1{x_i>0}\right]^2}
  +\frac{\frac{1}{nh_n}\sum_{i=1}^n k_n^-(x_i/h_n)^2\1{x_i<0}\sigma^2(x_i)}
         {\left[\frac{1}{nh_n}\sum_{i=1}^n k_n^-(x_i/h_n)\1{x_i<0}\right]^2}\right)
\end{align*}
Note that $v_n$ is the (constant over $Q\in\mathcal{Q}_n$) variance of $\hat L$,
and that, by arguments in \Cref{rd_bias_sec},
$\overline{\text{bias}}_n=\sup_{f\in\mathcal{F}}(E_{f,Q}\hat{L}-Lf)=-\inf_{f\in\mathcal{F}}
(E_{f,Q}\hat{L}-Lf)$ for any $Q\in\mathcal{Q}_n$ under
Assumption~\ref{rd_kern_assump}~(ii).

To form a feasible CI, we need an estimate of $v_n$. While the results below go
through with any variance estimate that is consistent uniformly over $f,\mathcal{Q}_n$, we propose
one here for concreteness. For a possibly data-dependent guess
$\tilde \sigma(\cdot)$ of the variance function, let $\tilde v_n$ denote $v_n$
with $\sigma(\cdot)$ replaced by $\tilde \sigma(\cdot)$. We record the limiting
behavior of $\overline{\text{bias}}_n$, $v_n$ and $\tilde v_n$ in the following
lemma. Let
\begin{align*}
&\overline{\text{bias}}_\infty
=Ch_\infty^p
\left(\frac{\int_{0}^\infty |k^+(u)||u|^p\, du}
         {\int_{0}^\infty k^+(u)\, du}
  +\frac{\int_{-\infty}^0 |k^-(u)||u|^p\, du}
         {\int_{-\infty}^0 k^-(u)\, du}\right)
\end{align*}
and
\begin{align*}
&v_\infty
=\frac{1}{h_\infty}
\left(\frac{\sigma^2_+(0)\int_{0}^\infty k^+(u)^2\, du}
         {p_{X,+}(0)\left[\int_0^\infty k^+(u)\, du\right]^2}
  +\frac{\sigma^2_-(0)\int_{-\infty}^0 k^-(u)^2\, du}
         {p_{X,-}(0)\left[\int_{-\infty}^0 k^-(u)\, du\right]^2}\right).
\end{align*}

\begin{lemma}\label{rd_bias_var_limit_lemma}
  Suppose that Assumption~\ref{rd_xs_assump} holds. If
  Assumption~\ref{rd_kern_assump} also holds, then
  $\lim_{n\to\infty}
  n^{p/(2p+1)}\overline{\text{bias}}_n=\overline{\text{bias}}_\infty$ and
  $\lim_{n\to\infty}n^{2p/(2p+1)}v_n=v_\infty$. If, in addition,
  $\tilde\sigma(\cdot)$ satisfies Assumption~\ref{estimated_tilde_sigma_assump}
  or Assumption~\ref{prespecified_tilde_sigma_assump} with
  $\tilde\sigma_+(0)=\sigma_+(0)$ and $\tilde\sigma_-(0)=\sigma_-(0)$, then
  $n^{2p/(2p+1)}\tilde v_n
  \underset{\mathcal{F},\mathcal{Q}_n}{\overset{p}{\to}}v_\infty$ under
  Assumption~\ref{estimated_tilde_sigma_assump} and
  $\lim_{n\to\infty}n^{2p/(2p+1)}\tilde v_n=v_\infty$ under
  Assumption~\ref{prespecified_tilde_sigma_assump}.
\end{lemma}
\begin{proof}
  The results follow from applying the convergence in
  Assumption~\ref{rd_xs_assump} along with Assumption~\ref{rd_kern_assump}(i) to
  the relevant terms in $\overline{\text{bias}}_n$ and $\tilde v_n$.
\end{proof}

\begin{theorem}\label{rd_limit_high_level_thm}
  Suppose that Assumptions~\ref{rd_xs_assump}, \ref{rd_errors_assump}
  and~\ref{rd_kern_assump} hold, and that $\tilde v_n$ is formed using a
  variance function $\tilde\sigma(\cdot)$ that satisfies
  Assumption~\ref{estimated_tilde_sigma_assump}
  or~\ref{prespecified_tilde_sigma_assump} with $\tilde\sigma_+(0)=\sigma_+(0)$
  and $\tilde\sigma_-(0)=\sigma_-(0)$. Then
\begin{equation*}
  \liminf_{n\to\infty}
  \inf_{f\in\mathcal{F}_{RDT,p}(C),Q\in\mathcal{Q}_n}
  P_{f,Q}\left(Lf\in\left\{\hat L\pm \cv_\alpha\left(\overline{\text{bias}}_n/\tilde v_n\right)\sqrt{\tilde v_n}\right\}\right)
  \ge 1-\alpha
\end{equation*}
and,
letting $\hat c=\hat L-\overline{\text{bias}}_n-z_{1-\alpha}\sqrt{\tilde v_n}$,
\begin{equation*}
  \liminf_{n\to\infty}\inf_{f\in\mathcal{F}_{RDT,p}(C),Q\in\mathcal{Q}_n}
  P_{f,Q}\left(Lf
    \in\hor{\hat c,\infty}
  \right)\ge 1-\alpha.
\end{equation*}
In addition,
$n^{p/(2p+1)}\cv_\alpha(\overline{\text{bias}}_n/\tilde v_{n})\tilde v_n
\underset{\mathcal{F},\mathcal{Q}_n}{\overset{p}{\to}}
\cv_\alpha(\overline{\text{bias}}_\infty/v_{\infty})v_\infty$ if
$\tilde\sigma(\cdot)$ satisfies Assumption~\ref{estimated_tilde_sigma_assump}
and
$n^{p/(2p+1)}\cv_\alpha(\overline{\text{bias}}_n/\tilde v_n)\tilde{v}_n\to
\cv_\alpha(\overline{\text{bias}}_\infty/v_\infty) v_\infty$ if
$\tilde\sigma(\cdot)$ satisfies
Assumption~\ref{prespecified_tilde_sigma_assump}. The minimax $\beta$ quantile
of the one-sided CI satisfies
\begin{align*}
\limsup_{n\to\infty} n^{p/(2p+1)}\sup_{f\in\mathcal{F}_{RDT,p}(C),Q\in\mathcal{Q}_n} q_{f,Q,\beta}(Lf-\hat c)
\le 2\overline{\text{bias}}_\infty+(z_\beta+z_{1-\alpha})\sqrt{v_\infty}.
\end{align*}
The worst-case $\beta$ quantile over $\mathcal{F}_{RDT,p}(0)$ satisfies
\begin{align*}
\limsup_{n\to\infty} n^{p/(2p+1)}\sup_{f\in\mathcal{F}_{RDT,p}(0),Q\in\mathcal{Q}_n} q_{f,Q,\beta}(Lf-\hat c)
\le \overline{\text{bias}}_\infty+(z_\beta+z_{1-\alpha})\sqrt{v_\infty}.
\end{align*}
Furthermore, the same holds with $\hat L$, $\overline{\text{bias}}_n$ and $\tilde v_n$ replaced by any $\hat L^*$, $\overline{\text{bias}}_n^*$ and $\tilde v_n^*$ such that
\begin{align*}
n^{p/(2p+1)}\left(\hat L-\hat L^*\right) \underset{\mathcal{F},\mathcal{Q}_n}{\overset{p}{\to}} 0,
& &
n^{p/(2p+1)}\left(\overline{\text{bias}}_n-\overline{\text{bias}}_n^*\right) \underset{\mathcal{F},\mathcal{Q}_n}{\overset{p}{\to}} 0,
& &
\frac{\tilde v_n}{\tilde v_n^*}
\underset{\mathcal{F},\mathcal{Q}_n}{\overset{p}{\to}} 1.
\end{align*}

\end{theorem}
\begin{proof}
  We verify the conditions of Theorem~\ref{high_level_asymptotic_thm}.
  Condition~\eqref{se_convergence_eq} follows from
  Lemma~\ref{rd_bias_var_limit_lemma}. To verify~\eqref{epsilon_clt_eq}, note
  that $\hat{L}$ takes the general form in
  Theorem~\ref{high_level_asymptotic_thm} with $w_{n,i}$ given by
  $w_{n,i}=k_n^+(x_i/h_n)/\sum_{j=1}^n k_n^+(x_j/h_n)\1{x_j>0}$ for $x_i>0$ and
  $w_{n,i}=k_n^-(x_i/h_n)/\sum_{j=1}^n k_n^-(x_j/h_n)\cdot \1{x_j<0}$ for $x_i<0$. The
  uniform central limit theorem in~\eqref{epsilon_clt_eq} with $w_{n,i}$ taking
  this form follows from Lemma~\ref{lindeberg_clt_lemma}. This gives the
  asymptotic coverage statements.

  For the asymptotic formulas for excess length of the one-sided CI and length
  of the two-sided CI, we apply Theorem~\ref{high_level_asym_equivalence_thm}
  with $n^{-p/(2p+1)}\overline{\text{bias}}_\infty$ playing the role of
  $\widetilde{\overline{\text{bias}}}_n$ and $n^{-p/(2p+1)}v_\infty$ playing the
  role of $\widetilde{\text{se}}_n$. Finally, the last statement of the theorem
  is immediate from Theorem~\ref{high_level_asym_equivalence_thm}.
\end{proof}

\subsection{Local Polynomial Estimators}%
\label{sec:sap:rd_local_polynomial}

The $(p-1)$th order local polynomial estimator of $f_+(0)$ based on kernel
$k^*_+$ and bandwidth $h_{+,n}$ is given by
\begin{align*}
  \hat f_{+}(0) =
  &e_1'
    \left(\sum_{i=1}^{n} p(x_i/h_{+,n})p(x_i/h_{+,n})'k^*_+(x_i/h_{+,n})\1{x_i>0}
    \right)^{-1}  \\
  &\sum_{i=1}^n k^*_+(x_i/h_{+,n})\1{x_i>0}p(x_i/h_{+,n}) y_i
\end{align*}
where $e_1=(1,0,\ldots,0)'$ and $p(x)=(1,x,x^2,\ldots,x^{p-1})'$.
Letting the local polynomial estimator of $f_-(0)$ be defined analogously for some kernel $k^*_-$ and bandwidth $h_{-,n}$, the local polynomial estimator of $Lf=f_+(0)-f_-(0)$ is given by
\begin{align*}
\hat L=\hat f_+(0)-\hat f_-(0).
\end{align*}
This takes the form given in \Cref{sec:sap:rd_asym_general_results}, with
$h_n=h_{n,+}$,
\begin{align*}
  k_n^+(u) =e_1' \left(\frac{1}{nh_n}\sum_{i=1}^{n}
    p(x_i/h_{+,n})p(x_i/h_{+,n})'k^*_+(x_i/h_{+,n})\1{x_i>0}\right)^{-1}
  k^*_+(u)p(u)\1{u>0}
\end{align*}
and
\begin{align*}
  k_n^-(u) =
  &e_1'
    \left(\frac{1}{nh_n}\sum_{i=1}^{n} p(x_i/h_{-,n})p(x_i/h_{-,n})'k^*_+(x_i/h_{-,n})
    \1{x_i<0}\right)^{-1}  \\
  &k^*_+(u (h_{n,+}/h_{n,-}))p(u(h_{n,+}/h_{n,-}))\1{u<0}.
\end{align*}
Let $M^+$ be the $(p-1)\times(p-1)$ matrix with
$\int_0^\infty u^{j+k-2}k_+^*(u)$ as the $i,j$th entry, and let $M^-$ be the
$(p-1)\times(p-1)$ matrix with $\int_{-\infty}^0 u^{j+k-2}k_-^*(u)$ as the
$i,j$th entry. Under Assumption~\ref{rd_xs_assump}, for $k_+^*$ and $k_-^*$
bounded with bounded support,
$\frac{1}{nh_n}\sum_{i=1}^{n}p(x_i/h_{+,n})p(x_i/h_{+,n})'k^*_+(x_i/h_{+,n})\cdot
\1{x_i>0}\to M^+p_{X,+}(0)$ and similarly
$\frac{1}{nh_n}\sum_{i=1}^{n}
p(x_i/h_{-,n})p(x_i/h_{-,n})'k^*_+(x_i/h_{-,n})\cdot \1{x_i<0} \to
M^{-}p_{X,-}(0)$. Furthermore, Assumption~\ref{rd_kern_assump}~(ii) follows
immediately from the normal equations for the local polynomial estimator. This
gives the following result.

\begin{theorem}
  Let $k^*_+$ and $k_-^*$ be bounded and uniformly continuous with bounded
  support. Let $h_{n,+}n^{1/(2p+1)}\to h_\infty>0$ and suppose $h_{n,-}/h_{n,+}$
  converges to a strictly positive constant. Then
  Assumption~\ref{rd_kern_assump} holds for the local polynomial estimator so
  long as Assumption~\ref{rd_xs_assump} holds.
\end{theorem}

\subsection{Optimal Affine Estimators}%
\label{sec:sap:rd_opt_est}

We now consider the class of affine estimators that are optimal under the
assumption that the variance function is given by $\tilde \sigma(\cdot)$, which
satisfies either Assumption~\ref{estimated_tilde_sigma_assump} or
Assumption~\ref{prespecified_tilde_sigma_assump}. We use the same notation as in
\Cref{sec:addit-deta-rd}, except that $n$ and/or $\tilde\sigma(\cdot)$ are added
as subscripts for many of the objects under consideration to make the dependence
on $\{x_i\}_{i=1}^n$ and $\tilde\sigma(\cdot)$ explicit.

The modulus problem is given by Equation~\eqref{rd_modulus_eq} in
\Cref{rd_modulus_solution_sec} with $\tilde\sigma(\cdot)$ in place of
$\sigma(\cdot)$. We use $\omega_{\tilde\sigma(\cdot),n}(\delta)$ to denote the
modulus, or $\omega_{n}(\delta)$ when the context is clear. The corresponding
estimator $\hat L_{\delta,\tilde\sigma(\cdot)}$ is then given by
Equation~\eqref{Ldelta_rd_eq} in \Cref{rd_modulus_solution_sec} with
$\tilde\sigma(\cdot)$ in place of $\sigma(\cdot)$.

We will deal with the inverse modulus, and use
Lemma~\ref{modulus_convergence_lemma} to obtain results for the modulus itself.
The inverse modulus $\omega_{\tilde\sigma(\cdot),n}^{-1}(2b)$ is given by
Equation~\eqref{rd_inverse_modulus_eq} in \Cref{rd_modulus_solution_sec}, with
$\tilde\sigma^2(x_i)$ in place of $\sigma^2(x_i)$, and the solution takes the
form given in that section. Let $h_n=n^{-1/(2p+1)}$. We will consider a sequence
$b=b_n$, and will define $\tilde b_n=n^{p/(2p+1)} b_n=h_n^{-p} b_n$. Under
Assumption~\ref{prespecified_tilde_sigma_assump}, we will assume that
$\tilde{b}_n\to \tilde b_\infty$ for some $\tilde b_\infty>0$. Under
Assumption~\ref{estimated_tilde_sigma_assump}, we will assume that
$\tilde{b}_n\underset{\mathcal{F},\mathcal{Q}_n}{\overset{p}{\to}}
\tilde{b}_\infty$ for some $\tilde b_\infty>0$. We will then show that this
indeed holds for $2b_n=\omega_{\tilde\sigma(\cdot),n}(\delta_n)$ with $\delta_n$
chosen as in Theorem~\ref{rd_optimal_estimator_thm} below.

Let $\tilde b_n=n^{p/(2p+1)} b_n=h_n^{-p} b_n$,
$\tilde b_{-,n}=n^{p/(2p+1)} b_{-,n}=h_n^{-p} b_{-,n}$,
$\tilde{d}_{+,j,n}=n^{(p-j)/(2p+1)}d_{+,j,n}=h_n^{j-p}d_{+,j,n}$ and
$\tilde{d}_{-,j,n}=n^{(p-j)/(2p+1)}d_{-,j,n}=h_n^{j-p}d_{-,j,n}$ for
$j=1,\ldots,p-1$, where $b_n$, $b_{-,n}$, $d_{+,n}$, and $d_{-n}$ correspond to
the function $g_{b,C}$ that solves the inverse modulus problem, given in
\Cref{rd_modulus_solution_sec}. These values of $\tilde b_{+,n}$,
$\tilde b_{-,n}$, $\tilde d_{+,n}$ and $\tilde{d}_{-,n}$ minimize
$G_n(b_+,b_-,d_+,d_-)$ subject to $b_++b_-=\tilde b_n$ where, letting
$\mathcal{A}(x_{i},b,d)=h_n^{p} b+\sum_{j=1}^{p-1}h_n^{p-j}d_{j} x_i^j$,
\begin{multline*}
  G_n(b_+,b_-,d_+,d_-)
  =\\
  \sum_{i=1}^n\tilde\sigma^{-2}(x_i) \left(
    (\mathcal{A}(x_{i},b_{+},d_{+})-C\abs{x_{i}^{p}} )_{+}
    +\left(\mathcal{A}(x_{i},b_{+},d_{+})+C|x_i|^p\right)_-\right)^2\1{x_i>0}  \\
  +\sum_{i=1}^n\tilde\sigma^{-2}(x_i) \left(
    (\mathcal{A}(x_{i},b_{-},d_{-})-C\abs{x_{i}}^{p})
    +\left(\mathcal{A}(x_{i},b_{-},d_{-})+C|x_i|^p\right)_-\right)^2\1{x_i<0}  \\
  =\frac{1}{nh_n} \sum_{i=1}^n
  k_{\tilde\sigma(\cdot)}^+(x_i/h_n;b_+,d_+)^2\tilde\sigma^2(x_i)
  +\frac{1}{nh_n} \sum_{i=1}^n
  k_{\tilde\sigma(\cdot)}^-(x_i/h_n;b_-,d_-)^2\tilde\sigma^2(x_i)
\end{multline*}
with
\begin{align*}
  &k_{\tilde\sigma(\cdot)}^+(u;b,d)=\tilde\sigma^{-2}(u h_n)
  \left(\left(b+\sum_{j=1}^{p-1}d_{j} u^{j} -C|u|^p\right)_+-
    \left(b+\sum_{j=1}^{p-1}d_j u^j+C|u|^p\right)_-\right)\1{u>0},  \\
  &k_{\tilde\sigma(\cdot)}^-(u;b,d)=\tilde\sigma^{-2}(u h_n)
  \left(\left(b+\sum_{j=1}^{p-1}d_{j}
      u^j-C|u|^p\right)_+-\left(b+\sum_{j=1}^{p-1}d_{j} u^j
      +C|u|^p\right)_-\right)\1{u<0}.
\end{align*}
We use the notation $k^+_c$ for a scalar $c$ to denote
$k^+_{\tilde\sigma(\cdot)}$ where $\tilde\sigma(\cdot)$ is given by the constant
function $\tilde\sigma(x)=c$.

With these definitions, the estimator
$\hat{L}_{\delta,\tilde\sigma(\cdot)}$ with
$\omega_{\tilde\sigma(\cdot),n}(\delta)=2b_n$ takes the general kernel form in
\Cref{sec:sap:rd_asym_general_results} with
$k_n^+(u)=k_{\tilde\sigma(\cdot)}^+(u;\tilde b_{+,n},\tilde d_{+,n})$ and
similarly for $k_n^-$. In the notation of
\Cref{sec:sap:rd_asym_general_results}, $\overline{\text{bias}}_n$
is given by
$\frac{1}{2}(\omega_{\tilde\sigma(\cdot),n}(\delta)-\delta
\omega_{\tilde\sigma(\cdot),n}'(\delta))$ and $\tilde v_n$ is given by
$\omega_{\tilde\sigma(\cdot),n}'(\delta)^2$
(see Equation (\ref{eq:maxbias-minbias}) in the main text).
If $\delta$ is chosen to minimize the length of the
fixed-length CI, the half-length will be given by
\begin{align*}
\cv_\alpha(\overline{\text{bias}}_n/\sqrt{\tilde v_n})\sqrt{\tilde v_n}
=\inf_{\delta>0}\cv_\alpha\left(\frac{\omega_{\tilde\sigma(\cdot),n}(\delta)}{2\omega_{\tilde\sigma(\cdot),n}'(\delta)}-\frac{\delta}{2}\right)
\omega_{\tilde\sigma(\cdot),n}'(\delta),
\end{align*}
and $\delta$ will achieve the minimum in the above display. Similarly, if the
MSE criterion is used, $\delta$ will minimize $\overline{\text{bias}}_n^2+v_n$.

We proceed by verifying the conditions of Theorem~\ref{rd_limit_high_level_thm}
for the case where $\tilde\sigma(\cdot)$ is nonrandom and satisfies
Assumption~\ref{prespecified_tilde_sigma_assump}, and then
verifying the conditions in the last display of Theorem~\ref{rd_limit_high_level_thm}
for the case where
$\tilde\sigma(\cdot)$ satisfies Assumption~\ref{estimated_tilde_sigma_assump}.
The limiting kernel $k^+$ and $k^-$ in Assumption~\ref{rd_kern_assump} will
correspond to an asymptotic version of the modulus problem, which we now
describe. Let
\begin{multline*}
  G_\infty(b_+,b_-,d_+,d_-)=
  p_{X,+}(0)\int_{0}^\infty
  \tilde\sigma_+^2(0)k^+_{\tilde\sigma_+(0)}(u;b_+,d_+)^2\, du
  \\+p_{X,-}(0)\int_{0}^\infty
  \tilde\sigma_-^2(0)k^+_{\tilde\sigma_-(0)}(u;b_+,d_+)^2\, du.
\end{multline*}
Consider the limiting inverse modulus problem
\begin{align*}
\omega_{\tilde\sigma_+(0),\tilde\sigma_-(0),\infty}^{-1}(2\tilde b_\infty)
=&\min_{f_+,f_-\in\mathcal{F}_{RDT,p}(C)}
\sqrt{\frac{p_{X,+}(0)}{\tilde\sigma_+^2(0)}\int_0^\infty f_+(u)^2\, du
+\frac{p_{X,-}(0)}{\tilde\sigma_-^2(0)}\int_{-\infty}^0 f_-(u)^2\, du}  \\
&\text{ s.t. }
f_+(0)+f_-(0)\ge \tilde b_\infty.
\end{align*}
We use
$\omega_\infty(\delta)=\omega_{\tilde\sigma_+(0),\tilde\sigma_-(0),\infty}(\delta)$
to denote the limiting modulus corresponding to this inverse modulus. The
limiting inverse modulus problem is solved by the functions
$f_+(u)=\tilde\sigma_+^2(0)k^+_{\tilde\sigma_+(0)}(u;b_+,d_+)=k^+_{1}(u;b_+,d_+)$
and
$f_-(u)=\tilde\sigma_-^2(0)k^+_{\tilde\sigma_-(0)}(u;b_-,d_-)=k^-_{1}(u;b_+,d_+)$
for some $(b_+,b_-,d_+,d_-)$ with $b_++b_-=\tilde b_\infty$ (this holds by the
same arguments as for the modulus problem in \Cref{rd_modulus_solution_sec}).
Thus, for any minimizer of $G_\infty$, the functions $k^+_{1}(\cdot;b_+,d_+)$
and $k^+_{1}(\cdot;b_+,d_+)$ must solve the above inverse modulus problem. The
solution to this problem is unique by strict convexity, which implies that
$G_\infty$ has a unique minimizer. Similarly, the minimizer of $G_n$ is unique
for each $n$. Let
$(\tilde{b}_{+,\infty}, \tilde b_{-,\infty},\tilde d_{+,\infty},
\tilde{d}_{-,\infty})$ denote the minimizer of $G_\infty$. The limiting kernel
$k^+$ in Assumption~\ref{rd_kern_assump} will be given by
$k^+_{\tilde\sigma_+(0)}(\cdot;\tilde b_{+,\infty},\tilde d_{+,\infty})$ and
similarly for $k^-$.

To derive the form of the limiting modulus of continuity, we argue as in
\citet{DoLo92}.
Let $k^+_{1}(\cdot;\tilde b_{+,\infty,1},\tilde d_{+,\infty,1})$ and $k^+_{1}(\cdot;\tilde b_{+,\infty,1},\tilde d_{+,\infty,1})$ solve the inverse modulus problem $\omega_\infty^{-1}(2\tilde b_\infty)$ for $\tilde b_{\infty}=1$.
The feasible set for a given $\tilde b_\infty$ consists of all $b_+,b_-,d_+,d_-$ such that $b_++b_-\ge \tilde b_\infty$, and a given $b_+,b_-,d_+,d_-$ in this set achieves the value
\begin{align*}
&\sqrt{\frac{p_{X,+}(0)}{\tilde\sigma_+^2(0)}\int_0^\infty k^+_{1}(u;b_+,d_+)^2\, du
+\frac{p_{X,-}(0)}{\tilde\sigma_-^2(0)}\int_{-\infty}^0 k^-_{1}(u;b_-,d_-)^2\, du}  \\
&=\sqrt{\frac{p_{X,+}(0)}{\tilde\sigma_+^2(0)}\int_0^\infty k^+_{1}(v b_\infty^{1/p};b_+,d_+)^2\, d(v b_\infty^{1/p})
+\frac{p_{X,-}(0)}{\tilde\sigma_-^2(0)}\int_{-\infty}^0 k^-_{1}(v b_\infty^{1/p};b_-,d_-)^2\, d(v b_\infty^{1/p})}  \\
&=\sqrt{\frac{p_{X,+}(0)}{\tilde\sigma_+^2(0)} \tilde b_\infty^{1/p}\int_0^\infty \tilde b_\infty^2 k^+_{1}(v;b_+/\tilde b_\infty,\bar d_+)^2\, dv
+\frac{p_{X,-}(0)}{\tilde\sigma_-^2(0)}\tilde b_\infty^{1/p}\int_{-\infty}^0 \tilde b_\infty^2 k^-_{1}(v ;b_-/\tilde b_\infty,\bar d_-)^2\, dv},
\end{align*}
where $\bar d_+=(d_{+,1}/\tilde b_\infty^{(p-1)/p},\ldots,d_{+,p-1}/\tilde b_\infty^{1/p})'$
and similarly for $\bar d_-$.
This uses the fact that, for any $h>0$,
$h^{p}k^+_{1}(u/h;b_+,d_+)=k^+_{1}(u;b_+h^p,d_{+,1}h^{p-1},d_{+,2}h^{p-2},\ldots,d_{+,p-1}h)$
and similarly for $k^-_1$.
This can be seen to be $\tilde b_\infty^{(2p+1)/(2p)}$ times the objective evaluated at $(b_+/\tilde b_\infty,b_-/\tilde b_\infty,\bar d_+,\bar d_-)$, which is feasible under $\tilde b_\infty=1$.  Similarly, for any feasible function under $\tilde b_\infty=1$, there is a feasible function under a given $\tilde b_\infty$ that achieves $\tilde b_\infty^{(2p+1)/(2p)}$ times the value of under $\tilde b_\infty=1$.  It follows that
$\omega_{\infty}^{-1}(2b)=b^{(2p+1)/(2p)}\omega_\infty(2)$.
Thus, $\omega_{\infty}^{-1}$ is invertible and the inverse $\omega_{\infty}$ satisfies
$\omega_{\infty}(\delta)=\omega_{\tilde\sigma_+(0),\tilde\sigma_-(0),\infty}(\delta)=\delta^{2p/(2p+1)}\omega_{\tilde\sigma_+(0),\tilde\sigma_-(0),\infty}(1)$.

If $\tilde b_\infty=\omega_\infty(\delta_\infty)$ for some $\delta_\infty$, then
it can be checked that the limiting variance and worst-case bias defined in
\Cref{sec:sap:rd_asym_general_results} correspond to the limiting modulus
problem:
\begin{align}\label{limiting_optimal_bias_var_eq}
\overline{\text{bias}}_\infty=\frac{1}{2}\left(\omega_\infty(\delta_\infty)-\delta_\infty\omega_\infty'(\delta_\infty)\right),
\quad
\sqrt{v_\infty}=\omega'_\infty(\delta_\infty).
\end{align}
Furthermore, we will show that, if $\delta$ is chosen to optimize FLCI length for $\omega_{\tilde\sigma(\cdot),n}$, then this will hold with $\delta_\infty$ optimizing $\cv_\alpha(\overline{\text{bias}}_\infty/\sqrt{v_\infty})\sqrt{v_\infty}$.
Similarly, if $\delta$ is chosen to optimize MSE for $\omega_{\tilde\sigma(\cdot),n}$, then this will hold with $\delta_\infty$ optimizing $\overline{\text{bias}}_\infty^2+v_\infty$.

We are now ready to state the main result concerning the asymptotic validity and efficiency of feasible CIs based on the estimator given in this section.

\begin{theorem}\label{rd_optimal_estimator_thm}
  Suppose Assumptions~\ref{rd_xs_assump} and~\ref{rd_errors_assump} hold. Let
  $\hat L=\hat L_{\delta_n,\tilde\sigma(\cdot)}$ where $\delta_n$ is chosen to
  optimize one of the performance criteria for $\omega_{\tilde\sigma(\cdot),n}$
  (FLCI length, RMSE, or a given quantile of excess length), and suppose that
  $\tilde\sigma(\cdot)$ satisfies Assumption~\ref{estimated_tilde_sigma_assump}
  or Assumption~\ref{prespecified_tilde_sigma_assump} with
  $\tilde\sigma_+(0)=\sigma_+(0)$ and $\tilde\sigma_+(0)=\sigma_-(0)$. Let
  $\overline{\text{bias}}_n=\frac{1}{2}(\omega_{\tilde\sigma(\cdot),n}(\delta_n)-\delta_n
  \omega_{\tilde\sigma(\cdot),n}'(\delta_n))$ and
  $\tilde v_n=\omega_{\tilde\sigma(\cdot),n}'(\delta_n)^2$ denote the worst-case
  bias and variance formulas. Let
  $\hat c_{\alpha,\delta_n}=\hat L-\overline{\text{bias}}_n-z_{1-\alpha}
  \sqrt{\tilde v_n}$ and
  $\hat \chi=\cv_\alpha(\overline{\text{bias}}_n/\sqrt{\tilde{v}_n})
  \sqrt{\tilde{v}_n}$ so that $\hor{\hat c_{\alpha,\delta_n},\infty}$ and
  $[\hat L-\hat\chi,\hat L+\hat \chi]$ give the corresponding CIs.

The CIs $\hor{\hat c_{\alpha,\delta_n},\infty}$ and
$[\hat L-\hat\chi,\hat L+\hat \chi]$ have uniform asymptotic coverage at least
$1-\alpha$. In addition,
$n^{p/(2p+1)}\hat \chi \underset{\mathcal{F},\mathcal{Q}_n}{\overset{p}{\to}}
\chi_\infty$ where
$\chi_\infty=\cv_{\alpha}(\overline{\text{bias}}_\infty/\sqrt{v_\infty})\sqrt{v_\infty}$
with $\overline{\text{bias}}_\infty$ and $\sqrt{v_\infty}$ given
in~\eqref{limiting_optimal_bias_var_eq} and $\delta_\infty=z_\beta+z_{1-\alpha}$
if excess length is the criterion,
$\delta_\infty=\arg\min_{\delta}
\cv_\alpha(\frac{\omega_{\infty}(\delta)}{2\omega_{\infty}'(\delta)}-\frac{\delta}{2})
\omega_{\infty}'(\delta)$ if FLCI length is the criterion, and
$\delta_\infty=\arg\min_{\delta}
[\frac{1}{4}\left(\omega_\infty(\delta_\infty)
    -\delta_\infty\omega_\infty'(\delta_\infty)\right)^2+\omega_\infty'(\delta)^2
]$ if RMSE is the criterion.

Suppose, in addition, that each $\mathcal{Q}_n$ contains a distribution where
the $u_i$s are normal. If the FLCI criterion is used, then no other sequence of
linear estimators $\tilde L$ can satisfy
\begin{align*}
\liminf_{n\to\infty}\inf_{f\in\mathcal{F},Q\in\mathcal{Q}_n}
P_{f,Q}\left(Lf\in\left\{\tilde L\pm n^{-p/(2p+1)}\chi\right\}\right)\ge 1-\alpha
\end{align*}
with $\chi$ a constant with
$\chi<\chi_\infty$.
In addition, for any sequence of confidence sets $\mathcal{C}$ with
$\liminf_{n\to\infty}\inf_{f\in\mathcal{F},Q\in\mathcal{Q}_n}
P_{f,Q}\left(Lf\in\mathcal{C}\right)\ge 1-\alpha$, we have the following bound
on the asymptotic efficiency improvement at any $f\in\mathcal{F}_{RDT,p}(0)$:
\begin{align*}
\liminf_{n\to\infty}\sup_{Q\in\mathcal{Q}_n} \frac{n^{p/(2p+1)}E_{f,Q} \lambda(\mathcal{C})}
  {2 \chi_\infty}
\ge \frac{(1-\alpha) 2^{2p/(2p+1)}E[(z_{1-\alpha}-Z)^{2p/(2p+1)}\mid Z\leq z_{1-\alpha}]}
{\frac{4p}{2p+1}\inf_{\delta>0}\cv_\alpha\left(\delta/(4p)\right)\delta^{-1/(2p+1)}}
\end{align*}
where $Z\sim \mathcal{N}(0,1)$.

If the excess length criterion is used with quantile $\beta$ (i.e. $\delta_n=z_\beta+z_{1-\alpha}$), then
any one-sided CI $\hor{\hat c,\infty}$ with
\begin{align*}
\liminf_{n\to \infty} \inf_{f\in\mathcal{F},Q\in\mathcal{Q}_n}
P_{f,Q}\left(Lf\in \hor{\hat c,\infty}\right)\ge 1-\alpha
\end{align*}
must satisfy
\begin{align*}
\liminf_{n\to\infty}
\frac{\sup_{f\in\mathcal{F},Q\in\mathcal{Q}_n} q_{f,Q,\beta}\left(Lf-\hat c\right)}
  {\sup_{f\in\mathcal{F},Q\in\mathcal{Q}_n} q_{f,Q,\beta}\left(Lf-\hat c_{\alpha,\delta_n}\right)}\ge 1
\end{align*}
and, for any $f\in\mathcal{F}_{RDT,p}(0)$,
\begin{align*}
\liminf_{n\to\infty}
\frac{\sup_{Q\in\mathcal{Q}_n} q_{f,Q,\beta}\left(Lf-\hat c\right)}
  {\sup_{Q\in\mathcal{Q}_n} q_{f,Q,\beta}\left(Lf-\hat c_{\alpha,\delta_n}\right)}\ge \frac{2^{2p/(2p+1)}}{1+2p/(2p+1)}.
\end{align*}

\end{theorem}

To prove this theorem, we first prove a series of lemmas. To deal with the case
where $\delta$ is chosen to optimize FLCI length or MSE, we will use the
characterization of the optimal $\delta$ for these criteria from
\citet{donoho94}, which is described at the beginning of
\Cref{sec:sap:asym_efficiency_bounds}. In particular, for $\rho_A$ and
$\chi_{A,\alpha}$ given in \Cref{sec:sap:asym_efficiency_bounds}, the $\delta$
that optimizes FLCI length is given by the $\delta$ that maximizes
$\omega_{\tilde\sigma(\cdot),n}(\delta)\chi_{A,\alpha}(\delta)/\delta$, and the
resulting FLCI half-length is given by
$\sup_{\delta>0}\omega_{\tilde\sigma(\cdot),n}(\delta)\chi_{A,\alpha}(\delta)/\delta$.
In addition, when $\delta$ is chosen to optimize FLCI length, $\chi_\infty$ in
Theorem~\ref{rd_optimal_estimator_thm} is given by
$\sup_{\delta>0}\omega_{\infty}(\delta)\chi_{A,\alpha}(\delta)/\delta$, and
$\delta_\infty$ maximizes this expression. If $\delta$ is chosen according to
the MSE criterion, then $\delta$ maximizes
$\omega_{\tilde\sigma(\cdot),n}(\delta)\sqrt{\rho_{A}(\delta)}/\delta$ and
$\delta_\infty$ maximizes
$\omega_{\infty}(\delta)\sqrt{\rho_{A}(\delta)}/\delta$.

\begin{lemma}\label{Gn_convergence_lemma}
For any constant $B$, the following holds.
Under Assumption~\ref{prespecified_tilde_sigma_assump},
\begin{align*}
\lim_{n\to\infty}\sup_{\|(b_+,b_-,d_+,d_-)\|\le B}|G_n(b_+,b_-,d_+,d_-)-G_\infty(b_+,b_-,d_+,d_-)|=0.
\end{align*}
Under Assumption~\ref{estimated_tilde_sigma_assump},
\begin{align*}
\sup_{\|(b_+,b_-,d_+,d_-)\|\le B}|G_n(b_+,b_-,d_+,d_-)-G_\infty(b_+,b_-,d_+,d_-)|
\underset{\mathcal{F},\mathcal{Q}_n}{\overset{p}{\to}} 0.
\end{align*}
\end{lemma}
\begin{proof}
  Define $\tilde G_n^+(b_+,d_+) =\frac{1}{nh_n} \sum_{i=1}^n
  k_{1}^+(x_i/h_n;b_+,d_+)^2$, and define $\tilde G_n^-$ analogously. Also,
  $\tilde G_\infty^+(b_+,d_+)=p_{X,+}(0)\int_0^\infty k_1^+(u;b_+,d_+)^2\, du$,
  with $G_\infty^-$ defined analogously. For each $(b_+,d_+)$,
  $\tilde{G}_n(b_+,d_+) \to G_\infty(b_+,d_+)$ by Assumption~\ref{rd_xs_assump}.
  To show uniform convergence, first note that, for some constant $K_1$, the
  support of $k_{1}^+(\cdot;b_+,d_+)$ is bounded by $K_1$ uniformly over
  $\|(b_+,d_+)\|\le B$ and similarly for $k_{1}^-(\cdot;b_-,d_-)$. Thus, for any
  $(b_+,d_+)$ and $(\bar b_+,\bar d_+)$,
\begin{equation*}
  |G_n^+(b_+,d_+)-G_n^+(\bar b_+,\bar d_+)| \le
  \left[\frac{1}{nh_n}\sum_{i=1}^n \1{|x_i/h_n|\le K_1}\right]  \sup_{|u|\le
    K_1} |k_1^+(u;b_+,d_+)-k_1^+(u;\bar b_+,\bar d_+)|.
\end{equation*}
Since the term in brackets converges to a finite constant by
Assumption~\ref{rd_xs_assump} and $k_1^+$ is Lipschitz continuous on any bounded
set, it follows that there exists a constant $K_2$ such that
$|G_n^+(b_+,d_+)-G_n^+(\bar{b}_+,\bar d_+)|\le K_2\|(b_+,d_+)-(\bar b_+,
\bar{d}_+)\|$ for all $n$. Using this and applying pointwise convergence of
$G_n^+(b_+,d_+)$ on a small enough grid along with uniform continuity of
$G_\infty(b_+,d_+)$ on compact sets, it follows that
\begin{align*}
\lim_{n\to\infty}\sup_{\|(b_+,b_-,d_+,d_-)\|\le B}|\tilde G_n(b_+,d_+)-\tilde G_\infty(b_+,d_+)|=0,
\end{align*}
and similar arguments give the same statement for $\tilde{G}_n^{-}$ and
$\tilde{G}_\infty^-$. Under Assumption~\ref{prespecified_tilde_sigma_assump},
\begin{multline*}
  \left|G_n(b_+,b_-,d_+,d_-) -\left[\tilde G_n(b_+,d_+)\tilde\sigma_+^2(0)+
      \tilde G_n(b_-,d_-)\tilde\sigma_-^2(0)\right]\right|  \le\\
  \overline k\cdot \left[\frac{1}{nh_n}\sum_{i=1}^n \1{|x_i/h_n|\le K_1}\right]
  \left[\sup_{0<x\le K_1h_n}
    \left|\tilde\sigma_+^2(0)-\tilde\sigma_+^2(x)\right| +\sup_{-K_1h_n\le x
      <0}\left|\tilde\sigma_-^2(0)-\tilde\sigma_-^2(x)\right|\right]
\end{multline*}
where $\overline k$ is an upper bound for $|k^+_{1}(x)|$ and $|k^-_{1}(x)|$. This
converges to zero by left- and right- continuity of $\tilde\sigma$ at $0$. The
result then follows since $G_\infty(b_+,b_-,d_+,d_-)=\tilde\sigma_+^2(0)
\tilde{G}_\infty^+(b_+,d_+)+\tilde\sigma_-^2(0)\tilde G_\infty^-(b_-,d_-)$.
Under Assumption~\ref{estimated_tilde_sigma_assump}, we have
$G_n(b_+,b_-,d_+,d_-) =\tilde G_n^+(b_+,d_+)\hat\sigma_{+}^2 +
\tilde{G}_n^+(b_-,d_-)\hat\sigma_-^2$, and the result follows from uniform
convergence in probability of $\hat\sigma_+^2$ and $\hat\sigma_-^2$ to
$\tilde\sigma_+^2(0)$ and $\tilde\sigma_-^2(0)$.
\end{proof}

\begin{lemma}\label{Gn_solution_bound_lemma}
  Under Assumption~\ref{prespecified_tilde_sigma_assump}, $\|(\tilde{b}_{+,n},
  \tilde b_{-,n},\tilde d_{+,n},\tilde d_{-,n})\|\le B$ for some constant $B$
  and $n$ large enough. Under Assumption~\ref{estimated_tilde_sigma_assump}, the
  same statement holds with probability approaching one uniformly over
  $\mathcal{F},\mathcal{Q}_n$.
\end{lemma}
\begin{proof}
  Let $\mathcal{A}(x,b,d)=b+\sum_{i=1}^{p-1}d(x/h_{n})^{j}$, where
  $d=(d_{1},\dotsc,d_{p-1})$. Note $G_n(b_+,b_-,d_+,d_-)$ is bounded from below
  by $1/\sup_{|x|\le h_n} \tilde\sigma^2(x)$ times
\begin{multline*}
  \frac{1}{nh_n}\sum_{i:0< x_i\le h_n}
  \left(|\mathcal{A}(x_{i},b_{+},d_{+})|-C\right)_+^2
  +\frac{1}{nh_n}\sum_{i:-h_n\le x_i<0} \left(|\mathcal{A}(x_{i},b_{-},d_{-})|-C\right)_+^2  \\
\ge \frac{1}{4 nh_n}\sum_{i:0< x_i\le h_n}
\left[\mathcal{A}(x_{i},b_{+},d_{+})^{2}-4C^2\right]
+\frac{1}{4 nh_n}\sum_{i:-h_n\le x_i< 0}
\left[\mathcal{A}(x_{i},b_{-},d_{-})^{2}-4C^2\right]
\end{multline*}
(the inequality follows since, for any $s\ge 2C$, $(s-C)^2\ge s^2/4\ge
s^2/4-C^2$ and, for $2C\ge s\ge C$, $(s-C)^2\ge 0\ge s^2/4-C^2$). Note that, for
any $B>0$
\begin{multline*}
  \inf_{\max\{|b_+|,|d_{+,1}|,\ldots,|d_{+,p-1}|\}\ge B} \frac{1}{4
    nh_n}\sum_{i:0< x_i\le h_n}
  \mathcal{A}(x_{i},b_{+},d_{+})^{2}  \\
  =B^2 \inf_{\max\{|b_+|,|d_{+,1}|,\ldots,|d_{+,p-1}|\}\ge 1}\frac{1}{4
    nh_n}\sum_{i:0< x_i\le h_n}
  \mathcal{A}(x_{i},b_{+},d_{+})^{2} \\
  \to \frac{p_{X,+}(0)}{4} B^2
  \inf_{\max\{|b_+|,|d_{+,1}|,\ldots,|d_{+,p-1}|\}\ge 1} \int_{0}^\infty
  \left(b_++\sum_{i=1}^{p-1}d_{+,j}u^j\right)^2\, du
\end{multline*}
and similarly for the term involving $\mathcal{A}(x_{i},b_{-},d_{-})$ (the
convergence follows since the infimum is taken on the compact set where
$\max\{|b_+|,|d_{+,1}|,\ldots,|d_{+,p-1}|\}=1$). Combining this with the
previous display and the fact that $\frac{1}{nh}\sum_{i:|x_i|\le h_n}C^2$
converges to a finite constant, it follows that, for some $\eta>0$,
$\inf_{\max\{|b_+|,|d_{+,1}|,\ldots,|d_{+,p-1}|\}\ge B}G_n(b_+,b_-,d_+,d_-)\ge
(B^2\eta-\eta^{-1})/\sup_{|x|\le h_n} \tilde\sigma^2(x)$ for large enough $n$.
Let $K$ be such that $G_\infty(\tilde{b}_{+,\infty},
\tilde{b}_{-,\infty},\tilde{d}_{+,\infty}, \tilde{d}_{-,\infty})\le K/2$ and
$\max\{\tilde\sigma_+^2(0),\tilde\sigma_-^2(0)\}\le K/2$. Under
Assumption~\ref{prespecified_tilde_sigma_assump}, $G_n(
\tilde{b}_{+,\infty},\tilde b_{-,\infty},\tilde d_{+,\infty},
\tilde{d}_{-,\infty})< K$ and $\sup_{|x|\le h_n} \tilde\sigma^2(x)\le K$ for
large enough $n$. Under Assumption~\ref{estimated_tilde_sigma_assump},
$G_n(\tilde b_{+,\infty},\tilde{b}_{-,\infty},
\tilde{d}_{+,\infty},\allowbreak\tilde d_{-,\infty})< K$ and $\sup_{|x|\le h_n}
\tilde\sigma^2(x)\le K$ with probability approaching one uniformly over
$\mathcal{F},\mathcal{Q}_n$. Let $B$ be large enough so that
$(B^2\eta-\eta^{-1})/K>K$. Then, when $G_n(\tilde b_{+,\infty},
\tilde{b}_{-,\infty},\tilde d_{+,\infty},\tilde d_{-,\infty})\le K$ and
$\sup_{|x|\le h_n} \tilde\sigma^2(x)\le K$, $(\tilde b_{+,\infty},
\tilde{b}_{-,\infty},\tilde d_{+,\infty},\tilde d_{-,\infty})$ will give a lower
value of $G_n$ than any $(b_+, b_-, d_+, d_-)$ with
$\max\{|b_+|,|d_{+,1}|,\ldots,|d_{+,p-1}|,|b_-|,|d_{-,1}|,\ldots,|d_{-,p-1}|\}\ge
B$. The result follows from the fact that the max norm on $\mathbb{R}^{2p}$ is
bounded from below by a constant times the Euclidean norm.
\end{proof}

\begin{lemma}\label{bn_convergence_lemma}
  If Assumption~\ref{prespecified_tilde_sigma_assump} holds and $\tilde b_n\to
  \tilde b_\infty$, then $(\tilde b_{+,n},\tilde b_{-,n},\tilde d_{+,n},
  \tilde{d}_{-,n}) \to (\tilde b_{+,\infty},\allowbreak \tilde b_{-,\infty},
  \tilde{d}_{+,\infty},\tilde d_{-,\infty})$. If
  Assumption~\ref{estimated_tilde_sigma_assump} holds and $\tilde{b}_{n}
  \underset{\mathcal{F},\mathcal{Q}_n}{\overset{p}{\to}} \tilde b_\infty>0$,
  $(\tilde b_{+,n},\tilde b_{-,n},\tilde d_{+,n},\tilde d_{-,n})
  \underset{\mathcal{F},\mathcal{Q}_n}{\overset{p}{\to}} (\tilde{b}_{+,\infty},
  \tilde{b}_{-,\infty},\tilde{d}_{+,\infty},\tilde{d}_{-,\infty})$.
\end{lemma}
\begin{proof}
  By Lemma~\ref{Gn_solution_bound_lemma}, $B$ can be chosen so that
  $\|(\tilde{b}_{+,n},\tilde b_{-,n},\tilde d_{+,n},\tilde d_{-,n})\|\le B$ for
  large enough $n$ under Assumption~\ref{prespecified_tilde_sigma_assump} and
  $\|(\tilde{b}_{+,n},\tilde b_{-,n},\tilde d_{+,n},\tilde d_{-,n})\le B\|$ with
  probability one uniformly over $\mathcal{F},\mathcal{Q}_n$ under
  Assumption~\ref{estimated_tilde_sigma_assump}. The result follows from
  Lemma~\ref{Gn_convergence_lemma}, continuity of $G_\infty$ and the fact that
  $G_\infty$ has a unique minimizer.
\end{proof}

\begin{lemma}\label{omega_inverse_convergence_lemma}
  If Assumption~\ref{prespecified_tilde_sigma_assump} holds and $\tilde{b}_n
  \to\tilde b_\infty>0$, then $\omega_n^{-1}(n^{p/(2p+1)} \tilde b_n)\to
  \omega_\infty^{-1}(\tilde b_\infty)$. If
  Assumption~\ref{estimated_tilde_sigma_assump} holds and $\tilde{b}_n
  \underset{\mathcal{F},\mathcal{Q}_n}{\overset{p}{\to}} b_\infty>0$, then
  $\omega_n^{-1}(n^{p/(2p+1)} \tilde{b}_n)
  \underset{\mathcal{F},\mathcal{Q}_n}{\overset{p}{\to}}
  \omega_\infty^{-1}(\tilde b_\infty)$.
\end{lemma}
\begin{proof}
The result is immediate from Lemmas~\ref{Gn_convergence_lemma} and~\ref{bn_convergence_lemma}.
\end{proof}

\begin{lemma}\label{rd_uniform_modulus_convergence_lemma}
Under Assumption~\ref{prespecified_tilde_sigma_assump}, we have, for any $\overline\delta>0$,
\begin{align*}
\sup_{0<\delta\le \overline\delta} \left|n^{p/(2p+1)}\omega_n(\delta)
-\omega_\infty(\delta)\right|\to 0.
\end{align*}
Under Assumption~\ref{estimated_tilde_sigma_assump}, we have, for any $\overline\delta>0$,
\begin{align*}
\sup_{0<\delta\le \overline\delta} \left|n^{p/(2p+1)}\omega_n(\delta)
-\omega_\infty(\delta)\right|
\underset{\mathcal{F},\mathcal{Q}_n}{\overset{p}{\to}} 0.
\end{align*}

\end{lemma}
\begin{proof}
  The first statement is immediate from
  Lemma~\ref{omega_inverse_convergence_lemma} and
  Lemma~\ref{modulus_convergence_lemma} (with $n^{p/(2p+1)}\omega_n$ playing the
  role of $\omega_n$ in that lemma). For the second claim, note that, if
  $|\hat{\sigma}_{+}-\sigma_{+}(0)|\le \eta$ and $|\hat \sigma_{-}-\sigma_{-}(0)|\le
  \eta$, $\omega_{n,\underline\sigma(\cdot)}(\delta)\le
  \omega_{\tilde\sigma(\cdot),n}(\delta) \le
  \omega_{n,\overline\sigma(\cdot)}(\delta)$, where
  $\underline\sigma(x)=(\sigma_+(0)-\eta)\1{x>0}+(\sigma_-(0)-\eta)\1{x<0}$ and
  $\overline\sigma(x)$ is defined similarly. Applying the first statement in the
  lemma and the fact that $|\hat \sigma_+-\sigma_+(0)|\le \eta$ and
  $|\hat{\sigma}_{-} -\sigma_{-}(0)|\le \eta$ with probability approaching one
  uniformly over $\mathcal{F},\mathcal{Q}_n$, it follows that, for any
  $\varepsilon>0$, we will have
\begin{align*}
\omega_{\underline\sigma_+(0),\underline\sigma_-(0),\infty}(\delta)
-\varepsilon
\le n^{p/(2p+1)}\omega_n(\delta)
\le \omega_{\overline\sigma_+(0),\overline\sigma_-(0),\infty}(\delta)
+\varepsilon
\end{align*}
for all $0<\delta<\overline\delta$
with probability approaching one uniformly over $\mathcal{F},\mathcal{Q}_n$.  By making $\eta$ and $\varepsilon$ small, both sides can be made arbitrarily close to $\omega_{\infty}(\delta)=\omega_{\infty,\sigma_+(0),\sigma_-(0)}(\delta)$.
\end{proof}

\begin{lemma}\label{rd_modulus_max_convergence_lemma}
  Let $r$ denote $\sqrt{\rho_A}$ or $\chi_{A,\alpha}$. Under
  Assumption~\ref{prespecified_tilde_sigma_assump},
\begin{align*}
\sup_{\delta>0}n^{p/(2p+1)}\omega_n(\delta)r(\delta/2)/\delta
\to \sup_{\delta>0}\omega_\infty(\delta)r(\delta/2)/\delta.
\end{align*}
Let $\delta_n$ minimize the left-hand side of the above display, and let $\delta^*$ minimize the right-hand side.
Then
$\delta_n\to\delta^*$ under Assumption~\ref{prespecified_tilde_sigma_assump} and
$\delta_n \underset{\mathcal{F},\mathcal{Q}_n}{\overset{p}{\to}} \delta^*$
under Assumption~\ref{estimated_tilde_sigma_assump}.
In addition, for any $0<\alpha<1$ and $Z$ a standard normal variable,
\begin{align*}
\lim_{n\to \infty} (1-\alpha)E[n^{p/(2p+1)}\omega_n(2(z_{1-\alpha}-Z))|Z\le z_{1-\alpha}]
=(1-\alpha)E[\omega_\infty(2(z_{1-\alpha}-Z))|Z\le z_{1-\alpha}].
\end{align*}
\end{lemma}
\begin{proof}
  All the statements are immediate from
  Lemmas~\ref{rd_uniform_modulus_convergence_lemma}
  and~\ref{modulus_objective_convergence_lemma} except for the statement that
  $\delta_n \underset{\mathcal{F},\mathcal{Q}_n}{\overset{p}{\to}} \delta^*$
  under Assumption~\ref{estimated_tilde_sigma_assump}. The statement that
  $\delta_n \underset{\mathcal{F},\mathcal{Q}_n}{\overset{p}{\to}} \delta^*$
  under Assumption~\ref{estimated_tilde_sigma_assump} follows by using
  Lemma~\ref{rd_uniform_modulus_convergence_lemma} and analogous arguments to
  those in Lemma~\ref{modulus_objective_convergence_lemma} to show that
  there exist $0<\underline\delta<\overline\delta$ such that $\delta_n\in
  [\underline\delta,\overline\delta]$ with probability approaching on uniformly
  in $\mathcal{F},\mathcal{Q}_n$, and that
  $\sup_{\delta\in[\underline\delta,\overline\delta]}\left|n^{p/(2p+1)}\omega_n(\delta)r(\delta/2)/\delta-\omega(\delta)r(\delta/2)/\delta\right|\underset{\mathcal{F},\mathcal{Q}_n}{\overset{p}{\to}}
  0$.
\end{proof}

\begin{lemma}\label{rd_opt_kern_limit_lemma}
Under Assumptions~\ref{rd_xs_assump} and~\ref{rd_errors_assump},
the following hold.
If Assumption~\ref{prespecified_tilde_sigma_assump} holds
and $\tilde b_n$ is a deterministic sequence with
$\tilde b_n \to \tilde b_\infty>0$, then
\begin{align*}
&\sup_{x}|k_{\tilde\sigma(\cdot)}^+(x;\tilde b_{+,n},\tilde d_{+,n})
  - k_{\tilde\sigma_+(0)}^+(x;\tilde b_{+,\infty},\tilde d_{+,\infty}) |
\to 0,\\
&\sup_{x}|k_{\tilde\sigma(\cdot)}^-(x;\tilde b_{-,n},\tilde d_{-,n})
  - k_{\tilde\sigma_-(0)}^-(x;\tilde b_{-,\infty},\tilde d_{-,\infty}) |
\to 0.
\end{align*}
If Assumption~\ref{estimated_tilde_sigma_assump} holds
and $\tilde b_n$ is a random sequence with
$\tilde b_n\underset{\mathcal{F},\mathcal{Q}_n}{\overset{p}{\to}}
\tilde b_\infty>0$, then
\begin{align*}
\sup_{x}|k_{\tilde\sigma(\cdot)}^+(x;\tilde b_{+,n},\tilde d_{+,n})
  - k_{\tilde\sigma_+(0)}^+(x;\tilde b_{+,\infty},\tilde d_{+,\infty}) |
&\underset{\mathcal{F},\mathcal{Q}_n}{\overset{p}{\to}} 0,\\
\sup_{x}|k_{\tilde\sigma(\cdot)}^-(x;\tilde b_{-,n},\tilde d_{-,n})
  - k_{\tilde\sigma_-(0)}^-(x;\tilde b_{-,\infty},\tilde d_{-,\infty}) |
  &\underset{\mathcal{F},\mathcal{Q}_n}{\overset{p}{\to}} 0
\end{align*}
\end{lemma}
\begin{proof}
Note that
\begin{multline*}
  |k_{\tilde\sigma(\cdot)}^+(x;\tilde b_{+,n},\tilde d_{+,n}) -
  k_{\tilde\sigma_+(0)}^+(x;\tilde b_{+,\infty},\tilde d_{+,\infty})| \le
  |k_{\tilde\sigma(\cdot)}^+(x;\tilde b_{+,n},\tilde{d}_{+,n})
  -k_{\tilde\sigma_+(0)}^+(x;\tilde b_{+,n},\tilde d_{+,n})|\\
  +|k_{\tilde\sigma_+(0)}^+(x;\tilde b_{+,n},\tilde d_{+,n})
  -k_{\tilde\sigma_+(0)}^+(x;\tilde b_{+,\infty},\tilde d_{+,\infty})|.
\end{multline*}
Under Assumption~\ref{prespecified_tilde_sigma_assump}, the first term is, for
large enough $n$, bounded by a constant times $\sup_{0<x< h_n
  K}|\tilde\sigma^{-2}(x)-\tilde\sigma_+^{-2}(0)|$, where $K$ is bound on the
support of $k^+_{1}(\cdot;b_+,d_+)$ over $b_+,d_+$ in a neighborhood of
$\tilde{b}_{+,\infty}, \tilde d_{+,\infty}$. This converges to zero by
Assumption~\ref{prespecified_tilde_sigma_assump}. The second term converges to
zero by Lipschitz continuity of $k^+_{\tilde\sigma_+(0)}$. Under
Assumption~\ref{estimated_tilde_sigma_assump}, the first term is bounded by a
constant times $|\hat\sigma_+^{-2}-\tilde\sigma_+(0)|$, which converges in
probability to zero uniformly over $\mathcal{F},\mathcal{Q}_n$ by assumption.
The second term converges in probability to zero uniformly over
$\mathcal{F},\mathcal{Q}_n$ by Lipschitz continuity of
$k^+_{\tilde\sigma_+(0)}$. Similar arguments apply to
$k^-_{\tilde\sigma(\cdot)}$ in both cases.
\end{proof}

\begin{lemma}\label{rd_estimated_sigma_equiv_lemma}
Under Assumptions~\ref{rd_xs_assump} and~\ref{rd_errors_assump},
the following holds.
Let $\hat L$ denote the estimator $\hat L_{\delta_n,\tilde\sigma(\cdot)}$ where $\tilde\sigma(\cdot)$ satisfies Assumption~\ref{prespecified_tilde_sigma_assump} and $\delta_n=\omega^{-1}_{\tilde\sigma(\cdot),n}(2n^{-p/(2p+1)}\tilde b_n)$ where $\tilde b_n$ is a deterministic sequence with $\tilde b_n\to \tilde b_\infty$.  Let $\overline{\text{bias}}_n$ and $\tilde v_n$ denote the corresponding worst-case bias and variance formulas.
Let $\hat L^*$ denote the estimator
$\hat L_{\delta_n^*,\tilde\sigma(\cdot)}$ where $\tilde\sigma^*(\cdot)=\hat\sigma_+\1{x>0}+\hat\sigma_-\1{x<0}$ satisfies Assumption~\ref{estimated_tilde_sigma_assump} with the same value of $\tilde\sigma_+(0)$ and $\tilde\sigma_-(0)$ and
$\delta_n^*=\omega^{-1}_{\tilde\sigma(\cdot),n}(2n^{-p/(2p+1)}\tilde b_n^*)$
where $\tilde b_n^*\underset{\mathcal{F},\mathcal{Q}_n}{\overset{p}{\to}} \tilde b_\infty$.  Let $\overline{\text{bias}}_n^*$ and $\tilde v_n^*$ denote the corresponding worst-case bias and variance formulas.  Then
\begin{align*}
n^{p/(2p+1)}\left(\hat L-\hat L^*\right) \underset{\mathcal{F},\mathcal{Q}_n}{\overset{p}{\to}} 0,
& &
n^{p/(2p+1)}\left(\overline{\text{bias}}_n-\overline{\text{bias}}_n^*\right) \underset{\mathcal{F},\mathcal{Q}_n}{\overset{p}{\to}} 0,
& &
\frac{\tilde v_n}{\tilde v_n^*}
\underset{\mathcal{F},\mathcal{Q}_n}{\overset{p}{\to}} 1.
\end{align*}

\end{lemma}
\begin{proof}
We have
\begin{align*}
\hat L=\frac{1}{nh_n}\sum_{i=1}^n w_n(x_i/h_n)y_i=\frac{1}{nh_n}\sum_{i=1}^n w_{n}(x_i/h_n)f(x_i)+\frac{1}{nh_n}\sum_{i=1}^n w_n(x_i/h_n)u_i
\end{align*}
where
$w_n(u)=\frac{k_{\tilde\sigma(\cdot)}^+(u;\tilde b_{+,n},\tilde d_{+,n})}{\frac{1}{nh_n}\sum_{j=1}^n k_{\tilde\sigma(\cdot)}^+(x_j/h_n;\tilde b_{+,n},\tilde d_{+,n})}$
for $u>0$ and similarly with $k_{\tilde\sigma(\cdot)}^+$ replaced by $k_{\tilde\sigma(\cdot)}^-$ for $u<0$
(here, $\tilde d_{+,n}$, $\tilde d_{-,n}$, $\tilde b_{+,n}$ and $\tilde b_{-,n}$ are the coefficients in the solution to the inverse modulus problem defined above).
Similarly,
$\hat L^*$ takes the same form with $w_n$ replaced by
$w_n^*(u)=\frac{k_{\tilde\sigma^*(\cdot)}^+(u;\tilde b_n^*,\tilde d_n^*)}{\frac{1}{nh_n}\sum_{j=1}^n k_{\tilde\sigma^*(\cdot)}^+(x_j/h_n;\tilde b_n^*,\tilde d_n^*)}$ for $u>0$ and similarly for $u<0$ (with $\tilde d_{+,n}^*$, $\tilde d_{-,n}^*$, $\tilde b_{+,n}^*$ and $\tilde b_{-,n}^*$ the coefficients in the solution to the corresponding inverse modulus problem).
Let
$w_\infty(u)=\frac{k_{\tilde\sigma(\cdot)}^+(u;\tilde b_n^*,\tilde d_n^*)}{p_{X,+}(0)\int k_{\tilde\sigma(\cdot)}^+(u;\tilde b_\infty,\tilde d_\infty)\, du}$
Note that, by Lemma~\ref{rd_opt_kern_limit_lemma}, $\sup_u |w_n(u)-w_\infty(u)|\to 0$ and $\sup_u |w_n^*(u)-w_\infty(u)|\underset{\mathcal{F},\mathcal{Q}_n}{\overset{p}{\to}} 0$.

We have
\begin{align*}
\hat L-\hat L^*=\frac{1}{nh_n}\sum_{i=1}^n [w_n(x_i/h_n)-w_n^*(x_i/h_n)]r(x_i)+\frac{1}{nh_n}\sum_{i=1}^n [w_n(x_i/h_n)-w_n^*(x_i/h_n)]u_i
\end{align*}
where
$f(x)=\sum_{j=0}^{p-1}f^{(j)}_+(0)x^j\1{x>0}/j!+\sum_{j=0}^{p-1}f^{(j)}_-(0)x^j\1{x<0}/j!+r(x)$
and we use the fact that
$\sum_{i=1}^{n}w_n(x_i/h_n)x_i^j=\sum_{i=1}^{n} w_n^*(x_i/h_n)x_i^j$ for
$j=0,\ldots,p-1$. Let $B$ be such that, with probability approaching one,
$w_n(x)=w_n^*(x)=0$ for all $x$ with $|x|\ge B $. The first term is bounded by
\begin{equation*}
  \frac{C}{nh_n}
  \sum_{i=1}^n |w_n(x_i/h_n)-w_n^*(x_i/h_n)|\cdot |x_i|^p
  \le \sup_{x}\left|w_n(x)-w_n^*(x)\right| B h_n^p \frac{C}{nh_n}
  \sum_{i=1}^n \1{|x_i/h_n|\le B}.
\end{equation*}
It follows from Lemma~\ref{rd_opt_kern_limit_lemma} that
$\sup_{x}\left|w_n(x)-w_n^*(x)\right|
\underset{\mathcal{F},\mathcal{Q}_n}{\overset{p}{\to}} 0$. Also,
$\frac{1}{nh_n}\sum_{i=1}^n \1{|x_i/h_n|\le B}$ converges to a finite constant
by Assumption~\ref{rd_xs_assump}. Thus, the above display converges uniformly in
probability to zero when scaled by $n^{p/(2p+1)}=h_n^{-p}$.

For the last term in $\hat L-\hat L^*$, scaling by $n^{p/(2p+1)}$ gives
\begin{align*}
\frac{1}{\sqrt{nh_n}}\sum_{i=1}^n [w_n(x_i/h_n)-w_\infty(x_i/h_n)]u_i
-\frac{1}{\sqrt{nh_n}}\sum_{i=1}^n [w_n^*(x_i/h_n)-w_\infty(x_i/h_n)]u_i.
\end{align*}
The first term has mean zero and variance
$\frac{1}{nh}\sum_{i=1}^n [w_n(x_i/h_n)-w_\infty(x_i/h_n)]^2\sigma^2(x_i)$
which is bounded by
$\left\{\sup_u [w_n(u)-w_\infty(u)]^2\right\}\left[\sup_{|x|\le Bh_n} \sigma^2(x)\right]\frac{1}{nh}\sum_{i=1}^n\1{|x_i/h_n|\le B}\to 0$.
Let $c_{n,+}=\frac{\hat\sigma_+^2}{nh_n}\sum_{i=1}^n k_{\tilde\sigma^*(\cdot)}(x_i/h_n;\tilde b_{+,n}^*,\tilde d_{+,n}^*)$
and $c_{\infty,+}=\tilde\sigma_+^2(0)p_{X,+}(0)\int k_{\tilde\sigma^*(\cdot)}(u;\tilde b_{\infty},\tilde d_{\infty})$ so that
$c_{n,+}\underset{\mathcal{F},\mathcal{Q}_n}{\overset{p}{\to}} c_{\infty,+}$, and define $c_{n,-}$ and $c_{\infty,-}$ analogously.
With this notation, we have, for $x_i>0$,
\begin{align*}
&w_n^*(x_i/h_n)
=c_{n,+}^{-1}\hat \sigma_+^2k_{\tilde\sigma^*(\cdot)}(x_i/h_n;\tilde b_{+,n}^*,\tilde d_{+,n}^*)
=c_{n,+}^{-1}h_+(x_i/h_n;\tilde b_{+,n}^*,\tilde d_{+,n}^*)
\end{align*}
and $w_\infty(u)=c_{\infty,+}^{-1}h_+(x_i/h_n; \tilde b_{+,\infty}, \tilde d_{+,\infty})$ where
\begin{equation*}
h_+(u;b_+,d_+)
=\Big( b_{+}+\sum_{j=1}^{p-1} d_{+,j}u^j-C|u|^p\Big)_+
-\Big( b_{+}+\sum_{j=1}^{p-1} d_{+,j}u^j+C|u|^p\Big)_-.
\end{equation*}
Thus,
\begin{multline*}
\frac{1}{\sqrt{nh}}\sum_{i=1}^n [w_n^*(x_i/h_n)-w_\infty(x_i/h_n)]\1{x_i>0}u_i  \\
= \frac{c_{n,+}^{-1}}{\sqrt{nh}}\sum_{i=1}^n
[h_+(u;\tilde b_{+,n},\tilde d_{+,n})-h_+(u;\tilde b_{+,\infty},\tilde d_{+,\infty})]\1{x_i>0}u_i  \\
+\frac{(c_{n,+}^{-1}-c_{n,\infty}^{-1})}{\sqrt{nh}}\sum_{i=1}^n h_+(u;\tilde b_{+,\infty},\tilde d_{+,\infty})\1{x_i>0}u_i.
\end{multline*}
The last term converges to zero uniformly in probability by Slutsky's Theorem.
The first term can be written as $c_{n,+}^{-1}$ times the sum of
\begin{multline*}
  \frac{1}{\sqrt{nh}}\sum_{i=1}^n\left[
  \left( \tilde b_{+,n}^*+\sum_{j=1}^{p-1}
  \tilde d_{+,n,j}^*\left(\frac{x_i}{h_n}\right)^j
  -C\left|\frac{x_i}{h_n}\right|^p\right)_+\right.\\
  - \left.\left( \tilde b_{+,\infty}+\sum_{j=1}^{p-1} \tilde d_{+,\infty,j}\left(\frac{x_i}{h_n}\right)^j-C\left|\frac{x_i}{h_n}\right|^p\right)_+\right]u_i
\end{multline*}
and a corresponding term with $(\cdot)_+$ replaced by $(\cdot)_-$, which can be dealt with using similar arguments.
Letting $A(b_+,d_+)=\{u\colon b_{+}+\sum_{j=1}^{p-1} d_{+,j}u^j-C|u|^p\ge 0\}$, the above display is equal to
\begin{align*}
&\frac{1}{\sqrt{nh}}\sum_{i=1}^n\left( \tilde b_{+,n}^*-\tilde b_{+,\infty}+\sum_{j=1}^{p-1} (\tilde d_{+,n,j}^*-\tilde d_{+,\infty,j})\left(\frac{x_i}{h_n}\right)^j\right)\1{x_i/h_n\in A(\tilde b_{+,\infty},\tilde d_{+,\infty})}u_i  \\
&+\frac{1}{\sqrt{nh}}\sum_{i=1}^n\left( \tilde b_{+,n}^*+\sum_{j=1}^{p-1} d_{+,n,j}^*\left(\frac{x_i}{h_n}\right)^j-C\left|\frac{x_i}{h_n}\right|^p\right)  \\
&\cdot\left[\1{x_i/h_n\in A(\tilde b_{+,n}^*,\tilde d_{+,n}^*)}-\1{x_i/h_n\in A(\tilde b_{+,\infty},\tilde d_{+,\infty})}\right]u_i.
\end{align*}
The first term converges to zero uniformly in probability by Slutsky's Theorem.  The second term can be written as a sum of terms of the form
\begin{equation*}
\frac{1}{\sqrt{nh_n}}
\sum_{i=1}^n (x_i/h_n)^j\left[\1{x_i/h_n\in A(\tilde b_{+,n}^*,\tilde d_{+,n}^*)}
-\1{x_i/h_n\in A(\tilde b_{+,\infty},\tilde d_{+,\infty})}\right]u_i
\end{equation*}
times sequences that converge uniformly in probability to finite constants.  To show that this converges in probability to zero uniformly over $\mathcal{F},\mathcal{Q}_n$,
note that, letting $u_1^*,\ldots,u_k^*$ be the positive zeros of the polynomial
$\tilde b_{+,\infty}+\sum_{j=1}^{p-1} \tilde d_{+,j,\infty} u^j+Cu^p$,
the following statement will hold
with probability approaching one uniformly over $\mathcal{F},\mathcal{Q}_n$ for any $\eta>0$:
for all $u$ with
$\1{u\in A(\tilde b_{+,n}^*,\tilde d_{+,n}^*)}
-\1{u\in A(\tilde b_{+,\infty},\tilde d_{+,\infty})}\ne 0$,
there exists $\ell$ such that $|u-u^*_\ell|\le \eta$.
It follows that the above display is,
with probability approaching one uniformly over $\mathcal{F},\mathcal{Q}_n$,
bounded by a constant times the sum over $j=0,\ldots,p$ and $\ell=1,\ldots,k$ of
\begin{equation*}
  \max_{-1\le t\le 1} \left|\frac{1}{\sqrt{nh_n}}\sum_{i\colon u_\ell-\eta \le
      x_i/h_n\le u_\ell+t\eta} (x_i/h_n)^{j} u_i\right|.
\end{equation*}
By Kolmogorov's inequality \citep[see pp. 62-63 in][]{durrett_probability:_1996}, the probability of this quantity being greater than a given $\delta>0$ under a given $f,Q$ is bounded by
\begin{multline*}
  \frac{1}{\delta^2}\frac{1}{nh_n}\sum_{i\colon u_\ell-\eta \le x_i/h_n\le
    u_\ell+\eta} var_{Q}\left[(x_i/h_n)^j u_i\right]
  =\frac{1}{\delta^2}\frac{1}{nh_n}\sum_{i\colon u_\ell-\eta \le x_i/h_n\le u_\ell+\eta} (x_i/h_n)^{2j}\sigma^2(x_i)  \\
  \to
  \frac{p_{X,+}(0)\sigma^2_+(0)}{\delta^2}\int_{u_\ell^*-\eta}^{u_\ell^*+\eta}
  u^{2j}\, du,
\end{multline*}
which can be made arbitrarily small by making $\eta$ small.

For the bias formulas, we have
\begin{multline*}
  \left| \overline{\text{bias}}_n-\overline{\text{bias}}_n^* \right|
  =\frac{C}{nh_n}\left| \sum_{i=1}^n |w_n(x_i/h_n)x_i^{p}|-\sum_{i=1}^n |w_n^*(x_i/h_n)x_i^{p}| \right|  \\
  \le \frac{C}{nh_n}\sum_{i=1}^n|w_n(x_i/h_n)-w_n^*(x_i/h_n)|\cdot |x_i|^p.
\end{multline*}
This converges to zero when scaled by $n^{p/(2p+1)}$ by arguments given above.

For the variance formulas, we have
\begin{multline*}
  \left| \tilde v_n-\tilde v_n^* \right|
  =\frac{1}{(nh_n)^2}\left| \sum_{i=1}^n w_n(x_i/h_n)^2\tilde\sigma^2(x_i)-\sum_{i=1}^n w_n^*(x_i/h_n)^2\tilde\sigma^{*2}(x_i) \right|  \\
  \le \frac{1}{(nh_n)^2}\sum_{i=1}^n \left| w_n(x_i/h_n)^2\tilde\sigma^2(x_i) - w_n^*(x_i/h_n)^2\tilde\sigma^{*2}(x_i) \right|  \\
  \le \frac{1}{nh_n}\max_{|x|\le B } \left| w_n(x)^2\tilde\sigma^2(x) -
    w_n^*(x)^2\tilde\sigma^{*2}(x) \right| \cdot \frac{1}{nh_n}\sum_{i=1}^n
  \1{|x_i/h_n|\le B}
\end{multline*}
with probability approaching one where $B $ is a bound on the support of
$w_n(x)$ and $w_n^*(x)$ that holds with probability approaching one. Since
$\frac{1}{nh_n}\sum_{i=1}^n \1{|x_i/h_n|\le B}$ converges to a constant by
Assumption~\ref{rd_xs_assump} and
$\tilde v_n=n^{-2p/(2p+1)} v_\infty(1+o(1))=(nh_n)^{-1}v_\infty(1+o(1))$,
dividing the above display by $\tilde v_n$ gives an expression that is bounded
by a constant times
$\max_{|x|\le B h_n} \left| w_n(x)^2\tilde\sigma^2(x) -
  w_n^*(x)^2\tilde\sigma^{*2}(x) \right|$, which converges uniformly in
probability to zero.
\end{proof}

We are now ready to prove Theorem~\ref{rd_optimal_estimator_thm}. First,
consider the case with $\tilde\sigma(\cdot)$ is deterministic and
Assumption~\ref{prespecified_tilde_sigma_assump} holding. By
Lemma~\ref{rd_modulus_max_convergence_lemma}, $\delta_n\to \delta_\infty$. By
Lemma~\ref{rd_uniform_modulus_convergence_lemma}, it then follows that, under
Assumption~\ref{prespecified_tilde_sigma_assump},
$n^{p/(2p+1)}w_n(\delta_n)\to \omega_\infty(\delta_\infty)$ so that
Lemma~\ref{rd_opt_kern_limit_lemma} applies to show that
Assumption~\ref{rd_kern_assump} holds with
$k^+(x)=k_{\tilde\sigma_+(0)}^+(x;\tilde b_{+,\infty},\tilde d_{+,\infty})$ and
$k^-(x)=k_{\tilde\sigma_-(0)}^-(x;\tilde b_{-,\infty}, \tilde{d}_{-,\infty})$,
where
$(\tilde b_{+,\infty},\tilde d_{+,\infty}, \tilde{b}_{-,\infty},
\tilde{d}_{-,\infty})$ minimize
$G_\infty(\tilde{b}_{+,\infty}, \tilde d_{+,\infty}, \tilde{b}_{-,\infty},
\tilde d_{-,\infty})$ subject to
$\tilde{b}_{+,\infty}+ \tilde b_{-,\infty}=\omega_\infty(\delta_\infty)/2$. The
coverage statements and convergence of $n^{p/(2p+1)}\hat \chi$ then follow from
Theorem~\ref{rd_limit_high_level_thm} and by calculating
$\overline{\text{bias}}_\infty$ and $v_\infty$ in terms of the limiting modulus.

We now prove the optimality statements (under which the assumption was made that, for each $n$, there exists a $Q\in\mathcal{Q}_n$ such that the errors are normally distributed).
In this case, for any $\eta>0$, if a linear estimator $\tilde L$ and constant $\chi$ satisfy
\begin{equation*}
  \inf_{f\in\mathcal{F},Q\in\mathcal{Q}_n}
  P\Big(Lf\in\{\tilde L\pm n^{-p/(2p+1)} \chi \}
  \Big)\ge 1-\alpha-\eta,
\end{equation*}
we must have $\chi\ge
\sup_{\delta>0}\frac{n^{p/(2p+1)}\omega_{\sigma(\cdot),n}(\delta)}{\delta}\chi_{A,\alpha+\eta}(\delta/2)$
by the results of \citet{donoho94} (using the characterization of optimal
half-length at the beginning of \Cref{sec:sap:asym_efficiency_bounds}). This
converges to
$\sup_{\delta>0}\frac{\omega_{\infty}(\delta)}{\delta}\chi_{A,\alpha+\eta}(\delta/2)$
by Lemma~\ref{rd_modulus_max_convergence_lemma}. If
$\liminf_n\inf_{f\in\mathcal{F},Q\in\mathcal{Q}_n}P(Lf\in\{\tilde L\pm
n^{-p/(2p+1)} \chi \} )\ge 1-\alpha$, then, for any $\eta>0$, the above display
must hold for large enough $n$, so that $\chi\ge \lim_{\eta\downarrow
  0}\sup_{\delta>0}\frac{\omega_{\infty}(\delta)}{\delta}\chi_{A,\alpha+\eta}(\delta/2)
=\sup_{\delta>0}\frac{\omega_{\infty}(\delta)}{\delta}\chi_{A,\alpha}(\delta/2)$
(the limit with respect to $\eta$ follows since there exist
$0<\underline\delta<\overline\delta<\infty$ such that the supremum over $\delta$
is taken on $[\underline\delta,\overline\delta]$ for $\eta$ in a neighborhood of
zero, and since $\chi_{A,\alpha}(\delta/2)$ is continuous with respect to
$\alpha$ uniformly over $\delta$ in compact sets).

For the asymptotic efficiency bound regarding expected length among all
confidence intervals, note that, for any $\eta>0$, any CI satisfying the
asymptotic coverage requirement must be a $1-\alpha-\eta$ CI for large enough
$n$, which means that, since the CI is valid under the $Q_n\in\mathcal{Q}_n$
where the errors are normal, the expected length of the CI at $f=0$ and this
$Q_n$ scaled by $n^{p/(2p+1)}$ is at least
\begin{equation*}
(1-\alpha-\eta)E\left[n^{p/(2p+1)}\omega_{\sigma(\cdot),n}(2(z_{1-\alpha-\eta}-Z))|Z\le z_{1-\alpha-\eta}\right]
\end{equation*}
by Corollary~\ref{th:centrosymmetric_adaptation_twosided}. This converges to
$(1-\alpha-\eta)E\left[\omega_{\infty}(2(z_{1-\alpha-\eta}-Z))\mid Z\le
  z_{1-\alpha-\eta}\right]$ by Lemma~\ref{rd_modulus_max_convergence_lemma}. The
result follows from taking $\eta\to 0$ and using the dominated convergence
theorem, and using the fact that
$\omega_{\infty}(\delta)=\omega_{\infty}(1)\delta^{2p/(2p+1)}$. The asymptotic
efficiency bounds for the feasible one-sided CI follow from similar arguments,
using Theorem~\ref{th:one_side_adapt_thm} and
Corollary~\ref{th:centrosymmetric_adaptation_corollary} along with
Theorem~\ref{rd_limit_high_level_thm} and
Lemma~\ref{modulus_derivative_convergence_lemma}.

In the case where Assumption~\ref{estimated_tilde_sigma_assump} holds rather
than Assumption~\ref{prespecified_tilde_sigma_assump}, it follows from
Lemma~\ref{rd_modulus_max_convergence_lemma} that
$\delta_n\underset{\mathcal{F},\mathcal{Q}_n}{\overset{p}{\to}} \delta_\infty$.
Then, by Lemma~\ref{rd_estimated_sigma_equiv_lemma}, the conditions
in the last display of Theorem~\ref{rd_limit_high_level_thm} hold with
$\hat L_{\delta_n,\tilde\sigma(\cdot)}$ playing the role of $\hat L^*$ and
$\hat L_{\delta_n,\sigma(\cdot)}$ playing the role of $\hat L$. The results then
follow from Theorem~\ref{rd_limit_high_level_thm} and the arguments above
applied to the CIs based on $\hat L_{\delta_n,\sigma(\cdot)}$.
\end{appendices}

\bibliography{np-testing-library}

\begin{table}[p]
  \centering \renewcommand{\arraystretch}{1.2}
  \begin{tabular}{lrrrrrr@{}}
    &\multicolumn{3}{c}{$\sigma^{2}=0.1295$}&
                                              \multicolumn{3}{c}{$\sigma^{2}=4\cdot 0.1295$}\\
    \cmidrule(rl){2-4}\cmidrule(rl){5-7}
    CI method & Cov. (\%) & Bias & RL& Cov. (\%) & Bias & RL\\
    \midrule
    \multicolumn{3}{@{}l}{Design 1, $(b_{1}, b_{2})=(0.45,0.75)$}\\
    \cmidrule(r){1-3}
    Conventional, $\hat{h}_{IK}$ &     10.1&  -0.098&   0.54&    81.7&  -0.099&   0.72\\
    RBC, $\hat{h}_{IK},\rho=1$& 64.4&  -0.049&   0.80&    93.9&  -0.050&   1.06\\
    Conventional, $\hat{h}_{CCT}$ &    91.2&  -0.010&   1.01&    92.7&  -0.010&   1.26 \\
    RBC, $\hat{h}_{CCT}$&       93.7&   0.003&   1.18&    93.6&   0.007&   1.48 \\
    FLCI, $C=1$ &               94.6&  -0.024&    1  &    94.9&  -0.069&    1    \\
    FLCI, $C=3$ &              96.7 & -0.009 &  1.25 &   96.5 & -0.028 &  1.25\\
    \multicolumn{3}{@{}l}{Design 2, $(b_{1}, b_{2})=(0.4,0.9)$}\\\cmidrule(r){1-3}
    Conventional, $\hat{h}_{IK}$ &     54.2&  -0.063&   0.68&    89.6&  -0.085&   0.77  \\
    RBC, $\hat{h}_{IK},\rho=1$& 94.8&  -0.006&   1.00&    95.9&  -0.043&   1.13   \\
    Conventional, $\hat{h}_{CCT}$ &    91.4&  -0.009&   1.02&    92.7&  -0.009&   1.26   \\
    RBC, $\hat{h}_{CCT}$&       93.6&   0.003&   1.19&    93.6&   0.007&   1.49   \\
    FLCI, $C=1$ &               94.5&  -0.024&    1  &     95.0 &  -0.065&    1     \\
    FLCI, $C=3$ &               96.8&  -0.009&   1.25&    96.5&  -0.028&   1.25   \\
    \multicolumn{3}{@{}l}{Design 3, $(b_{1}, b_{2})=(0.25,0.65)$}\\\cmidrule(r){1-3}
    Conventional, $\hat{h}_{IK}$ &     87.8&  -0.030 &  0.74&    91.4&  -0.009&   0.76 \\
    RBC, $\hat{h}_{IK},\rho=1$& 94.8&  -0.014 &  1.09&    95.0&  -0.044&   1.12  \\
    Conventional, $\hat{h}_{CCT}$ &    90.9&  -0.014 &  0.97&    92.8&  -0.013&   1.25  \\
    RBC, $\hat{h}_{CCT}$&       92.2&  -0.009 &  1.14&    93.5&  -0.007&   1.48  \\
    FLCI, $C=1$ &               94.7&  -0.022 &   1  &    96.7&  -0.028&    1    \\
    FLCI, $C=3$ &               96.8&  -0.009 &  1.25&    96.6&  -0.025&   1.25  \\
    \multicolumn{3}{@{}l}{Design 4, $f(x)=0$}\\\cmidrule(r){1-3}
    Conventional, $\hat{h}_{IK}$ &     93.2&  0.000&    0.54&  93.2  &-0.001   & 0.72  \\
    RBC, $\hat{h}_{IK},\rho=1$& 95.2&  0.000&    0.80&  95.2  & 0.001   & 1.06   \\
    Conventional, $\hat{h}_{CCT}$ &    93.1&  0.001&    0.94&  93.1  & 0.003   & 1.25   \\
    RBC, $\hat{h}_{CCT}$&       93.5&  0.001&    1.12&  93.5  & 0.004   & 1.48   \\
    FLCI, $C=1$ &               96.8&  0.001&    1   &  96.9  & 0.000   &  1     \\
    FLCI, $C=3$ &               96.8&  0.001&    1.25&  96.8  & 0.002   & 1.25   \\
  \end{tabular}
  \caption{Monte Carlo simulation, $C=1$. Coverage (``Cov'') and relative length
    relative to optimal fixed-length CI for $\mathcal{F}_{RDH,2}(1)$ (``RL'').
    ``Bias'' refers to bias of estimator around which CI is centered. 11,000
    simulation draws.}\label{tab:mc2-C1}
\end{table}
\begin{table}[p]
  \centering \renewcommand{\arraystretch}{1.2}
  \begin{tabular}{lrrrrrr@{}}
    &\multicolumn{3}{c}{$\sigma^{2}=0.1295$}&
                                              \multicolumn{3}{c}{$\sigma^{2}=4\cdot 0.1295$}\\
    \cmidrule(rl){2-4}\cmidrule(rl){5-7}
    CI method & Cov. (\%) & Bias & RL& Cov. (\%) & Bias & RL\\
    \midrule
    \multicolumn{3}{@{}l}{Design 1, $(b_{1}, b_{2})=(0.45,0.75)$}\\
    \cmidrule(r){1-3}
    Conventional, $\hat{h}_{IK}$ &      0.1&  -0.292 &  0.44 &   22.4 & -0.296 &  0.58\\
    RBC, $\hat{h}_{IK},\rho=1$& 27.1&  -0.127 &  0.65 &   77.8 & -0.149 &  0.85\\
    Conventional, $\hat{h}_{CCT}$ &    89.3&  -0.019 &  0.94 &   91.6 & -0.031 &  1.05 \\
    RBC, $\hat{h}_{CCT}$&       93.7&   0.004 &  1.06 &   93.7 &  0.012 &  1.22 \\
    FLCI, $C=1$ &               67.3&  -8.078 &  0.80 &   73.1 & -0.209 &  0.80 \\
    FLCI, $C=3$ &               94.5&  -0.032 &   1   &   94.6 & -0.089 &   1  \\
    \multicolumn{3}{@{}l}{Design 2, $(b_{1}, b_{2})=(0.4,0.9)$}\\\cmidrule(r){1-3}
    Conventional, $\hat{h}_{IK}$ &     60.0&  -0.071 &  0.71 &   71.4 & -0.193 &  0.72  \\
    RBC, $\hat{h}_{IK},\rho=1$& 93.5&   0.000 &  1.04 &   95.1 & -0.020 &  1.05   \\
    Conventional, $\hat{h}_{CCT}$ &    89.7&  -0.018 &  0.95 &   91.7 & -0.029 &  1.05   \\
    RBC, $\hat{h}_{CCT}$&       93.6&   0.004 &  1.09 &   93.6 &  0.012 &  1.24   \\
    FLCI, $C=1$ &               70.3&  -0.073 &  0.80 &   76.3 & -0.197 &  0.80    \\
    FLCI, $C=3$ &               94.3&  -0.030 &   1   &   94.6 & -0.089 &   1     \\
    \multicolumn{3}{@{}l}{Design 3, $(b_{1}, b_{2})=(0.25,0.65)$}\\\cmidrule(r){1-3}
    Conventional, $\hat{h}_{IK}$ &     79.9&  -0.052&   0.76&    89.2&  -0.085&   0.73 \\
    RBC, $\hat{h}_{IK},\rho=1$& 93.3&  0.001 &   1.13&    94.6&  -0.072&   1.07 \\
    Conventional, $\hat{h}_{CCT}$ &    80.7&  -0.032&   0.87&    91.8&  -0.042&   1.01 \\
    RBC, $\hat{h}_{CCT}$&       86.2&  -0.017&   1.00&    92.7&  -0.027&   1.20 \\
    FLCI, $C=1$ &               73.5&  -0.069&   0.8 &    93.8&  -0.084&   0.80 \\
    FLCI, $C=3$ &               94.4&  -0.030&    1  &    95.1&  -0.078&    1   \\
    \multicolumn{3}{@{}l}{Design 5, $f(x)=0$}\\\cmidrule(r){1-3}
    Conventional, $\hat{h}_{IK}$ &     93.2&  0.000 &   0.43&    93.2& -0.001 &   0.57  \\
    RBC, $\hat{h}_{IK},\rho=1$& 95.2&  0.000 &   0.64&    95.2&  0.001 &   0.85   \\
    Conventional, $\hat{h}_{CCT}$ &    93.1&  0.001 &   0.75&    93.1&  0.003 &   1.00   \\
    RBC, $\hat{h}_{CCT}$&       93.5&  0.001 &   0.89&    93.5&  0.004 &   1.18   \\
    FLCI, $C=1$ &               96.8&  0.001 &   0.80&    96.9&  0.000 &   0.80   \\
    FLCI, $C=3$ &               96.8&  0.001 &    1  &    96.7&  0.002 &    1     \\
  \end{tabular}
  \caption{Monte Carlo simulation, $C=3$. Coverage (``Cov'') and relative length
    relative to optimal fixed-length CI for $\mathcal{F}_{RDH,2}(1)$ (``RL'').
    ``Bias'' refers to bias of estimator around which CI is centered. 11,000
    simulation draws.}\label{tab:mc2-C3}
\end{table}

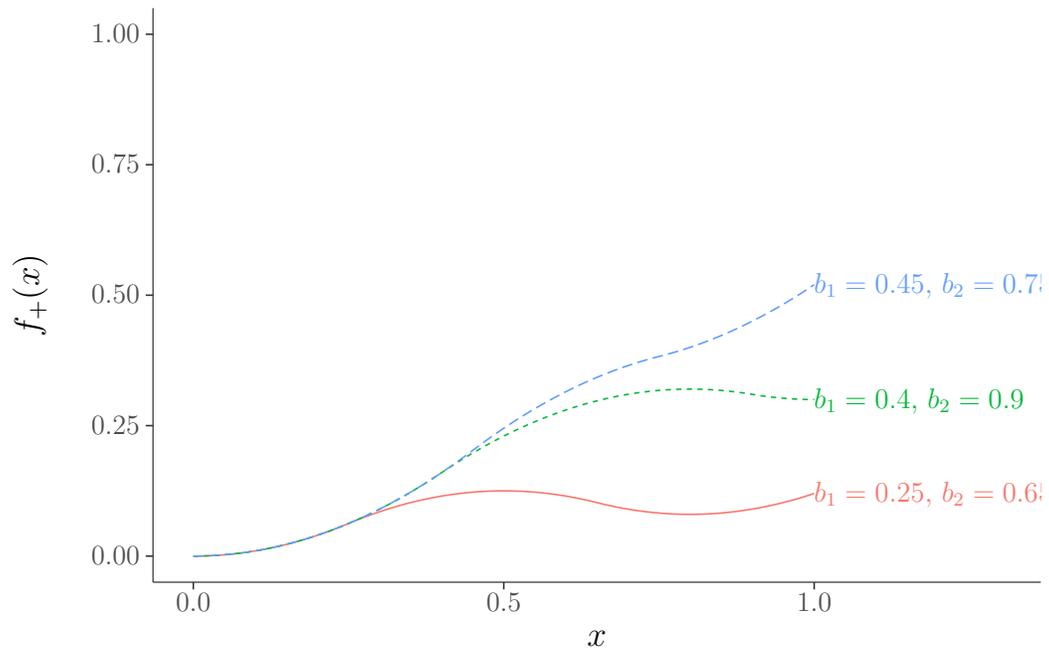
\begin{figure}[htp]
  \centering \input{sim-2.tex}
  \caption{Regression function for Monte Carlo simulation, Designs 1--3, and
    $C=1$. Knots $b_{1}=0.45, b_{2}=0.75$ correspond to Design 1,
    $b_{1}=0.4, b_{2}=0.9$ to Design 2, and $b_{1}=0.25, b_{2}=0.65$ to Design
    3.}\label{fig:reg-fkt}
\end{figure}

%% file: preamble.tex
\usepackage[utf8]{inputenc} %

\usepackage{tikz} 

\usepackage{amsmath} %
\usepackage{amsfonts} %
\usepackage{amssymb} %
\usepackage{amsthm} %
\usepackage{graphicx} %
\usepackage{fullpage} %
\usepackage{setspace} %
\usepackage{verbatim} %
\usepackage{natbib} %
\usepackage{booktabs} %
\usepackage{cleveref}
\crefname{appsec}{Appendix}{Appendices} 
\crefname{appsubsec}{Appendix}{Appendices} %
\crefname{sappsec}{Supplemental Appendix}{Supplemental Appendices}
\crefname{sappsubsec}{Supplemental Appendix}{Supplemental Appendices}

\newtheorem{theorem}{Theorem}[section] %
\newtheorem{lemma}{Lemma}[section] %
\newtheorem{corollary}{Corollary}[section] %

\theoremstyle{definition}
\newtheorem{remark}{Remark}[section] %

\newcommand{\norm}[1]{\lVert#1\rVert}
\newcommand{\Rlower}[1]{\ensuremath q_{#1}} 
\newcommand{\Rfl}[1]{\ensuremath\chi_{#1}} 
\DeclareMathOperator{\cv}{cv} 

\DeclareMathOperator{\bias}{bias} 
\DeclareMathOperator{\maxbias}{\overline{bias}} 
\DeclareMathOperator{\minbias}{\underline{bias}} 
\DeclareMathOperator*{\argmin}{argmin} 
\newcommand{\dd}{\mathrm{d}} 
\DeclareMathOperator{\sd}{sd} 

\RequirePackage{mathtools}

\DeclarePairedDelimiter\abs{\lvert}{\rvert}

\DeclarePairedDelimiter\hol{(}{]}
\DeclarePairedDelimiter\hor{[}{)}

\DeclarePairedDelimiter\1{1(}{)}

%% file: lipschitz-rd-small.tex
\begin{tikzpicture}[x=1pt,y=1pt]
\definecolor{fillColor}{RGB}{255,255,255}
\path[use as bounding box,fill=fillColor,fill opacity=0.00] (0,0) rectangle (289.08,216.81);
\begin{scope}
\path[clip] (  0.00,  0.00) rectangle (289.08,216.81);
\definecolor{drawColor}{RGB}{255,255,255}
\definecolor{fillColor}{RGB}{255,255,255}

\path[draw=drawColor,line width= 0.6pt,line join=round,line cap=round,fill=fillColor] (  0.00,  0.00) rectangle (289.08,216.81);
\end{scope}
\begin{scope}
\path[clip] ( 76.58, 30.14) rectangle (283.58,211.31);
\definecolor{fillColor}{RGB}{255,255,255}

\path[fill=fillColor] ( 76.58, 30.14) rectangle (283.58,211.31);
\definecolor{drawColor}{RGB}{128,128,128}

\path[draw=drawColor,line width= 0.6pt,dash pattern=on 2pt off 2pt ,line join=round] ( 85.99, 74.52) --
	( 86.93, 74.52) --
	( 87.87, 74.52) --
	( 88.81, 74.52) --
	( 89.75, 74.52) --
	( 90.70, 74.52) --
	( 91.64, 74.52) --
	( 92.58, 74.52) --
	( 93.52, 74.52) --
	( 94.46, 74.52) --
	( 95.40, 74.52) --
	( 96.34, 74.52) --
	( 97.28, 74.52) --
	( 98.22, 74.52) --
	( 99.16, 74.52) --
	(100.10, 74.52) --
	(101.05, 74.52) --
	(101.99, 74.52) --
	(102.93, 74.52) --
	(103.87, 74.52) --
	(104.81, 74.52) --
	(105.75, 74.52) --
	(106.69, 74.52) --
	(107.63, 74.52) --
	(108.57, 74.52) --
	(109.51, 74.52) --
	(110.45, 74.52) --
	(111.40, 74.52) --
	(112.34, 74.52) --
	(113.28, 74.52) --
	(114.22, 74.52) --
	(115.16, 74.52) --
	(116.10, 74.52) --
	(117.04, 74.52) --
	(117.98, 74.52) --
	(118.92, 74.52) --
	(119.86, 74.52) --
	(120.80, 74.52) --
	(121.75, 74.52) --
	(122.69, 74.52) --
	(123.63, 74.52) --
	(124.57, 74.52) --
	(125.51, 74.52) --
	(126.45, 74.52) --
	(127.39, 74.52) --
	(128.33, 74.52) --
	(129.27, 74.52) --
	(130.21, 74.52) --
	(131.15, 74.52) --
	(132.10, 74.52) --
	(133.04, 74.52) --
	(133.98, 73.92) --
	(134.92, 73.33) --
	(135.86, 72.74) --
	(136.80, 72.15) --
	(137.74, 71.55) --
	(138.68, 70.96) --
	(139.62, 70.37) --
	(140.56, 69.78) --
	(141.50, 69.18) --
	(142.44, 68.59) --
	(143.39, 68.00) --
	(144.33, 67.41) --
	(145.27, 66.81) --
	(146.21, 66.22) --
	(147.15, 65.63) --
	(148.09, 65.04) --
	(149.03, 64.44) --
	(149.97, 63.85) --
	(150.91, 63.26) --
	(151.85, 62.67) --
	(152.79, 62.07) --
	(153.74, 61.48) --
	(154.68, 60.89) --
	(155.62, 60.30) --
	(156.56, 59.70) --
	(157.50, 59.11) --
	(158.44, 58.52) --
	(159.38, 57.93) --
	(160.32, 57.34) --
	(161.26, 56.74) --
	(162.20, 56.15) --
	(163.14, 55.56) --
	(164.09, 54.97) --
	(165.03, 54.37) --
	(165.97, 53.78) --
	(166.91, 53.19) --
	(167.85, 52.60) --
	(168.79, 52.00) --
	(169.73, 51.41) --
	(170.67, 50.82) --
	(171.61, 50.23) --
	(172.55, 49.63) --
	(173.49, 49.04) --
	(174.44, 48.45) --
	(175.38, 47.86) --
	(176.32, 47.26) --
	(177.26, 46.67) --
	(178.20, 46.08) --
	(179.14, 45.49);

\path[draw=drawColor,line width= 0.6pt,dash pattern=on 2pt off 2pt ,line join=round] (180.08,195.37) --
	(181.02,194.78) --
	(181.96,194.19) --
	(182.90,193.60) --
	(183.84,193.00) --
	(184.79,192.41) --
	(185.73,191.82) --
	(186.67,191.23) --
	(187.61,190.63) --
	(188.55,190.04) --
	(189.49,189.45) --
	(190.43,188.86) --
	(191.37,188.26) --
	(192.31,187.67) --
	(193.25,187.08) --
	(194.19,186.49) --
	(195.14,185.89) --
	(196.08,185.30) --
	(197.02,184.71) --
	(197.96,184.12) --
	(198.90,183.52) --
	(199.84,182.93) --
	(200.78,182.34) --
	(201.72,181.75) --
	(202.66,181.15) --
	(203.60,180.56) --
	(204.54,179.97) --
	(205.49,179.38) --
	(206.43,178.79) --
	(207.37,178.19) --
	(208.31,177.60) --
	(209.25,177.01) --
	(210.19,176.42) --
	(211.13,175.82) --
	(212.07,175.23) --
	(213.01,174.64) --
	(213.95,174.05) --
	(214.89,173.45) --
	(215.84,172.86) --
	(216.78,172.27) --
	(217.72,171.68) --
	(218.66,171.08) --
	(219.60,170.49) --
	(220.54,169.90) --
	(221.48,169.31) --
	(222.42,168.71) --
	(223.36,168.12) --
	(224.30,167.53) --
	(225.24,166.94) --
	(226.19,166.34) --
	(227.13,165.75) --
	(228.07,165.75) --
	(229.01,165.75) --
	(229.95,165.75) --
	(230.89,165.75) --
	(231.83,165.75) --
	(232.77,165.75) --
	(233.71,165.75) --
	(234.65,165.75) --
	(235.59,165.75) --
	(236.53,165.75) --
	(237.48,165.75) --
	(238.42,165.75) --
	(239.36,165.75) --
	(240.30,165.75) --
	(241.24,165.75) --
	(242.18,165.75) --
	(243.12,165.75) --
	(244.06,165.75) --
	(245.00,165.75) --
	(245.94,165.75) --
	(246.88,165.75) --
	(247.83,165.75) --
	(248.77,165.75) --
	(249.71,165.75) --
	(250.65,165.75) --
	(251.59,165.75) --
	(252.53,165.75) --
	(253.47,165.75) --
	(254.41,165.75) --
	(255.35,165.75) --
	(256.29,165.75) --
	(257.23,165.75) --
	(258.18,165.75) --
	(259.12,165.75) --
	(260.06,165.75) --
	(261.00,165.75) --
	(261.94,165.75) --
	(262.88,165.75) --
	(263.82,165.75) --
	(264.76,165.75) --
	(265.70,165.75) --
	(266.64,165.75) --
	(267.58,165.75) --
	(268.53,165.75) --
	(269.47,165.75) --
	(270.41,165.75) --
	(271.35,165.75) --
	(272.29,165.75) --
	(273.23,165.75) --
	(274.17,165.75);
\definecolor{drawColor}{RGB}{0,0,0}

\path[draw=drawColor,line width= 0.6pt,line join=round] ( 85.99, 75.70) --
	( 86.93, 75.70) --
	( 87.87, 75.70) --
	( 88.81, 75.70) --
	( 89.75, 75.70) --
	( 90.70, 75.70) --
	( 91.64, 75.70) --
	( 92.58, 75.70) --
	( 93.52, 75.70) --
	( 94.46, 75.70) --
	( 95.40, 75.70) --
	( 96.34, 75.70) --
	( 97.28, 75.70) --
	( 98.22, 75.70) --
	( 99.16, 75.70) --
	(100.10, 75.70) --
	(101.05, 75.70) --
	(101.99, 75.70) --
	(102.93, 75.70) --
	(103.87, 75.70) --
	(104.81, 75.70) --
	(105.75, 75.70) --
	(106.69, 75.70) --
	(107.63, 75.70) --
	(108.57, 75.70) --
	(109.51, 75.70) --
	(110.45, 75.70) --
	(111.40, 75.70) --
	(112.34, 75.70) --
	(113.28, 75.70) --
	(114.22, 75.70) --
	(115.16, 75.70) --
	(116.10, 75.70) --
	(117.04, 75.70) --
	(117.98, 75.70) --
	(118.92, 75.70) --
	(119.86, 75.70) --
	(120.80, 75.70) --
	(121.75, 75.70) --
	(122.69, 75.70) --
	(123.63, 75.70) --
	(124.57, 75.70) --
	(125.51, 75.70) --
	(126.45, 75.70) --
	(127.39, 75.70) --
	(128.33, 75.70) --
	(129.27, 75.70) --
	(130.21, 75.70) --
	(131.15, 75.70) --
	(132.10, 75.70) --
	(133.04, 75.70) --
	(133.98, 76.29) --
	(134.92, 76.89) --
	(135.86, 77.48) --
	(136.80, 78.07) --
	(137.74, 78.66) --
	(138.68, 79.26) --
	(139.62, 79.85) --
	(140.56, 80.44) --
	(141.50, 81.03) --
	(142.44, 81.63) --
	(143.39, 82.22) --
	(144.33, 82.81) --
	(145.27, 83.40) --
	(146.21, 83.99) --
	(147.15, 84.59) --
	(148.09, 85.18) --
	(149.03, 85.77) --
	(149.97, 86.36) --
	(150.91, 86.96) --
	(151.85, 87.55) --
	(152.79, 88.14) --
	(153.74, 88.73) --
	(154.68, 89.33) --
	(155.62, 89.92) --
	(156.56, 90.51) --
	(157.50, 91.10) --
	(158.44, 91.70) --
	(159.38, 92.29) --
	(160.32, 92.88) --
	(161.26, 93.47) --
	(162.20, 94.07) --
	(163.14, 94.66) --
	(164.09, 95.25) --
	(165.03, 95.84) --
	(165.97, 96.44) --
	(166.91, 97.03) --
	(167.85, 97.62) --
	(168.79, 98.21) --
	(169.73, 98.81) --
	(170.67, 99.40) --
	(171.61, 99.99) --
	(172.55,100.58) --
	(173.49,101.18) --
	(174.44,101.77) --
	(175.38,102.36) --
	(176.32,102.95) --
	(177.26,103.55) --
	(178.20,104.14) --
	(179.14,104.73);

\path[draw=drawColor,line width= 0.6pt,line join=round] (180.08,134.94) --
	(181.02,135.54) --
	(181.96,136.13) --
	(182.90,136.72) --
	(183.84,137.31) --
	(184.79,137.91) --
	(185.73,138.50) --
	(186.67,139.09) --
	(187.61,139.68) --
	(188.55,140.28) --
	(189.49,140.87) --
	(190.43,141.46) --
	(191.37,142.05) --
	(192.31,142.65) --
	(193.25,143.24) --
	(194.19,143.83) --
	(195.14,144.42) --
	(196.08,145.02) --
	(197.02,145.61) --
	(197.96,146.20) --
	(198.90,146.79) --
	(199.84,147.39) --
	(200.78,147.98) --
	(201.72,148.57) --
	(202.66,149.16) --
	(203.60,149.76) --
	(204.54,150.35) --
	(205.49,150.94) --
	(206.43,151.53) --
	(207.37,152.13) --
	(208.31,152.72) --
	(209.25,153.31) --
	(210.19,153.90) --
	(211.13,154.50) --
	(212.07,155.09) --
	(213.01,155.68) --
	(213.95,156.27) --
	(214.89,156.86) --
	(215.84,157.46) --
	(216.78,158.05) --
	(217.72,158.64) --
	(218.66,159.23) --
	(219.60,159.83) --
	(220.54,160.42) --
	(221.48,161.01) --
	(222.42,161.60) --
	(223.36,162.20) --
	(224.30,162.79) --
	(225.24,163.38) --
	(226.19,163.97) --
	(227.13,164.57) --
	(228.07,164.57) --
	(229.01,164.57) --
	(229.95,164.57) --
	(230.89,164.57) --
	(231.83,164.57) --
	(232.77,164.57) --
	(233.71,164.57) --
	(234.65,164.57) --
	(235.59,164.57) --
	(236.53,164.57) --
	(237.48,164.57) --
	(238.42,164.57) --
	(239.36,164.57) --
	(240.30,164.57) --
	(241.24,164.57) --
	(242.18,164.57) --
	(243.12,164.57) --
	(244.06,164.57) --
	(245.00,164.57) --
	(245.94,164.57) --
	(246.88,164.57) --
	(247.83,164.57) --
	(248.77,164.57) --
	(249.71,164.57) --
	(250.65,164.57) --
	(251.59,164.57) --
	(252.53,164.57) --
	(253.47,164.57) --
	(254.41,164.57) --
	(255.35,164.57) --
	(256.29,164.57) --
	(257.23,164.57) --
	(258.18,164.57) --
	(259.12,164.57) --
	(260.06,164.57) --
	(261.00,164.57) --
	(261.94,164.57) --
	(262.88,164.57) --
	(263.82,164.57) --
	(264.76,164.57) --
	(265.70,164.57) --
	(266.64,164.57) --
	(267.58,164.57) --
	(268.53,164.57) --
	(269.47,164.57) --
	(270.41,164.57) --
	(271.35,164.57) --
	(272.29,164.57) --
	(273.23,164.57) --
	(274.17,164.57);

\node[text=drawColor,anchor=base,inner sep=0pt, outer sep=0pt, scale=  1.00] at (174.39,125.99) {$f^{*}$};
\definecolor{drawColor}{RGB}{128,128,128}

\node[text=drawColor,anchor=base,inner sep=0pt, outer sep=0pt, scale=  1.00] at (174.39,197.44) {$g^{*}$};
\end{scope}
\begin{scope}
\path[clip] (  0.00,  0.00) rectangle (289.08,216.81);
\definecolor{drawColor}{RGB}{0,0,0}

\path[draw=drawColor,line width= 0.6pt,line join=round] ( 76.58, 30.14) --
	( 76.58,211.31);
\end{scope}
\begin{scope}
\path[clip] (  0.00,  0.00) rectangle (289.08,216.81);
\definecolor{drawColor}{gray}{0.30}

\node[text=drawColor,anchor=base east,inner sep=0pt, outer sep=0pt, scale=  0.88] at ( 71.63, 43.64) {$-Ch_{-}$};

\node[text=drawColor,anchor=base east,inner sep=0pt, outer sep=0pt, scale=  0.88] at ( 71.63, 73.26) {$0$};

\node[text=drawColor,anchor=base east,inner sep=0pt, outer sep=0pt, scale=  0.88] at ( 71.63,102.88) {$Ch_{-}$};

\node[text=drawColor,anchor=base east,inner sep=0pt, outer sep=0pt, scale=  0.88] at ( 71.63,132.51) {$Ch_{-}+L_{0}$};

\node[text=drawColor,anchor=base east,inner sep=0pt, outer sep=0pt, scale=  0.88] at ( 71.63,162.13) {$b+L_{0}$};

\node[text=drawColor,anchor=base east,inner sep=0pt, outer sep=0pt, scale=  0.88] at ( 71.63,191.75) {$b+L_{0}+Ch_{+}$};
\end{scope}
\begin{scope}
\path[clip] (  0.00,  0.00) rectangle (289.08,216.81);
\definecolor{drawColor}{gray}{0.20}

\path[draw=drawColor,line width= 0.6pt,line join=round] ( 73.83, 46.67) --
	( 76.58, 46.67);

\path[draw=drawColor,line width= 0.6pt,line join=round] ( 73.83, 76.29) --
	( 76.58, 76.29);

\path[draw=drawColor,line width= 0.6pt,line join=round] ( 73.83,105.92) --
	( 76.58,105.92);

\path[draw=drawColor,line width= 0.6pt,line join=round] ( 73.83,135.54) --
	( 76.58,135.54);

\path[draw=drawColor,line width= 0.6pt,line join=round] ( 73.83,165.16) --
	( 76.58,165.16);

\path[draw=drawColor,line width= 0.6pt,line join=round] ( 73.83,194.78) --
	( 76.58,194.78);
\end{scope}
\begin{scope}
\path[clip] (  0.00,  0.00) rectangle (289.08,216.81);
\definecolor{drawColor}{RGB}{0,0,0}

\path[draw=drawColor,line width= 0.6pt,line join=round] ( 76.58, 30.14) --
	(283.58, 30.14);
\end{scope}
\begin{scope}
\path[clip] (  0.00,  0.00) rectangle (289.08,216.81);
\definecolor{drawColor}{gray}{0.20}

\path[draw=drawColor,line width= 0.6pt,line join=round] (133.04, 27.39) --
	(133.04, 30.14);

\path[draw=drawColor,line width= 0.6pt,line join=round] (180.08, 27.39) --
	(180.08, 30.14);

\path[draw=drawColor,line width= 0.6pt,line join=round] (227.13, 27.39) --
	(227.13, 30.14);
\end{scope}
\begin{scope}
\path[clip] (  0.00,  0.00) rectangle (289.08,216.81);
\definecolor{drawColor}{gray}{0.30}

\node[text=drawColor,anchor=base,inner sep=0pt, outer sep=0pt, scale=  0.88] at (133.04, 19.13) {$-h_{-}$};

\node[text=drawColor,anchor=base,inner sep=0pt, outer sep=0pt, scale=  0.88] at (180.08, 19.13) {$0$};

\node[text=drawColor,anchor=base,inner sep=0pt, outer sep=0pt, scale=  0.88] at (227.13, 19.13) {$h_{+}$};
\end{scope}
\begin{scope}
\path[clip] (  0.00,  0.00) rectangle (289.08,216.81);
\definecolor{drawColor}{RGB}{0,0,0}

\node[text=drawColor,anchor=base,inner sep=0pt, outer sep=0pt, scale=  1.10] at (180.08,  6.06) {$x$};
\end{scope}
\end{tikzpicture}

%% file: adaptivity-bounds-bw.tex
\begin{tikzpicture}[x=1pt,y=1pt]
\definecolor{fillColor}{RGB}{255,255,255}
\path[use as bounding box,fill=fillColor,fill opacity=0.00] (0,0) rectangle (397.48,252.94);
\begin{scope}
\path[clip] (  0.00,  0.00) rectangle (397.48,252.94);
\definecolor{drawColor}{RGB}{255,255,255}
\definecolor{fillColor}{RGB}{255,255,255}

\path[draw=drawColor,line width= 0.6pt,line join=round,line cap=round,fill=fillColor] (  0.00, -0.00) rectangle (397.48,252.94);
\end{scope}
\begin{scope}
\path[clip] ( 45.51, 30.14) rectangle (391.98,247.45);
\definecolor{fillColor}{RGB}{255,255,255}

\path[fill=fillColor] ( 45.51, 30.14) rectangle (391.98,247.45);
\definecolor{drawColor}{gray}{0.20}

\path[draw=drawColor,line width= 0.6pt,line join=round] ( 61.26,122.40) --
	( 66.99,124.13) --
	( 72.71,125.93) --
	( 78.44,127.80) --
	( 84.17,129.73) --
	( 89.89,131.72) --
	( 95.62,133.76) --
	(101.35,135.86) --
	(107.07,138.00) --
	(112.80,140.17) --
	(118.53,142.37) --
	(124.25,144.58) --
	(129.98,146.78) --
	(135.71,148.96) --
	(141.43,151.09) --
	(147.16,153.16) --
	(152.89,155.14) --
	(158.62,157.01) --
	(164.34,158.75) --
	(170.07,160.32) --
	(175.80,161.70) --
	(181.52,162.87) --
	(187.25,163.80) --
	(192.98,164.46) --
	(198.70,164.82) --
	(204.43,164.86) --
	(210.16,164.54) --
	(215.88,163.85) --
	(221.61,162.74) --
	(227.34,161.18) --
	(233.06,159.16) --
	(238.79,156.62) --
	(244.52,153.54) --
	(250.25,149.88) --
	(255.97,145.60) --
	(261.70,140.65) --
	(267.43,135.00) --
	(273.15,128.57) --
	(278.88,121.33) --
	(284.61,113.19) --
	(290.33,104.09) --
	(296.06, 93.92) --
	(301.79, 82.59) --
	(307.51, 69.95) --
	(313.24, 55.84) --
	(318.97, 40.02);
\definecolor{drawColor}{gray}{0.80}

\path[draw=drawColor,line width= 0.6pt,dash pattern=on 2pt off 2pt ,line join=round] ( 61.26,125.14) --
	( 66.99,125.77) --
	( 72.71,126.51) --
	( 78.44,127.33) --
	( 84.17,128.26) --
	( 89.89,129.28) --
	( 95.62,130.39) --
	(101.35,131.59) --
	(107.07,132.88) --
	(112.80,134.26) --
	(118.53,135.73) --
	(124.25,137.29) --
	(129.98,138.94) --
	(135.71,140.67) --
	(141.43,142.49) --
	(147.16,144.39) --
	(152.89,146.38) --
	(158.62,148.45) --
	(164.34,150.60) --
	(170.07,152.83) --
	(175.80,155.14) --
	(181.52,157.54) --
	(187.25,160.01) --
	(192.98,162.56) --
	(198.70,165.19) --
	(204.43,167.90) --
	(210.16,170.68) --
	(215.88,173.54) --
	(221.61,176.48) --
	(227.34,179.49) --
	(233.06,182.58) --
	(238.79,185.74) --
	(244.52,188.97) --
	(250.25,192.28) --
	(255.97,195.66) --
	(261.70,199.12) --
	(267.43,202.65) --
	(273.15,206.24) --
	(278.88,209.91) --
	(284.61,213.65) --
	(290.33,217.46) --
	(296.06,221.35) --
	(301.79,225.30) --
	(307.51,229.32) --
	(313.24,233.41) --
	(318.97,237.57);
\definecolor{drawColor}{gray}{0.20}

\node[text=drawColor,anchor=base,inner sep=0pt, outer sep=0pt, scale=  0.90] at (318.97, 31.96) {Fixed-length};
\definecolor{drawColor}{gray}{0.80}

\node[text=drawColor,anchor=base,inner sep=0pt, outer sep=0pt, scale=  0.90] at (318.97,239.43) {Minimax onesided};
\end{scope}
\begin{scope}
\path[clip] (  0.00,  0.00) rectangle (397.48,252.94);
\definecolor{drawColor}{RGB}{0,0,0}

\path[draw=drawColor,line width= 0.2pt,line join=round] ( 45.51, 30.14) --
	( 45.51,247.45);
\end{scope}
\begin{scope}
\path[clip] (  0.00,  0.00) rectangle (397.48,252.94);
\definecolor{drawColor}{gray}{0.30}

\node[text=drawColor,anchor=base east,inner sep=0pt, outer sep=0pt, scale=  0.88] at ( 40.56, 80.30) {0.925};

\node[text=drawColor,anchor=base east,inner sep=0pt, outer sep=0pt, scale=  0.88] at ( 40.56,138.99) {0.950};

\node[text=drawColor,anchor=base east,inner sep=0pt, outer sep=0pt, scale=  0.88] at ( 40.56,197.67) {0.975};
\end{scope}
\begin{scope}
\path[clip] (  0.00,  0.00) rectangle (397.48,252.94);
\definecolor{drawColor}{gray}{0.20}

\path[draw=drawColor,line width= 0.6pt,line join=round] ( 42.76, 83.33) --
	( 45.51, 83.33);

\path[draw=drawColor,line width= 0.6pt,line join=round] ( 42.76,142.02) --
	( 45.51,142.02);

\path[draw=drawColor,line width= 0.6pt,line join=round] ( 42.76,200.70) --
	( 45.51,200.70);
\end{scope}
\begin{scope}
\path[clip] (  0.00,  0.00) rectangle (397.48,252.94);
\definecolor{drawColor}{RGB}{0,0,0}

\path[draw=drawColor,line width= 0.2pt,line join=round] ( 45.51, 30.14) --
	(391.98, 30.14);
\end{scope}
\begin{scope}
\path[clip] (  0.00,  0.00) rectangle (397.48,252.94);
\definecolor{drawColor}{gray}{0.20}

\path[draw=drawColor,line width= 0.6pt,line join=round] ( 61.26, 27.39) --
	( 61.26, 30.14);

\path[draw=drawColor,line width= 0.6pt,line join=round] (118.53, 27.39) --
	(118.53, 30.14);

\path[draw=drawColor,line width= 0.6pt,line join=round] (175.80, 27.39) --
	(175.80, 30.14);

\path[draw=drawColor,line width= 0.6pt,line join=round] (233.06, 27.39) --
	(233.06, 30.14);

\path[draw=drawColor,line width= 0.6pt,line join=round] (290.33, 27.39) --
	(290.33, 30.14);

\path[draw=drawColor,line width= 0.6pt,line join=round] (347.60, 27.39) --
	(347.60, 30.14);
\end{scope}
\begin{scope}
\path[clip] (  0.00,  0.00) rectangle (397.48,252.94);
\definecolor{drawColor}{gray}{0.30}

\node[text=drawColor,anchor=base,inner sep=0pt, outer sep=0pt, scale=  0.88] at ( 61.26, 19.13) {0.5};

\node[text=drawColor,anchor=base,inner sep=0pt, outer sep=0pt, scale=  0.88] at (118.53, 19.13) {0.6};

\node[text=drawColor,anchor=base,inner sep=0pt, outer sep=0pt, scale=  0.88] at (175.80, 19.13) {0.7};

\node[text=drawColor,anchor=base,inner sep=0pt, outer sep=0pt, scale=  0.88] at (233.06, 19.13) {0.8};

\node[text=drawColor,anchor=base,inner sep=0pt, outer sep=0pt, scale=  0.88] at (290.33, 19.13) {0.9};

\node[text=drawColor,anchor=base,inner sep=0pt, outer sep=0pt, scale=  0.88] at (347.60, 19.13) {1.0};
\end{scope}
\begin{scope}
\path[clip] (  0.00,  0.00) rectangle (397.48,252.94);
\definecolor{drawColor}{RGB}{0,0,0}

\node[text=drawColor,anchor=base,inner sep=0pt, outer sep=0pt, scale=  1.10] at (218.75,  6.06) {$r$};
\end{scope}
\begin{scope}
\path[clip] (  0.00,  0.00) rectangle (397.48,252.94);
\definecolor{drawColor}{RGB}{0,0,0}

\node[text=drawColor,rotate= 90.00,anchor=base,inner sep=0pt, outer sep=0pt, scale=  1.10] at ( 13.08,138.79) {Relative efficiency};
\end{scope}
\end{tikzpicture}

%% file: sim-2.tex
\begin{tikzpicture}[x=1pt,y=1pt]
\definecolor{fillColor}{RGB}{255,255,255}
\path[use as bounding box,fill=fillColor,fill opacity=0.00] (0,0) rectangle (397.48,252.94);
\begin{scope}
\path[clip] (  0.00,  0.00) rectangle (397.48,252.94);
\definecolor{drawColor}{RGB}{255,255,255}
\definecolor{fillColor}{RGB}{255,255,255}

\path[draw=drawColor,line width= 0.6pt,line join=round,line cap=round,fill=fillColor] (  0.00, -0.00) rectangle (397.48,252.94);
\end{scope}
\begin{scope}
\path[clip] ( 56.17, 30.14) rectangle (391.98,247.45);
\definecolor{fillColor}{RGB}{255,255,255}

\path[fill=fillColor] ( 56.17, 30.14) rectangle (391.98,247.45);
\definecolor{drawColor}{RGB}{248,118,109}

\path[draw=drawColor,line width= 0.6pt,line join=round] ( 71.43, 40.02) --
	( 72.60, 40.02) --
	( 73.78, 40.04) --
	( 74.95, 40.06) --
	( 76.13, 40.10) --
	( 77.30, 40.14) --
	( 78.48, 40.20) --
	( 79.65, 40.26) --
	( 80.82, 40.34) --
	( 82.00, 40.42) --
	( 83.17, 40.51) --
	( 84.35, 40.62) --
	( 85.52, 40.73) --
	( 86.70, 40.85) --
	( 87.87, 40.99) --
	( 89.04, 41.13) --
	( 90.22, 41.28) --
	( 91.39, 41.45) --
	( 92.57, 41.62) --
	( 93.74, 41.80) --
	( 94.91, 42.00) --
	( 96.09, 42.20) --
	( 97.26, 42.41) --
	( 98.44, 42.63) --
	( 99.61, 42.86) --
	(100.79, 43.11) --
	(101.96, 43.36) --
	(103.13, 43.62) --
	(104.31, 43.89) --
	(105.48, 44.17) --
	(106.66, 44.46) --
	(107.83, 44.77) --
	(109.00, 45.08) --
	(110.18, 45.40) --
	(111.35, 45.73) --
	(112.53, 46.07) --
	(113.70, 46.42) --
	(114.88, 46.78) --
	(116.05, 47.15) --
	(117.22, 47.53) --
	(118.40, 47.92) --
	(119.57, 48.32) --
	(120.75, 48.73) --
	(121.92, 49.15) --
	(123.10, 49.58) --
	(124.27, 50.02) --
	(125.44, 50.47) --
	(126.62, 50.93) --
	(127.79, 51.40) --
	(128.97, 51.88) --
	(130.14, 52.37) --
	(131.31, 52.86) --
	(132.49, 53.33) --
	(133.66, 53.80) --
	(134.84, 54.26) --
	(136.01, 54.71) --
	(137.19, 55.15) --
	(138.36, 55.58) --
	(139.53, 56.00) --
	(140.71, 56.41) --
	(141.88, 56.81) --
	(143.06, 57.20) --
	(144.23, 57.58) --
	(145.40, 57.95) --
	(146.58, 58.31) --
	(147.75, 58.66) --
	(148.93, 59.00) --
	(150.10, 59.33) --
	(151.28, 59.66) --
	(152.45, 59.97) --
	(153.62, 60.27) --
	(154.80, 60.56) --
	(155.97, 60.84) --
	(157.15, 61.11) --
	(158.32, 61.37) --
	(159.49, 61.63) --
	(160.67, 61.87) --
	(161.84, 62.10) --
	(163.02, 62.32) --
	(164.19, 62.54) --
	(165.37, 62.74) --
	(166.54, 62.93) --
	(167.71, 63.11) --
	(168.89, 63.29) --
	(170.06, 63.45) --
	(171.24, 63.60) --
	(172.41, 63.75) --
	(173.59, 63.88) --
	(174.76, 64.00) --
	(175.93, 64.12) --
	(177.11, 64.22) --
	(178.28, 64.31) --
	(179.46, 64.40) --
	(180.63, 64.47) --
	(181.80, 64.54) --
	(182.98, 64.59) --
	(184.15, 64.63) --
	(185.33, 64.67) --
	(186.50, 64.69) --
	(187.68, 64.71) --
	(188.85, 64.71) --
	(190.02, 64.71) --
	(191.20, 64.69) --
	(192.37, 64.67) --
	(193.55, 64.63) --
	(194.72, 64.59) --
	(195.89, 64.54) --
	(197.07, 64.47) --
	(198.24, 64.40) --
	(199.42, 64.31) --
	(200.59, 64.22) --
	(201.77, 64.12) --
	(202.94, 64.00) --
	(204.11, 63.88) --
	(205.29, 63.75) --
	(206.46, 63.60) --
	(207.64, 63.45) --
	(208.81, 63.29) --
	(209.99, 63.11) --
	(211.16, 62.93) --
	(212.33, 62.74) --
	(213.51, 62.54) --
	(214.68, 62.32) --
	(215.86, 62.10) --
	(217.03, 61.87) --
	(218.20, 61.63) --
	(219.38, 61.37) --
	(220.55, 61.11) --
	(221.73, 60.84) --
	(222.90, 60.56) --
	(224.08, 60.27) --
	(225.25, 59.98) --
	(226.42, 59.70) --
	(227.60, 59.42) --
	(228.77, 59.16) --
	(229.95, 58.91) --
	(231.12, 58.67) --
	(232.29, 58.44) --
	(233.47, 58.21) --
	(234.64, 58.00) --
	(235.82, 57.80) --
	(236.99, 57.61) --
	(238.17, 57.42) --
	(239.34, 57.25) --
	(240.51, 57.09) --
	(241.69, 56.93) --
	(242.86, 56.79) --
	(244.04, 56.66) --
	(245.21, 56.53) --
	(246.39, 56.42) --
	(247.56, 56.32) --
	(248.73, 56.22) --
	(249.91, 56.14) --
	(251.08, 56.07) --
	(252.26, 56.00) --
	(253.43, 55.95) --
	(254.60, 55.90) --
	(255.78, 55.87) --
	(256.95, 55.84) --
	(258.13, 55.83) --
	(259.30, 55.82) --
	(260.48, 55.83) --
	(261.65, 55.84) --
	(262.82, 55.87) --
	(264.00, 55.90) --
	(265.17, 55.95) --
	(266.35, 56.00) --
	(267.52, 56.07) --
	(268.69, 56.14) --
	(269.87, 56.22) --
	(271.04, 56.32) --
	(272.22, 56.42) --
	(273.39, 56.53) --
	(274.57, 56.66) --
	(275.74, 56.79) --
	(276.91, 56.93) --
	(278.09, 57.09) --
	(279.26, 57.25) --
	(280.44, 57.42) --
	(281.61, 57.61) --
	(282.79, 57.80) --
	(283.96, 58.00) --
	(285.13, 58.21) --
	(286.31, 58.44) --
	(287.48, 58.67) --
	(288.66, 58.91) --
	(289.83, 59.16) --
	(291.00, 59.42) --
	(292.18, 59.70) --
	(293.35, 59.98) --
	(294.53, 60.27) --
	(295.70, 60.57) --
	(296.88, 60.88) --
	(298.05, 61.20) --
	(299.22, 61.53) --
	(300.40, 61.87) --
	(301.57, 62.22) --
	(302.75, 62.58) --
	(303.92, 62.95) --
	(305.09, 63.34) --
	(306.27, 63.73);
\definecolor{drawColor}{RGB}{0,186,56}

\path[draw=drawColor,line width= 0.6pt,dash pattern=on 2pt off 2pt ,line join=round] ( 71.43, 40.02) --
	( 72.60, 40.02) --
	( 73.78, 40.04) --
	( 74.95, 40.06) --
	( 76.13, 40.10) --
	( 77.30, 40.14) --
	( 78.48, 40.20) --
	( 79.65, 40.26) --
	( 80.82, 40.34) --
	( 82.00, 40.42) --
	( 83.17, 40.51) --
	( 84.35, 40.62) --
	( 85.52, 40.73) --
	( 86.70, 40.85) --
	( 87.87, 40.99) --
	( 89.04, 41.13) --
	( 90.22, 41.28) --
	( 91.39, 41.45) --
	( 92.57, 41.62) --
	( 93.74, 41.80) --
	( 94.91, 42.00) --
	( 96.09, 42.20) --
	( 97.26, 42.41) --
	( 98.44, 42.63) --
	( 99.61, 42.86) --
	(100.79, 43.11) --
	(101.96, 43.36) --
	(103.13, 43.62) --
	(104.31, 43.89) --
	(105.48, 44.17) --
	(106.66, 44.46) --
	(107.83, 44.77) --
	(109.00, 45.08) --
	(110.18, 45.40) --
	(111.35, 45.73) --
	(112.53, 46.07) --
	(113.70, 46.42) --
	(114.88, 46.78) --
	(116.05, 47.15) --
	(117.22, 47.53) --
	(118.40, 47.92) --
	(119.57, 48.32) --
	(120.75, 48.73) --
	(121.92, 49.15) --
	(123.10, 49.58) --
	(124.27, 50.02) --
	(125.44, 50.47) --
	(126.62, 50.93) --
	(127.79, 51.40) --
	(128.97, 51.88) --
	(130.14, 52.37) --
	(131.31, 52.87) --
	(132.49, 53.37) --
	(133.66, 53.89) --
	(134.84, 54.42) --
	(136.01, 54.96) --
	(137.19, 55.51) --
	(138.36, 56.07) --
	(139.53, 56.63) --
	(140.71, 57.21) --
	(141.88, 57.80) --
	(143.06, 58.40) --
	(144.23, 59.00) --
	(145.40, 59.62) --
	(146.58, 60.25) --
	(147.75, 60.89) --
	(148.93, 61.53) --
	(150.10, 62.19) --
	(151.28, 62.86) --
	(152.45, 63.53) --
	(153.62, 64.22) --
	(154.80, 64.92) --
	(155.97, 65.62) --
	(157.15, 66.34) --
	(158.32, 67.06) --
	(159.49, 67.80) --
	(160.67, 68.55) --
	(161.84, 69.30) --
	(163.02, 70.07) --
	(164.19, 70.84) --
	(165.37, 71.63) --
	(166.54, 72.41) --
	(167.71, 73.19) --
	(168.89, 73.95) --
	(170.06, 74.71) --
	(171.24, 75.45) --
	(172.41, 76.19) --
	(173.59, 76.92) --
	(174.76, 77.63) --
	(175.93, 78.34) --
	(177.11, 79.04) --
	(178.28, 79.72) --
	(179.46, 80.40) --
	(180.63, 81.07) --
	(181.80, 81.72) --
	(182.98, 82.37) --
	(184.15, 83.01) --
	(185.33, 83.63) --
	(186.50, 84.25) --
	(187.68, 84.86) --
	(188.85, 85.46) --
	(190.02, 86.04) --
	(191.20, 86.62) --
	(192.37, 87.19) --
	(193.55, 87.75) --
	(194.72, 88.30) --
	(195.89, 88.83) --
	(197.07, 89.36) --
	(198.24, 89.88) --
	(199.42, 90.39) --
	(200.59, 90.89) --
	(201.77, 91.38) --
	(202.94, 91.86) --
	(204.11, 92.33) --
	(205.29, 92.78) --
	(206.46, 93.23) --
	(207.64, 93.67) --
	(208.81, 94.10) --
	(209.99, 94.52) --
	(211.16, 94.93) --
	(212.33, 95.33) --
	(213.51, 95.72) --
	(214.68, 96.10) --
	(215.86, 96.47) --
	(217.03, 96.83) --
	(218.20, 97.19) --
	(219.38, 97.53) --
	(220.55, 97.86) --
	(221.73, 98.18) --
	(222.90, 98.49) --
	(224.08, 98.79) --
	(225.25, 99.08) --
	(226.42, 99.36) --
	(227.60, 99.63) --
	(228.77, 99.90) --
	(229.95,100.15) --
	(231.12,100.39) --
	(232.29,100.62) --
	(233.47,100.84) --
	(234.64,101.06) --
	(235.82,101.26) --
	(236.99,101.45) --
	(238.17,101.63) --
	(239.34,101.81) --
	(240.51,101.97) --
	(241.69,102.12) --
	(242.86,102.27) --
	(244.04,102.40) --
	(245.21,102.52) --
	(246.39,102.64) --
	(247.56,102.74) --
	(248.73,102.83) --
	(249.91,102.92) --
	(251.08,102.99) --
	(252.26,103.06) --
	(253.43,103.11) --
	(254.60,103.16) --
	(255.78,103.19) --
	(256.95,103.22) --
	(258.13,103.23) --
	(259.30,103.23) --
	(260.48,103.23) --
	(261.65,103.22) --
	(262.82,103.19) --
	(264.00,103.16) --
	(265.17,103.11) --
	(266.35,103.06) --
	(267.52,102.99) --
	(268.69,102.92) --
	(269.87,102.83) --
	(271.04,102.74) --
	(272.22,102.64) --
	(273.39,102.52) --
	(274.57,102.40) --
	(275.74,102.27) --
	(276.91,102.12) --
	(278.09,101.97) --
	(279.26,101.81) --
	(280.44,101.63) --
	(281.61,101.45) --
	(282.79,101.26) --
	(283.96,101.07) --
	(285.13,100.88) --
	(286.31,100.71) --
	(287.48,100.55) --
	(288.66,100.40) --
	(289.83,100.25) --
	(291.00,100.12) --
	(292.18,100.00) --
	(293.35, 99.88) --
	(294.53, 99.78) --
	(295.70, 99.68) --
	(296.88, 99.60) --
	(298.05, 99.53) --
	(299.22, 99.46) --
	(300.40, 99.41) --
	(301.57, 99.36) --
	(302.75, 99.33) --
	(303.92, 99.30) --
	(305.09, 99.29) --
	(306.27, 99.28);
\definecolor{drawColor}{RGB}{97,156,255}

\path[draw=drawColor,line width= 0.6pt,dash pattern=on 4pt off 2pt ,line join=round] ( 71.43, 40.02) --
	( 72.60, 40.02) --
	( 73.78, 40.04) --
	( 74.95, 40.06) --
	( 76.13, 40.10) --
	( 77.30, 40.14) --
	( 78.48, 40.20) --
	( 79.65, 40.26) --
	( 80.82, 40.34) --
	( 82.00, 40.42) --
	( 83.17, 40.51) --
	( 84.35, 40.62) --
	( 85.52, 40.73) --
	( 86.70, 40.85) --
	( 87.87, 40.99) --
	( 89.04, 41.13) --
	( 90.22, 41.28) --
	( 91.39, 41.45) --
	( 92.57, 41.62) --
	( 93.74, 41.80) --
	( 94.91, 42.00) --
	( 96.09, 42.20) --
	( 97.26, 42.41) --
	( 98.44, 42.63) --
	( 99.61, 42.86) --
	(100.79, 43.11) --
	(101.96, 43.36) --
	(103.13, 43.62) --
	(104.31, 43.89) --
	(105.48, 44.17) --
	(106.66, 44.46) --
	(107.83, 44.77) --
	(109.00, 45.08) --
	(110.18, 45.40) --
	(111.35, 45.73) --
	(112.53, 46.07) --
	(113.70, 46.42) --
	(114.88, 46.78) --
	(116.05, 47.15) --
	(117.22, 47.53) --
	(118.40, 47.92) --
	(119.57, 48.32) --
	(120.75, 48.73) --
	(121.92, 49.15) --
	(123.10, 49.58) --
	(124.27, 50.02) --
	(125.44, 50.47) --
	(126.62, 50.93) --
	(127.79, 51.40) --
	(128.97, 51.88) --
	(130.14, 52.37) --
	(131.31, 52.87) --
	(132.49, 53.37) --
	(133.66, 53.89) --
	(134.84, 54.42) --
	(136.01, 54.96) --
	(137.19, 55.51) --
	(138.36, 56.07) --
	(139.53, 56.63) --
	(140.71, 57.21) --
	(141.88, 57.80) --
	(143.06, 58.40) --
	(144.23, 59.00) --
	(145.40, 59.62) --
	(146.58, 60.25) --
	(147.75, 60.89) --
	(148.93, 61.53) --
	(150.10, 62.19) --
	(151.28, 62.86) --
	(152.45, 63.53) --
	(153.62, 64.22) --
	(154.80, 64.92) --
	(155.97, 65.62) --
	(157.15, 66.34) --
	(158.32, 67.06) --
	(159.49, 67.80) --
	(160.67, 68.55) --
	(161.84, 69.30) --
	(163.02, 70.07) --
	(164.19, 70.84) --
	(165.37, 71.63) --
	(166.54, 72.42) --
	(167.71, 73.23) --
	(168.89, 74.04) --
	(170.06, 74.87) --
	(171.24, 75.70) --
	(172.41, 76.55) --
	(173.59, 77.40) --
	(174.76, 78.26) --
	(175.93, 79.14) --
	(177.11, 80.02) --
	(178.28, 80.91) --
	(179.46, 81.78) --
	(180.63, 82.65) --
	(181.80, 83.50) --
	(182.98, 84.34) --
	(184.15, 85.18) --
	(185.33, 86.00) --
	(186.50, 86.82) --
	(187.68, 87.62) --
	(188.85, 88.42) --
	(190.02, 89.20) --
	(191.20, 89.98) --
	(192.37, 90.74) --
	(193.55, 91.50) --
	(194.72, 92.25) --
	(195.89, 92.98) --
	(197.07, 93.71) --
	(198.24, 94.42) --
	(199.42, 95.13) --
	(200.59, 95.83) --
	(201.77, 96.51) --
	(202.94, 97.19) --
	(204.11, 97.86) --
	(205.29, 98.51) --
	(206.46, 99.16) --
	(207.64, 99.80) --
	(208.81,100.42) --
	(209.99,101.04) --
	(211.16,101.65) --
	(212.33,102.25) --
	(213.51,102.83) --
	(214.68,103.41) --
	(215.86,103.98) --
	(217.03,104.54) --
	(218.20,105.09) --
	(219.38,105.63) --
	(220.55,106.15) --
	(221.73,106.67) --
	(222.90,107.18) --
	(224.08,107.68) --
	(225.25,108.17) --
	(226.42,108.65) --
	(227.60,109.12) --
	(228.77,109.58) --
	(229.95,110.03) --
	(231.12,110.47) --
	(232.29,110.89) --
	(233.47,111.31) --
	(234.64,111.72) --
	(235.82,112.12) --
	(236.99,112.51) --
	(238.17,112.90) --
	(239.34,113.27) --
	(240.51,113.63) --
	(241.69,113.98) --
	(242.86,114.32) --
	(244.04,114.65) --
	(245.21,114.97) --
	(246.39,115.28) --
	(247.56,115.58) --
	(248.73,115.88) --
	(249.91,116.19) --
	(251.08,116.52) --
	(252.26,116.85) --
	(253.43,117.19) --
	(254.60,117.54) --
	(255.78,117.90) --
	(256.95,118.27) --
	(258.13,118.65) --
	(259.30,119.04) --
	(260.48,119.44) --
	(261.65,119.85) --
	(262.82,120.27) --
	(264.00,120.70) --
	(265.17,121.14) --
	(266.35,121.59) --
	(267.52,122.05) --
	(268.69,122.52) --
	(269.87,122.99) --
	(271.04,123.48) --
	(272.22,123.98) --
	(273.39,124.49) --
	(274.57,125.01) --
	(275.74,125.54) --
	(276.91,126.08) --
	(278.09,126.62) --
	(279.26,127.18) --
	(280.44,127.75) --
	(281.61,128.33) --
	(282.79,128.92) --
	(283.96,129.51) --
	(285.13,130.12) --
	(286.31,130.74) --
	(287.48,131.37) --
	(288.66,132.00) --
	(289.83,132.65) --
	(291.00,133.31) --
	(292.18,133.97) --
	(293.35,134.65) --
	(294.53,135.34) --
	(295.70,136.03) --
	(296.88,136.74) --
	(298.05,137.46) --
	(299.22,138.18) --
	(300.40,138.92) --
	(301.57,139.66) --
	(302.75,140.42) --
	(303.92,141.18) --
	(305.09,141.96) --
	(306.27,142.74);
\end{scope}
\begin{scope}
\path[clip] ( 56.17, 30.14) rectangle (391.98,247.45);
\definecolor{drawColor}{RGB}{248,118,109}

\node[text=drawColor,anchor=base west,inner sep=0pt, outer sep=0pt, scale=  0.90] at (306.27, 60.63) {$b_{1}=0.25$, $b_{2}=0.65$};
\definecolor{drawColor}{RGB}{0,186,56}

\node[text=drawColor,anchor=base west,inner sep=0pt, outer sep=0pt, scale=  0.90] at (306.27, 96.18) {$b_{1}=0.4$, $b_{2}=0.9$};
\definecolor{drawColor}{RGB}{97,156,255}

\node[text=drawColor,anchor=base west,inner sep=0pt, outer sep=0pt, scale=  0.90] at (306.27,139.65) {$b_{1}=0.45$, $b_{2}=0.75$};
\end{scope}
\begin{scope}
\path[clip] (  0.00,  0.00) rectangle (397.48,252.94);
\definecolor{drawColor}{RGB}{0,0,0}

\path[draw=drawColor,line width= 0.2pt,line join=round] ( 56.17, 30.14) --
	( 56.17,247.45);
\end{scope}
\begin{scope}
\path[clip] (  0.00,  0.00) rectangle (397.48,252.94);
\definecolor{drawColor}{gray}{0.30}

\node[text=drawColor,anchor=base east,inner sep=0pt, outer sep=0pt, scale=  0.88] at ( 51.22, 36.99) {0.00};

\node[text=drawColor,anchor=base east,inner sep=0pt, outer sep=0pt, scale=  0.88] at ( 51.22, 86.38) {0.25};

\node[text=drawColor,anchor=base east,inner sep=0pt, outer sep=0pt, scale=  0.88] at ( 51.22,135.76) {0.50};

\node[text=drawColor,anchor=base east,inner sep=0pt, outer sep=0pt, scale=  0.88] at ( 51.22,185.15) {0.75};

\node[text=drawColor,anchor=base east,inner sep=0pt, outer sep=0pt, scale=  0.88] at ( 51.22,234.54) {1.00};
\end{scope}
\begin{scope}
\path[clip] (  0.00,  0.00) rectangle (397.48,252.94);
\definecolor{drawColor}{gray}{0.20}

\path[draw=drawColor,line width= 0.6pt,line join=round] ( 53.42, 40.02) --
	( 56.17, 40.02);

\path[draw=drawColor,line width= 0.6pt,line join=round] ( 53.42, 89.41) --
	( 56.17, 89.41);

\path[draw=drawColor,line width= 0.6pt,line join=round] ( 53.42,138.79) --
	( 56.17,138.79);

\path[draw=drawColor,line width= 0.6pt,line join=round] ( 53.42,188.18) --
	( 56.17,188.18);

\path[draw=drawColor,line width= 0.6pt,line join=round] ( 53.42,237.57) --
	( 56.17,237.57);
\end{scope}
\begin{scope}
\path[clip] (  0.00,  0.00) rectangle (397.48,252.94);
\definecolor{drawColor}{RGB}{0,0,0}

\path[draw=drawColor,line width= 0.2pt,line join=round] ( 56.17, 30.14) --
	(391.98, 30.14);
\end{scope}
\begin{scope}
\path[clip] (  0.00,  0.00) rectangle (397.48,252.94);
\definecolor{drawColor}{gray}{0.20}

\path[draw=drawColor,line width= 0.6pt,line join=round] ( 71.43, 27.39) --
	( 71.43, 30.14);

\path[draw=drawColor,line width= 0.6pt,line join=round] (188.85, 27.39) --
	(188.85, 30.14);

\path[draw=drawColor,line width= 0.6pt,line join=round] (306.27, 27.39) --
	(306.27, 30.14);
\end{scope}
\begin{scope}
\path[clip] (  0.00,  0.00) rectangle (397.48,252.94);
\definecolor{drawColor}{gray}{0.30}

\node[text=drawColor,anchor=base,inner sep=0pt, outer sep=0pt, scale=  0.88] at ( 71.43, 19.13) {0.0};

\node[text=drawColor,anchor=base,inner sep=0pt, outer sep=0pt, scale=  0.88] at (188.85, 19.13) {0.5};

\node[text=drawColor,anchor=base,inner sep=0pt, outer sep=0pt, scale=  0.88] at (306.27, 19.13) {1.0};
\end{scope}
\begin{scope}
\path[clip] (  0.00,  0.00) rectangle (397.48,252.94);
\definecolor{drawColor}{RGB}{0,0,0}

\node[text=drawColor,anchor=base,inner sep=0pt, outer sep=0pt, scale=  1.10] at (224.08,  6.06) {$x$};
\end{scope}
\begin{scope}
\path[clip] (  0.00,  0.00) rectangle (397.48,252.94);
\definecolor{drawColor}{RGB}{0,0,0}

\node[text=drawColor,rotate= 90.00,anchor=base,inner sep=0pt, outer sep=0pt, scale=  1.10] at ( 13.08,138.79) {$f_{+}(x)$};
\end{scope}
\end{tikzpicture}